\documentclass{svmono-copie}%
\usepackage{amsmath}
\usepackage{graphicx}%
\usepackage{amsfonts}%
\usepackage{amssymb}

\usepackage{makeidx}         
\usepackage{multicol}        
\usepackage[bottom]{footmisc}

\begin{document}

\title{MARKOV PATHS, LOOPS AND FIELDS\\}
\author{Yves Le Jan\\D\'{e}partement de Math\'{e}matiques\\Universit\'{e} Paris Sud 11\\yves.lejan@math.u-psud.fr}
\maketitle

\bigskip

\bigskip
\tableofcontents
\baselineskip                                       =16pt

\chapter*{Introduction}

\bigskip

The purpose of these notes is to explore some simple relations between
Markovian path and loop measures, the Poissonian ensembles of loops they
determine, their occupation fields, uniform spanning trees, determinants, and
Gaussian Markov fields such as the free field. These relations are first
studied in complete generality  in the finite discrete setting, then partly
generalized to specific examples in infinite and continuous spaces.

These notes contain the results published in \cite{Lejaop} where the main
emphasis was put on the study of occupation fields defined by Poissonian
ensembles of Markov loops. These were defined in \cite{LW} for planar Brownian
motion in relation with SLE processes and in \cite{LT} for simple random
walks. They appeared informally already in \cite{Symanz}. For half integral
values $\frac{k}{2}$ of the intensity parameter $\alpha$, these occupation
fields can be identified with the sum of squares of $k$ copies of the
associated free field (i.e. the Gaussian field whose covariance is given by
the Green function)$.$ This is related to Dynkin's isomorphism (cf \cite{Dy},
\cite{MR}, \cite{LJ1}).

As in \cite{Lejaop}, we first present the theory in the elementary framework
of symmetric Markov chains on a finite space. After some generalities on
graphs and symmetric Markov chains, we study the $\sigma$-finite loop measure
associated to a field of conductances. Then we study geodesic loops with an
exposition of results of independent interest, such as the calculation of
Ihara's zeta function. After that, we turn our attention to the Poisson
process of loops and its occupation field, proving also several other
interesting results such as the relation between loop ensembles and spanning
trees given by Wilson algorithm and the reflection positivity property. Spanning trees are related to the fermionic Fock space as Markovian loop
ensembles are related to the bosonic Fock space, represented by the free
field. We
also study the decompositions of the loop ensemble induced by the excursions
into the complement of any given set.

Then we show that some results can be extended to more general Markov
processes defined on continuous spaces. There are no essential difficulties
for the occupation field when points are not polar but other cases are more
problematic. As for the square of the free field, cases for which the Green
function is Hilbert Schmidt such as those corresponding to two and three
dimensional Brownian motion can be dealt with through appropriate renormalization.

We show that the renormalized powers of the occupation field (i.e. the self
intersection local times of the loop ensemble) converge in the case of the two
dimensional Brownian motion and that they can be identified with higher even
Wick powers of the free field when $\alpha$ is a half integer.

At first, we suggest the reader could omit a few sections which are not
essential for the understanding of the main results. These are essentially
some of the generalities on graphs, results about wreath products, infinite
discrete graphs, boundaries, zeta functions, geodesics and geodesic loops. The
section on reflexion positivity, and, to a lesser extent, the one on
decompositions are not central. The last section on continuous spaces is not
written in full detail and may seem difficult to the least experienced readers.

These notes include those of the lecture I gave in St Flour in July 2008 with
some additional material. I choose this opportunity to express my thanks to
Jean Picard, to the audience and to the readers of the preliminary versions
whose suggestions were very useful, in particular to\ Juergen Angst, Cedric
Bordenave, Cedric Boutiller, Antoine Dahlqvist, Thomas Duquesne, Michel Emery,
Jacques Franchi, Liza Jones, Adrien Kassel, Rick Kenyon, Sophie Lemaire,
Thierry Levy, Gregorio Moreno, Jay Rosen (who pointed out a mistake in the expression of renormalization polynomials), Bruno Shapira, Alain Sznitman, Vincent Vigon, Lorenzo Zambotti and Jean Claude Zambrini.

\chapter{Symmetric Markov processes on finite spaces}

Notations: functions and measures on finite (or countable) spaces are often
denoted as vectors and covectors, i.e. with upper and lower indices, respectively.

The multiplication operator defined by a function $f$\ acting on functions or
on measures is in general simply denoted by $f$, but sometimes, to avoid
confusion, it will be denoted by $M_{f}$. The function obtained as the density
of a measure $\mu$ with respect to some other measure $\nu$ is simply denoted
$\frac{\mu}{\nu}$.

\section{Graphs}
\index{conductance}
Our basic object will be a finite space $X$\ and a set of \ non negative
\emph{conductances} $C_{x,y}=C_{y,x}$, indexed by pairs of distinct points of
$X$. This situation allows to define a kind of \textsl{discrete topology and
geometry.} In this first section, we will briefly study the topological aspects.

We say that $\{x,y\}$, for $x\neq y$ belonging to $X$, is a link or an edge
iff $C_{x,y}>0$. An oriented edge $(x,y)$\ is defined by the choice of an
ordering in an edge$.$ We set $-(x,y)=(y,x)$ and if $e=(x,y)$, we denote it
also $(e^{-},e^{+})$. The degree $d_{x}$ of a vertex $x$ is by definition the
number of edges incident at $x$.

The points of $X$ together with the set of non oriented edges $E$\ define a
graph $(X,E)$. \textsl{We assume it is connected}. The set of oriented edges
is denoted $E^{o}$. It will always be viewed as a subset of $X^{2}$, without
reference to any imbedding.

The associated line graph is the oriented graph defined by $E^{o}$ as set of
vertices and in which oriented edges are pairs $(e_{1},e_{2})$ such that
$e_{1}^{+}=e_{2}^{-}$. The mapping $e\rightarrow-e$ is an involution of the
line graph.

An important example is the case in which conductances are equal to zero or
one. Then the conductance matrix is the adjacency matrix of the graph:
$C_{x,y}=1_{\{x,y\}\in E}$

A complete graph is defined by all conductances equal to one.

The complete graph with $n$ vertices is denoted $K_{n}$. The complete graph
$K_{4}$ is the graph defined by the tetrahedron. $K_{5}$ \ is not planar (i.e.
cannot be imbedded in a plane), but $K_{4}$ is.%

\begin{center}
\includegraphics[
height=2.2018in,
width=4.039in
]%
{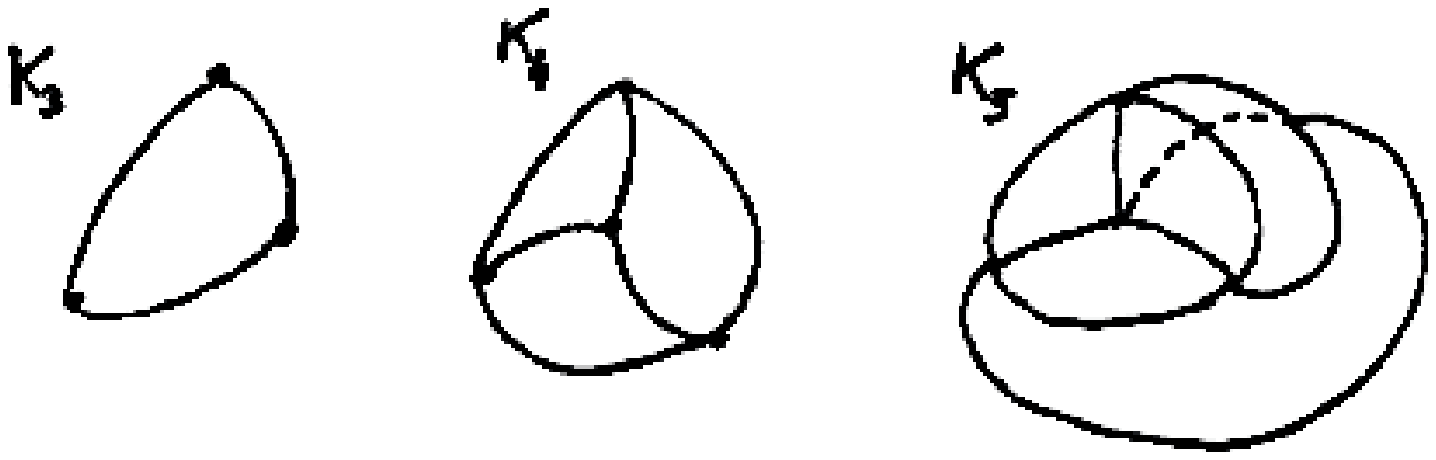}%
\end{center}
\index{geodesic}
A finite discrete path on $X$, say $(x_{0},x_{1},...,x_{n})$ is called a
(discrete) \emph{geodesic arc} iff $\{x_{i},x_{i+1}\}\in E$ (path segment on
the graph) and $x_{i-1}\neq x_{i+1}$ (without backtraking). Geodesic arcs
starting at $x_{0}$ form a \emph{marked tree} $\mathfrak{T}_{x_{0}}$ rooted in
$x_{0}$ (the marks belong to $X$: they are the endpoints of the geodesic
arcs). Oriented edges of $\mathfrak{T}_{x_{0}}$ are defined by pairs of
geodesic arcs of the form:
\index{universal cover}
$((x_{0},x_{1},...,x_{n}),(x_{0},x_{1},...,x_{n},x_{n+1}))$ (the orientation
is defined in reference to the root). $\mathfrak{T}_{x_{0}}$ is a
\emph{universal cover} of $X$ \cite{Mass}.

A (discrete) loop based at $x_{0}\in X$ is by definition a path $\xi=(\xi
_{1},...,\xi_{p(\xi)})$, with $\xi_{1}=x_{0}$, and $\{\xi_{i},\xi_{i+1}\}\in
E$, for all $1\leq i\leq p$ with the convention $\xi_{p+1}=\xi_{1}$. On the
space $\mathfrak{L}_{x_{0}}$\ of discrete loops based at some point $x_{0}$,
we can define an operation of concatenation, which provides a monoid
structure, i.e. is associative with a neutral element (the empty loop). The
concatenation of two closed geodesics (i.e. geodesic loops) based at $x_{0}$
is not directly a closed geodesic. It can involve backtracking ''in the
middle'' but then after cancellation of the two inverse subarcs, we get a
closed geodesic, possibly empty if the two closed geodesics are identical up
to reverse order. With this operation, \emph{closed geodesics based at }%
$x_{0}$\emph{ define a group} $\Gamma_{x_{0}}$. The structure of
$\Gamma_{x_{0}}$ does not depend on the base point and defines the
\index{fundamental group}\emph{fundamental group} $\Gamma$ of the graph (as the graph is connected: see
for example \cite{Mass}). Indeed, any geodesic arc $\gamma_{1}$\ from $x_{0}$
to another point $y_{0}$ of $X$ defines an isomorphism between $\Gamma_{x_{0}%
}$ and $\Gamma_{y_{0}}$. It associates to a closed geodesic $\gamma$ based in
$x_{0}$ the closed geodesic $[\gamma_{1}]^{-1}\gamma\gamma_{1}$ (here
$[\gamma_{1}]^{-1}$ denotes the backward arc). In the case where $x_{0}=y_{0}%
$, it is an interior isomorphism\ (conjugation by $\gamma_{1}$).

There is a natural left action of $\Gamma_{x_{0}}$\ on $\mathfrak{T}_{x_{0}}$.
It can be interpreted as a change of root in the tree \ (with the same mark).
Besides, any geodesic arc between $x_{0}$ and another point $y_{0}$ of $X$
defines an isomorphism between $\mathfrak{T}_{x_{0}}$ and $\mathfrak{T}%
_{y_{0}}$ (change of root, with different marks) .

We have just seen that the universal covering of the finite graph $(X,E)$ at
$x_{0}$\ is a tree $\mathfrak{T}_{x_{0}}$ projecting on $X\mathfrak{.}$ The
fiber at $x_{0}$ is $\Gamma_{x_{0}}$. The groups $\Gamma_{x_{0}},x_{0}\in
X$\ are conjugated in a non canonical way. Note that $X=\Gamma_{x_{0}%
}\backslash\mathfrak{T}_{x_{0}}$ (here the use of the quotient on the left
corresponds to the left action).

\begin{example}
Among graphs, the simplest ones are $r-$regular graphs, in which each point
has $r$ neighbours. A universal covering of any $r-$regular graph is
isomorphic to the $r-$regular tree $\mathfrak{T}^{(r)}$.
\end{example}

\begin{example}
Cayley graphs: a finite group with a set of generators $S=\{g_{1},..g_{k}\}$
such that $S\cap S^{-1}$ is empty defines an oriented $2k$-regular graph.
\end{example}

\bigskip

A \emph{spanning tree} $T$\ is by definition a subgraph of $(X,E)$ which is a
tree and covers all points in $X$.
\index{spanning tree} 
It has necessarily $\left|  X\right|
-1$\ edges, see for example two spanning trees of $K_{4}$.%

\begin{center}
\includegraphics[
height=2.6282in,
width=4.0603in
]%
{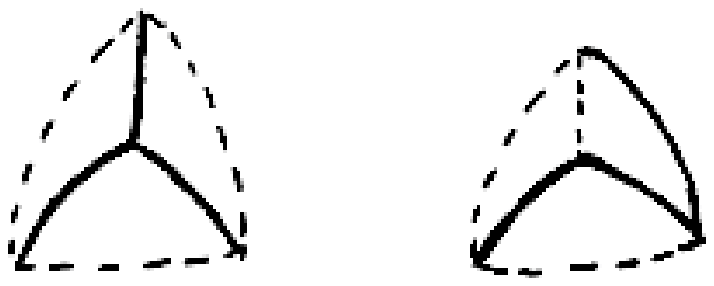}%
\\
Two spanning trees of $K_{4}$%
\end{center}
\bigskip

\ The inverse images of a spanning tree by the canonical projection from a
universal cover $\mathfrak{T}_{x_{0}}$\ onto $X$ form a tesselation on
$\mathfrak{T}_{x_{0}}$, i.e. a partition of $\mathfrak{T}_{x_{0}}$\ in
identical subtrees, which are fundamental domains for the action of
$\Gamma_{x_{0}}$. Conversely, a section of the canonical projection from the
universal cover defines a spanning tree.

Fixing a spanning tree determines a unique geodesic between two points of $X$.
Therefore, it determines\ the conjugation isomorphisms between the various
groups $\Gamma_{x_{0}}$ and the isomorphisms between the universal covers
$\mathfrak{T}_{x_{0}}$.

\begin{remark}
\label{intrinsic}Equivalently, we could have started with an infinite tree
$\mathfrak{T}$ and a group $\Gamma$\ of isomorphisms of this tree such that
the quotient graph $\Gamma\backslash\mathfrak{T}$ is finite.
\end{remark}

\emph{The fundamental group }$\Gamma$\emph{ is a free group} with $\left|
E\right|  -\left|  X\right|  +1=r$ generators. To construct a set of
generators, one considers a spanning tree $T$\ of the graph, and choose an
orientation on each of the $r$ remaining links. This defines $r$ oriented
cycles on the graph and a system of $r$ generators for the fundamental group.
(See \cite{Mass} or Serres (\cite{Ser}) in a more general context).

\begin{example}
Consider $K_{3}$ and $K_{4}$.
\end{example}

Here is a picture of the universal covering of $K_{4}$, and of the action of
the fundamental group with the tesselation defined by a spanning tree.%

\begin{center}
\includegraphics[
scale=0.4
]%
{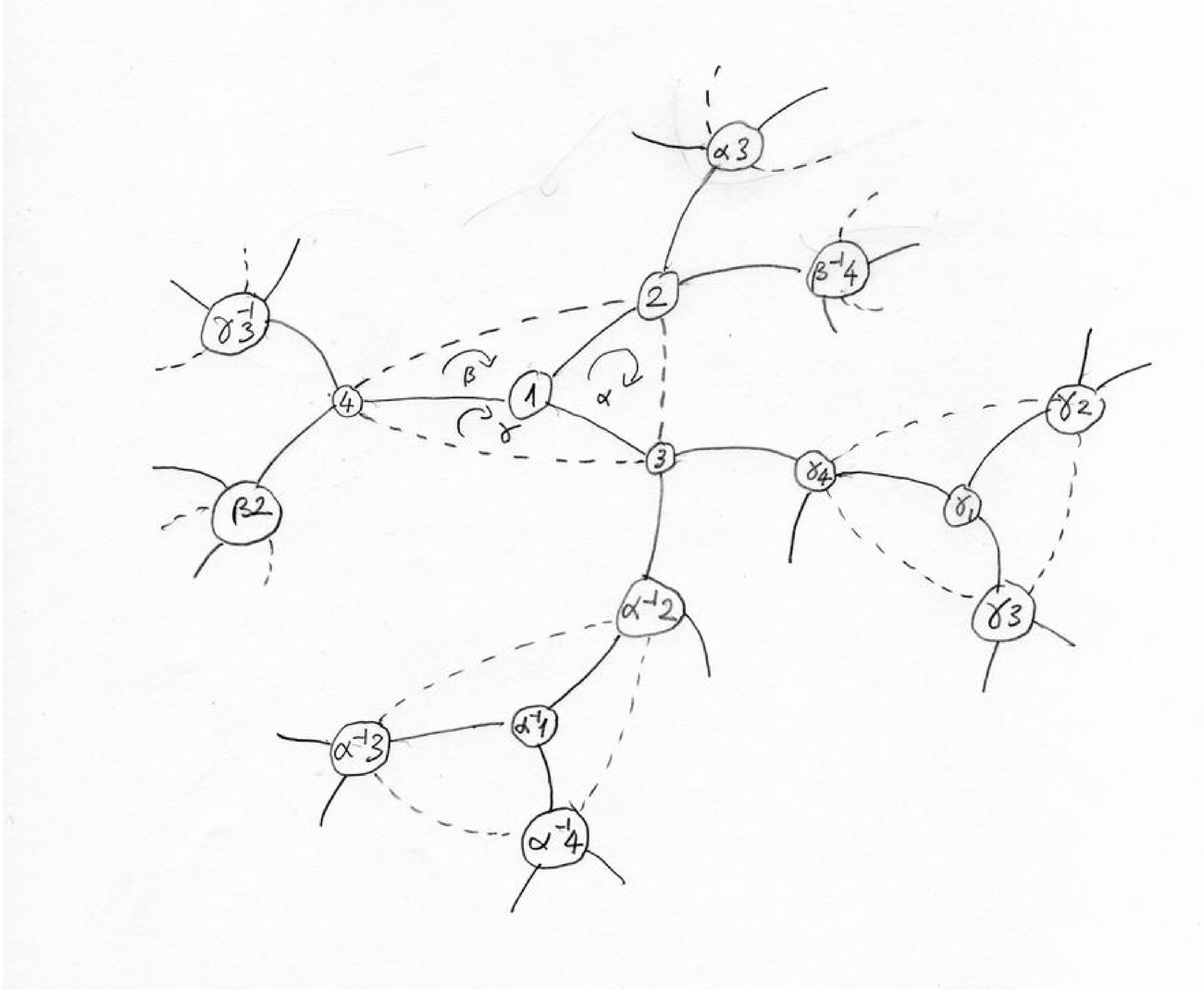}%
\\
Universal cover and tesselation of $K_{4}$%
\end{center}

\bigskip

\index{non-ramified covering}There are various non-ramified coverings, intermediate between $(X,E)$ and the
universal covering. Non ramified means that locally, the covering space is
identical to the graph (same incident edges). Then each oriented path segment
on $X$ can be lifted to the covering in a unique way, given a lift of its
starting point.

Each non ramified covering is (up to an isomorphism) associated with a
subgroup $H$\ of $\Gamma$, defined up to conjugation. More precisely, given a
non ramified covering $\widetilde{X}$, a point $x_{0}$ of $X$ and a point
$\widetilde{x}_{0}$ in the fiber above $x_{0}$, the closed geodesics based at
$x_{0}$ whose lift to the covering\ starting at $\widetilde{x}_{0}$\ are
closed form a subgroup $H_{\widetilde{x}_{0}}$\ of $\Gamma_{x_{0}}$,
canonicaly isomorphic to the fundamental group of $\widetilde{X}$ represented
by closed geodesics based at $\widetilde{x_{0}}$. If we consider a different
point $\widetilde{y_{0}}$, any\ geodesic path segment $\widetilde{\gamma_{1}}%
$\ between $\widetilde{x}_{0}$ and $\widetilde{y}_{0}$ defines an isomorphism
between $\Gamma_{x_{0}}$ and $\Gamma_{y_{0}}$ which exchanges $H_{\widetilde
{x}_{0}}$ and $H_{\widetilde{y}_{0}}$. Denoting $\gamma_{1}$\ the projection
of $\widetilde{\gamma_{1}}$ on $X$ , it associates to a closed geodesic
$\gamma$ based in $x_{0}$ whose lift to the covering is closed the closed
geodesic $[\gamma_{1}]^{-1}\gamma\gamma_{1}$ whose lift to the covering is
also closed.

\begin{example}
By central symmetry, the cube is a two fold covering of the tetrahedron
associated with the group $\mathbb{Z}/2\mathbb{Z}$.
\end{example}

Conversely, if $H$ is a subgroup of $\Gamma_{x_{0}}$, the covering is defined
as the quotient graph $(Y,F)$ with $Y=H\backslash\mathfrak{T}_{x_{0}}$ and $F$
the set of edges defined by the canonical projection from $\mathfrak{T}%
_{x_{0}}$ onto $Y$. $H$ can be interpreted as the group of closed geodesics on
the quotient graph, based at $H_{x_{0}}$, i.e. as the fundamental group of $Y$.

If $H$ is a normal subgroup, the quotient group (also called the covering
group) $H\backslash\Gamma_{x_{0}}$ acts faithfully on\ the fiber at $x_{0}$.
An example is the commutator subgroup $[\Gamma_{x_{0}},\Gamma_{x_{0}}]$. The
associate covering is the maximal Abelian covering at $x_{0}$.

\begin{exercise}
Determine the maximal Abelian cover of the tetrahedron.
\end{exercise}

\section{Energy}

Let us consider a nonnegative function $\kappa$ on $X$. Set $\lambda
_{x}=\kappa_{x}+\sum_{y}C_{x,y}$\ and $P_{y}^{x}=\frac{C_{x,y}}{\lambda_{x}}$.
$P$ is a (sub) stochastic transition matrix which is $\lambda$-symmetric (i.e.
such that $\lambda_{x}P_{y}^{x}=\lambda_{y}P_{x}^{y}$) with $P_{x}^{x}=0$\ for
all $x$ in $X$.

It defines a symmetric irreducible Markov chain $\xi_{n}$.

We can define above it a continuous time $\lambda$-symmetric irreducible
Markov chain $x_{t}$, with exponential holding times of parameter $1$. We have
$x_{t}=\xi_{N_{t}}$, where $N_{t}$ denotes a Poisson process of intensity $1$.
\index{infinitesimal generator}The \emph{infinitesimal generator} is given by $L_{y}^{x}=P_{y}^{x}-\delta
_{y}^{x}$.

We denote by $P_{t}$ its (sub) Markovian semigroup $\exp(Lt)=\sum\frac{t^{k}%
}{k!}L^{k}$. $L$ and $P_{t}$ are $\lambda$-symmetric.

We will use the Markov chain associated with $C,\kappa$, sometimes in discrete
time, sometimes in continuous time (with exponential holding times).

\index{energy}Recall that for any complex function $z^{x},x\in X$, the ``energy''%
\[
e(z)=\left\langle -Lz,\overline{z}\right\rangle _{\lambda}=\sum_{x\in
X}-(Lz)^{x}\overline{z}^{x}\lambda_{x}%
\]
is nonnegative as it can be easily written
\[
e(z)=\sum_{x}\lambda_{x}z^{x}\overline{z}^{x}-\sum_{x,y}C_{x,y}z^{x}%
\overline{z}^{y}=\frac{1}{2}\sum_{x,y}C_{x,y}(z^{x}-z^{y})(\overline{z}%
^{x}-\overline{z}^{y})+\sum_{x}\kappa_{x}z^{x}\overline{z}^{x}%
\]
The Dirichlet space (\cite{Fukutak}) is the space of\ real functions equipped
with the \emph{energy} scalar product
\[
e(f,g)=\frac{1}{2}\sum_{x,y}C_{x,y}(f^{x}-f^{y})(g^{x}-g^{y})+\sum_{x}%
\kappa_{x}f^{x}g^{x}=\sum_{x}\lambda_{x}f^{x}g^{x}-\sum_{x,y}C_{x,y}f^{x}g^{y}%
\]
defined by polarization of $e$.

Note that the non negative symmetric ''\emph{conductance matrix}'' $C$ and the
non negative \emph{equilibrium or ``killing'' measure} $\kappa$ are the free
parameters of the model.
\index{killing measure}
\begin{exercise}
Prove that the eigenfunction associated with the lowest eigenvalue of $-L$\ is
unique and has constant sign by an argument based on the fact that the map
$z\rightarrow\left|  z\right|  $ lowers the energy (which follows easily from
the expression given above).
\end{exercise}

In quantum mechanics, the infinitesimal generator $-L$ is called the
Hamiltonian and its eigenvalues are the energy levels.

One can learn more on graphs and eigenvalues in \cite{Big1}.\bigskip

We have a dichotomy between:

\begin{itemize}
\item[-] the recurrent case where $0$ is the lowest eigenvalue of $-L$, and
the corresponding eigenspace is formed by constants. Equivalently, $P1=1$ and
$\kappa$ vanishes.

\item[-] the transient case where the lowest eigenvalue is positive which
means there is a ''Poincar\'{e} inequality'': For some positive $\varepsilon$,
the energy $e(f,f)$ dominates $\varepsilon\left\langle f,f\right\rangle
_{\lambda}$ for all $f$. Equivalently, as we are on a finite space, $\kappa$
does not vanish. Note however these equivalences doe not hold in general on
infinite spaces, though the dichotomy is still valid.
\end{itemize}

\index{potential operator}In the transient case, we denote by $V$ the associated \emph{potential
operator} $(-L)^{-1}=\int_{0}^{\infty}P_{t}dt$. It can be expressed in terms
of the spectral resolution of $L$. We will denote $\sum_{y}V_{y}^{x}f^{y}$ by
$(Vf)^{x}$\ or $Vf(x)$.

Note that the function $Vf$\ ( called the potential of $f$) is characterized
by the identity
\[
e(Vf,g)=\left\langle f,g\right\rangle _{\lambda}%
\]
valid for all functions $f$ and $g$. The potential operator diverges on
positive functions in the recurrent case. These properties define the
dichotomy transient/recurrent on infinite spaces.

\index{Green function}We denote by $G$ the \emph{Green function} defined on $X^{2}$ as
$G^{x,y}=\frac{V_{y}^{x}}{\lambda_{y}}=\frac{1}{\lambda_{y}}[(I-P)^{-1}%
]_{y}^{x}$ i.e. $G=(M_{\lambda}-C)^{-1}$. It induces a linear bijection from
measures into functions. We will denote $\sum_{y}G^{x,y}\mu_{y}$ by
$(G\mu)^{x}$\ or $G\mu(x)$.

Note that the function $G\mu$\ ( called the potential of $\mu$) is
characterized by the identity
\[
e(f,G\mu)=\left\langle f,\mu\right\rangle
\]
valid for all functions $f$ and measures $\mu$. In particular $G\kappa=1$ as
$\ e(1,f)=\sum f^{x}\kappa_{x}=\left\langle f,1\right\rangle _{\kappa}$.

\begin{example}
The Green function in the case of the complete graph $K_{n}$\ with uniform
killing measure of intensity $c>0$ is given by the matrix
\[
\frac{1}{n+c}(I+\frac{1}{c}J)
\]

where $J$ denotes the $(n,n)$ matrix with all entries equal to $1$.
\end{example}

\begin{proof}
Note first that $M_{\lambda}-C=(n+c)I-J$. The inverse is easily checked.
\end{proof}

See (\cite{Fukutak}) for a development of this theory in a more general setting.

In the recurrent case, the potential operator $V$ can be defined on the space
$\lambda^{\perp}$ of functions $f$ such that $\left\langle f,1\right\rangle
_{\lambda}=0$ as the inverse of the restriction of $I-P$ to $\lambda^{\perp}$.
The Green operator $G$ maps the space of measures of total charge zero onto
$\lambda^{\perp}$: setting for any signed measure $\nu$ of total charge zero
$G\nu=V\frac{\nu}{\lambda}$, we have for any function $f$, $\left\langle
\nu,f\right\rangle =e(G\nu,f)$ (as $e(G\nu,1)=0$) and in particular$\ f^{x}%
-f^{y}=e(G(\delta_{x}-\delta_{y}),f)$.

\begin{exercise}
\label{greencomp}In the case of the complete graph $K_{n}$, show that the
Green operator is given by:%
\[
G\nu(x)=\frac{\nu_{x}}{n}%
\]
\bigskip
\end{exercise}

\begin{remark}
Markov chains with different holding times parameters are associated with the
same energy form. If $q$ is any positive function on $X$, the Markov chain
with $y_{t}$\ holding times parameter $q^{x}$, $x\in X$ is obtained from
$x_{t}$ by time change: $y_{t}=x_{\sigma_{t}}$, where $\sigma_{t}$ is the
right continuous increasing family of stopping times defined by $\int
_{0}^{\sigma_{t}}q^{-1}(x_{s})ds=t$. Its semigroup is $q^{-1}\lambda
$-symmetric with infinitesimal generator given by $qL$. The potential opertor
is different but the Green function does not change. In particular, if we set
$q^{x}=\lambda_{x}$ for all $x$, the duality measure is the counting measure
and the potential operator $V$ is given by the Green function $G$. The
associated rescaled Markov chain will be used in the next chapters.
\end{remark}

\section{Feynman-Kac formula}

A discrete analogue of the Feynman-Kac formula can be given as follows: Let
$s$ be any function on $X$ taking values in $(0,1]$. Then, for the discrete
Markov chain $\xi_{n}$\ associated with $P$, it is a straightforward
consequence of the Markov property that:
\[
\mathbb{E}_{x}(\prod_{j=0}^{n-1}s(\xi_{j})1_{\{\xi_{n}=y\}})=[(M_{s}%
P)^{n}]_{y}^{x}%
\]
Similarly, for the continuous time Markov chain $x_{t}$ (with exponential
holding times), we have the Feynman-Kac formula:

\begin{proposition}
If $k(x)$ is a nonnegative function defined on $X$,%
\[
\mathbb{E}_{x}(e^{-\int_{0}^{t}k(x_{s})ds}1_{\{x_{t}=y\}})=[\exp
(t(L-M_{k})]_{y}^{x}.
\]
\end{proposition}

\begin{proof}
It is enough to check, by differentiating the first member $V(t)$\ with
respect to $t$, that $V^{\prime}(t)=(L-M_{k})V(t)$.

Precisely, if we set $(V_{t})_{y}^{x}=\mathbb{E}_{x}(e^{-\int_{0}%
^{t}k(x_{s})ds}1_{\{x_{t}=y\}})$, by the Markov property,
\begin{align*}
(V_{t+\Delta t})_{y}^{x}  &  =\mathbb{E}_{x}(e^{-\int_{0}^{t}k(x_{s}%
)ds}\mathbb{E}_{X_{t}}(e^{-\int_{0}^{\Delta t}k(x_{s})ds}1_{\{x_{\Delta
t}=y\}})\\
&  =\sum_{z\in X}\mathbb{E}_{x}(e^{-\int_{0}^{t}k(x_{s})ds}1_{\{x_{t}%
=z\}}\mathbb{E}_{z}(e^{-\int_{0}^{\Delta t}k(x_{s})ds}1_{\{x_{\Delta t}%
=y\}})\\
&  =\sum_{z\in X}(V_{t})_{z}^{x}\mathbb{E}_{z}(e^{-\int_{0}^{\Delta t}%
k(x_{s})ds}1_{\{x_{\Delta t}=y\}}).
\end{align*}
Then one verifies easily by considering the first and second times of jump
that as $\Delta t$ goes to zero, $\mathbb{E}_{y}(e^{-\int_{0}^{\Delta
t}k(x_{s})ds}1_{\{x_{\Delta t}=y\}})-1$ is equivalent to $-(k(y)+1)\Delta t$
and for $z\neq y$, $\mathbb{E}_{z}(e^{-\int_{0}^{\Delta t}k(x_{s}%
)ds}1_{\{x_{\Delta t}=y\}})$ is equivalent to $P_{y}^{z}\Delta t$.
\end{proof}

\begin{exercise}
Verify the last assertion of the proof.
\end{exercise}

For any nonnegative measure $\chi$, set $V_{\chi}=(-L+M_{_{\frac{\chi}%
{\lambda}}})^{-1}$ and $G_{\chi}=V_{\chi}M_{\frac{1}{\lambda}}=(M_{\lambda
}+M_{\chi}-C)^{-1}$. It is a symmetric nonnegative function on $X\times X$.
$G_{0}$ is the Green function $G$, and $G_{\chi}$ can be viewed as the Green
function of the energy form $e_{\chi}=e+\left\|  \quad\right\|  _{L^{2}(\chi
)}^{2}$.

Note that $e_{\chi}$ has the same conductances $C$ as $e,$ but $\chi$ is added
to the killing measure. Note also that $V_{\chi}$ is not the potential of the
Markov chain associated with $e_{\chi}$ when one takes exponential holding
times of parameter $1$: the holding time expectation at $x$ becomes
$\frac{1}{1+\chi(x)}$. But the Green function is intrinsic i.e. invariant
under a change of time scale. Still, we have by Feynman Kac formula
\[
\int_{0}^{\infty}\mathbb{E}_{x}(e^{-\int_{0}^{t}\frac{\chi}{\lambda}(x_{s}%
)ds}1_{\{x_{t}=y\}})dt=[V_{\chi}]_{y}^{x}.
\]
We have also the ''generalized resolvent equation'' $V-V_{\chi}=VM_{\frac{\chi
}{\lambda}}V_{\chi}=V_{\chi}M_{\frac{\chi}{\lambda}}V$. Then,
\begin{equation}
G-G_{\chi}=GM_{\chi}G_{\chi}=G_{\chi}M_{\chi}G \label{Res}%
\end{equation}

\begin{exercise}
Prove the generalized resolvent equation.
\end{exercise}

Note that the recurrent Green operator $G$ defined on signed measures of zero
charge is the limit of the transient Green operator $G_{\chi}$, as
$\chi\rightarrow0$.

\section{Recurrent extension of a transient chain}

It will be convenient to add a cemetery point $\Delta$ to $X$, and extend $C$,
$\lambda$ and $G$ to $X^{\Delta}=\{X\cup\Delta\}$ by setting , $\lambda
_{\Delta}=\sum_{x\in X}\kappa_{x}$, $C_{x,\Delta}=\kappa_{x}$\ and
$G^{x,\Delta}=G^{\Delta,x}=G^{\Delta,\Delta}=0$ for all $x\in X$. Note that
$\lambda(X^{\Delta})=\sum_{X\times X}C_{x,y}+2\sum_{X}\kappa_{x}%
=\lambda(X)+\lambda_{\Delta}$.

One can consider the recurrent ''resurrected'' Markov chain defined by the
extensions of the conductances to $X^{\Delta}$. An energy $e^{\Delta}$\ is
defined by the formula
\[
e^{\Delta}(z)=\frac{1}{2}\sum_{x,y\in X^{\Delta}}C_{x,y}(z^{x}-z^{y}%
)(\overline{z}^{x}-\overline{z}^{y})
\]
From the irreducibility assumption, it follows that $e^{\Delta}$ vanishes only
on constants. We denote by $P^{\Delta}$ the transition kernel on $X^{\Delta}$
defined by%
\[
\lbrack P^{\Delta}]_{y}^{x}=\frac{C_{x,y}}{\sum_{y\in X^{\Delta}}C_{x,y}%
}=\frac{C_{x,y}}{\lambda_{x}}%
\]
Note that $P^{\Delta}1=1$ so that $\lambda$ is now an invariant measure.with
$\lambda_{x}[P^{\Delta}]_{y}^{x}=\lambda_{y}[P^{\Delta}]_{x}^{y}$ on
$X^{\Delta}$. Also
\[
e^{\Delta}(f,g)=\left\langle f-P^{\Delta}f,g\right\rangle _{\lambda}%
\]
Denote $V^{\Delta}$\ and $G^{\Delta}$ the associated potential and Green operators.

Note that for $\mu$ carried by $X$, for all $x\in X$, denoting by
$\varepsilon_{\Delta}$ the unit point mass at $\Delta$,
\begin{align*}
\mu_{x}  &  =e^{\Delta}(G^{\Delta}(\mu-\mu(X)\varepsilon_{\Delta}%
),1_{x})=\lambda_{x}((I-P^{\Delta})G^{\Delta}(\mu-\mu(X)\varepsilon_{\Delta
})(x)\\
&  =\lambda_{x}((I-P)G^{\Delta}(\mu-\mu(X)\varepsilon_{\Delta}))(x)-\kappa
_{x}G^{\Delta}(\mu-\mu(X)\varepsilon_{\Delta})(\Delta).
\end{align*}
Hence, applying $G$ , it follows that on $X^{\Delta}$,
\[
G\mu=G^{\Delta}(\mu-\mu(X)\varepsilon_{\Delta})-G^{\Delta}(\mu-\mu
(X)\varepsilon_{\Delta})(\Delta)G\kappa=G^{\Delta}(\mu-\mu(X)\varepsilon
_{\Delta})-G^{\Delta}(\mu-\mu(X)\varepsilon_{\Delta})(\Delta).
\]
Moreover, as $G^{\Delta}(\mu-\mu(X)\varepsilon_{\Delta})$ is in $\lambda
^{\perp}$, integrating by $\lambda$, we obtain that%
\[
\sum_{x\in X}\lambda_{x}G(\mu)^{x}=-G^{\Delta}(\mu-\mu(X)\varepsilon_{\Delta
})(\Delta)\lambda(X^{\Delta}).
\]
Therefore, $G^{\Delta}(\mu-\mu(X)\varepsilon_{\Delta})(\Delta
)=\frac{-\left\langle \lambda,G\mu\right\rangle }{\lambda(X^{\Delta})}$ and we
get the following:

\begin{proposition}
For any measure $\mu$ on $X$, $G^{\Delta}(\mu-\mu(X)\varepsilon_{\Delta
})=-\frac{\left\langle \lambda,G\mu\right\rangle }{\lambda(X^{\Delta})}+G\mu.$
\end{proposition}

This type of extension can be done in a more general context ( See \cite{LJ4}
and Dellacherie-Meyer \cite{DelMey})

\bigskip

\begin{remark}
Conversely, a recurrent chain can be killed at any point $x_{0}$ of $X$,
defining a Green function $G^{X-\{x_{0}\}}$ on $X-\{x_{0}\}$. Then, for any
$\mu$ carried by $X-\{x_{0}\}$,
\[
G^{X-\{x_{0}\}}\mu=G(\mu-\mu(X)\varepsilon_{x_{0}})-G(\mu-\mu(X)\varepsilon
_{x_{0}})(x_{0}).
\]
This transient chain allows to recover the recurrent one by the above procedure.
\end{remark}

\begin{exercise}
Consider a transient process which is killed with probability $p$ at each
passage in $\Delta$. Determine the associated energy and Green operator.
\end{exercise}

\section{Transfer matrix\label{difforms}}

Let us suppose in this section that we are in the \emph{recurrent case}: We
can define a scalar product on the space $\mathbb{A}$ of functions on $E^{o}$
(oriented edges) as follows

$\left\langle \omega,\eta\right\rangle _{\mathbb{A}}=\frac{1}{2}\sum
_{x,y}C_{x,y}\omega^{x,y}\eta^{x,y}.$ Denoting as in \cite{Lyo2}%
\ \ $df^{u,v}=f^{v}-f^{u}$, we note that $\left\langle df,dg\right\rangle
_{\mathbb{A}}=e(f,g)$. In particular
\[
\left\langle df,dG(\delta_{y}-\delta_{x})\right\rangle _{\mathbb{A}}=df^{x,y}%
\]
Denote $\mathbb{A}_{-}$, ($\mathbb{A}_{+}$) the space of real valued functions
on $E^{o}$ odd (even) for orientation reversal. Note that the spaces
$\mathbb{A}_{+}$ and $\mathbb{A}_{-}$ are orthogonal for the scalar product
defined on $\mathbb{A}$. The space $\mathbb{A}_{-}$ should be viewed as the
space of ''discrete differential forms''.

Following this analogy, define for any $\alpha$ in $\mathbb{A}_{-}$, define
$d^{\ast}\alpha$ by $(d^{\ast}\alpha)^{x}=-\sum_{y\in X}P_{y}^{x}\alpha^{x,y}%
$. Note it belongs to $\lambda^{\perp}$ as $\sum_{x,y}C_{x,y}\alpha^{x,y}$ vanishes.

We have
\begin{align*}
\left\langle \alpha,df\right\rangle _{\mathbb{A}}  &  =\frac{1}{2}\sum
_{x,y}\lambda_{x}P_{y}^{x}\alpha^{x,y}(f^{y}-f^{x})\\
&  =\frac{1}{2}\sum_{x\in X}(d^{\ast}\alpha)^{x}f^{x}\lambda_{x}-\frac{1}%
{2}\sum_{x,y}\lambda_{x}P_{y}^{x}\alpha^{y,x}f^{y}=\sum_{x\in X}(d^{\ast
}\alpha)^{x}f^{x}\lambda_{x}%
\end{align*}
as the two terms of the difference are in fact opposite since $\alpha$ is skew
symmetric. The image of $d$ and the kernel of $d^{\ast}$are therefore
orthogonal in $\mathbb{A}_{-}$. We say $\alpha$ in $\mathbb{A}_{-}$\ is
harmonic iff $d^{\ast}\alpha=0$.

Moreover,
\[
e(f,f)=\left\langle df,df\right\rangle _{\mathbb{A}}=\sum_{x\in X}(d^{\ast
}df)^{x}f^{x}\lambda_{x}.
\]

Note also that for any function $f$ ,%
\[
d^{\ast}df=-Pf+f=-Lf.
\]

$d$ is the discrete analogue of the differential and $d^{\ast}$ the analogue
of its adjoint, depending on the metric which is here defined by the
conductances. \bigskip$L\,$\ is a discrete version of the Laplacian.

\begin{proposition}
The projection of any $\alpha$ in $\mathbb{A}_{-}$\ on the image of $d$ is
$dVd^{\ast}(\alpha)$.
\end{proposition}

\begin{proof}
Indeed, for any function $g,$ $\left\langle \alpha,dg\right\rangle
_{\mathbb{A}}=\left\langle d^{\ast}\alpha,g\right\rangle _{\lambda}%
=e(Vd^{\ast}\alpha,g)=\left\langle dVd^{\ast}(\alpha),dg\right\rangle
_{\mathbb{A}}$.\newline 
\end{proof}

We now can come to the definition of the transfer matrix: Set $\alpha
_{(u,v)}^{x,y}=\pm\frac{1}{C_{u,v}}$ if $(x,y)=\pm(u,v)$ and $0$ elsewhere.
Then $\lambda_{x}d^{\ast}\alpha_{(u,v)}(x)=\delta_{v}^{x}-\delta_{u}^{x}$ and
$dVd^{\ast}(\alpha_{(u,v)})=dG(\delta_{v}-\delta_{u})$. Note that given any
orientation of the graph, the family $\{\alpha_{(u,v)}^{\ast}=\sqrt{C_{u,v}%
}\alpha_{(u,v)},\;(u,v)\in E^{+}\}$ is an orthonormal basis of $\mathbb{A}%
_{-}$ (here $E^{+}$ denotes the set of positively oriented edges).

\index{transfer matrix}The symmetric transfer matrix $K^{(x,y),(u,v)}$, indexed by pairs of oriented
edges, is defined to be
\[
K^{(x,y),(u,v)}=[dG(\delta_{v}-\delta_{u})]^{x,y}=G(\delta_{v}-\delta_{u}%
)^{y}-G(\delta_{v}-\delta_{u})^{x}=<dG(\delta_{y}-\delta_{x}),dG(\delta
_{v}-\delta_{u})>_{\mathbb{A}}%
\]
for $x,y,u,v\in X$, with $C_{x,y}C_{u,v}>0$.

As $dG((d^{\ast}\alpha)\lambda)=dVd^{\ast}(\alpha)$ is the projection
$\Pi(\alpha)$\ of $\alpha$ on the image of $d$ in $\mathbb{A}_{-}$, we have also:%

\[
<\alpha_{(u,v)},\Pi(\alpha_{(u^{\prime},v^{\prime})})>_{\mathbb{A}}%
=<\alpha_{(u,v)},dG(\delta_{v^{\prime}}-\delta_{u^{\prime}})>_{\mathbb{A}%
}=K^{(u,v),(u^{\prime},v^{\prime})}%
\]

For every oriented edge $h=(x,y)$ in \ $X$, set $K^{h}=dG(\delta^{y}%
-\delta^{x})$. We have $\left\langle K^{h},K^{g}\right\rangle _{\mathbb{A}%
}=K^{h,g}.$We can view $dG$ as a linear operator mapping the space measures of
total charge zero into $\mathbb{A}_{-}$. As measures of the form $\delta
_{y}-\delta_{x}$ span the space of measures of total charge zero, it is
determined by the transfer matrix.

Note that $d^{\ast}dG\upsilon=\upsilon/\lambda$ for any $\upsilon$ of total
charge zero and that for all $\alpha$ in $\mathbb{A}_{-}$, $(d^{\ast}%
\alpha)\lambda$ has total charge zero.

\bigskip

Consider now, in the transient case, the transfer matrix associated with
$G^{\Delta}$.\newline We see that for $x$\ and $y$ in $X$, $G^{\Delta}%
(\delta_{x}-\delta_{y})^{u}-G^{\Delta}(\delta_{x}-\delta_{y})^{v}=G(\delta
_{x}-\delta_{y})^{u}-G(\delta_{x}-\delta_{y})^{v}$.\newline We can see also
that $G^{\Delta}(\delta_{x}-\delta_{\Delta})=G\delta_{x}-\frac{\left\langle
\lambda,G\delta_{x}\right\rangle }{\lambda(X^{\Delta})}$. So the same identity
holds in $X^{\Delta}$.\newline Therefore, as $G^{x,\Delta}=0$, in all cases,
\[
K^{(x,y),(u,v)}=G^{x,u}+G^{y,v}-G^{x,v}-G^{y,u}%
\]

\begin{exercise}
Cohomology and complex transition matrices.

Consider, in the recurrent case, $\omega\in\mathbb{A}_{-}$ such that $d^{\ast
}\omega=0$. Note that the space $H^{1}$\ of such $\omega$'s is isomorphic to
the first cohomology space, defined as the quotient $\mathbb{A}_{-}%
/\operatorname{Im}(d)$. Prove that $P(I+i\omega)$ is $\lambda-$self adjoint on
$X$, maps $1$ onto $1$ and that we have $\mathbb{E}_{x}(\prod_{j=0}%
^{n-1}(1+\omega(\xi_{j},\xi_{j+1}))1_{\{\xi_{n}=y\}})=[(P(I+i\omega))^{n}%
]_{y}^{x}$.
\end{exercise}

\chapter{Loop measures}

\section{A measure on based loops}

We denote by $\mathbb{P}^{x}$\ the family of probability laws on piecewise
constant paths defined by $P_{t}$.%
\[
\mathbb{P}^{x}(\gamma(t_{1})=x_{1},...,\gamma(t_{h})=x_{h})=P_{t_{1}}%
(x,x_{1})P_{t_{2}-t_{1}}(x_{1},x_{2})\ldots P_{t_{h}-t_{h-1}}(x_{h-1},x_{h})
\]
The corresponding process is a Markov chain in continuous time. It can also be
constructed as the process $\xi_{Nt}$, where $\xi_{n}$ is the discrete time
Markov chain starting at $x$, with transition matrix $P$, and $N_{t}$ an
independent Poisson process.

In the transient case, the lifetime is a.s. finite and denoting by $p(\gamma)$
the number of jumps and $T_{i}$ the jump times, we have:%
\begin{multline*}
\mathbb{P}^{x}(p(\gamma)=k,\gamma_{T_{1}}=x_{1},...,\gamma_{T_{k}}=x_{k}%
,T_{1}\in dt_{1},...,T_{k}\in dt_{k})\\
=\frac{C_{x,x_{1}}...C_{x_{k-1},x_{k}}\kappa_{x_{k}}}{\lambda_{x}%
\lambda_{x_{1}}...\lambda_{x_{k}}}1_{\{0<t_{1}<...<t_{k}\}}e^{-t_{k}}%
dt_{1}...dt_{k}%
\end{multline*}

\bigskip

\index{based loops}For any integer $p\geq2$, let us define a \emph{based loop} with $p$ points in
$X$ as a couple $l=(\xi,\tau)=((\xi_{m},1\leq m\leq p),(\tau_{m},1\leq m\leq
p+1)\mathbb{)}$ in $X^{p}\times\mathbb{R}_{+}^{p+1}$, and set $\xi_{p+1}%
=\xi_{1}$ (equivalently, we can parametrize the associated discrete based loop
by $\mathbb{Z}/p\mathbb{Z}$). The integer $p$ represents the number of points
in the discrete based loop $\xi=(\xi_{1},...,\xi_{p(\xi)})$ and will be
denoted $p(\xi)$, and the $\tau_{m}$ are holding times.  Note however that two time parameters are attached to the base point
since the based loops do not in general end or start with a jump.

Based loops with one point $(p=1)$\ are simply given by a pair $(\xi,\tau
)$\ in $X\times\mathbb{R}_{+}$.

Based loops have a natural time parametrization $l(t)$ and a time period
$T(\xi)=\sum_{i=1}^{p(\xi)+1}\tau_{i}$. If we denote $\sum_{i=1}^{m}\tau_{i}%
$\ by $T_{m}$: $l(t)=\xi_{m}$ on $[T_{m-1},T_{m})$ (with by convention\ $T_{0}%
=0$ and $\xi_{1}=\xi_{p+1}$).

\index{bridge measure}Let $\mathbb{P}_{t}^{x,y}$ denote the (non normalized) ''bridge measure'' on
piecewise constant paths from $x$ to $y$\ of duration $t$ constructed as follows:

If $t_{1}<t_{2}<...<t_{h}<t$,%
\[
\mathbb{P}_{t}^{x,y}(l(t_{1})=x_{1},...,l(t_{h})=x_{h})=[P_{t_{1}}]_{x_{1}%
}^{x_{h}}[P_{t_{2}-t_{1}}]_{x_{2}}^{x_{1}}...[P_{t-t_{h}}]_{y}^{x_{h}}%
\frac{1}{\lambda_{y}}%
\]
Its mass is $p_{t}^{x,y}=\frac{[P_{t}]_{y}^{x}}{\lambda_{y}}$. For any
measurable set $A$\ of piecewise constant paths indexed by $[0\;t]$, we can
also write
\[
\mathbb{P}_{t}^{x,y}(A)=\mathbb{P}_{x}(A\cap\{x_{t}=y\})\frac{1}{\lambda_{y}%
}.
\]
\begin{exercise}
Prove that $\mathbb{P}_{t}^{y,x}$\ is the image of $\mathbb{P}_{t}^{x,y}$\ by
the operation of time reversal on paths indexed by $[0\;t]$.
\end{exercise}

A $\sigma$-finite measure $\mu$ is defined on based loops by%
\[
\mu=\sum_{x\in X}\int_{0}^{\infty}\frac{1}{t}\mathbb{P}_{t}^{x,x}\lambda
_{x}dt
\]
\begin{remark}
The introduction of the factor $\frac{1}{t}$ will be justified in the
following. See in particular formula \ref{gems}. It can be interpreted as the
normalization of the uniform measure on the loop, according to which the base
point is chosen.\bigskip
\end{remark}

From the expression of the bridge measure, we see that by definition of $\mu$,
if $t_{1}<t_{2}<...<t_{h}<t$,%
\begin{equation}
\mu(l(t_{1})=x_{1},...,l(t_{h})=x_{h},T\in dt)=[P_{t_{1}+t-t_{h}}]_{x_{1}%
}^{x_{h}}[P_{t_{2}-t_{1}}]_{x_{2}}^{x_{1}}...[P_{t_{h}-t_{h-1}}]_{x_{h}%
}^{x_{h-1}}\frac{1}{t}dt. \label{mu1}%
\end{equation}

Note also that for $k>1$, using the second expression of $\mathbb{P}_{t}%
^{x,y}$ and the fact that conditionally on $N_{t}=k$, the jump times are
distributed like an increasingly reordered $k-$uniform sample of $[0\;t]$%
\begin{align*}
\lambda_{x}\mathbb{P}_{t}^{x,x}(p  &  =k,\xi_{1}=x_{1},\xi_{2}=x_{2}%
,...,\xi_{_{k}}=x_{k},T_{1}\in dt_{1},...,T_{k}\in dt_{k})\\
&  =1_{\{x=x_{1}\}}e^{-t}\frac{t^{k}}{k!}P_{x_{2}}^{x_{1}}P_{x_{3}}^{x_{2}%
}...P_{x_{1}}^{x_{k}}1_{\{0<t_{1}<...t_{k}<t\}}\frac{k!}{t^{k}}dt_{1}%
...dt_{k}\\
&  =1_{\{x=x_{1}\}}P_{x_{2}}^{x_{1}}P_{x_{3}}^{x_{2}}...P_{x}^{x_{k}%
}1_{\{0<t_{1}<...t_{k}<t\}}e^{-t}dt_{1}...dt_{k}%
\end{align*}
Therefore,
\begin{align}
\mu(p  &  =k,\xi_{1}=x_{1},..,\xi_{k}=x_{k},T_{1}\in dt_{1},..,T_{k}\in
dt_{k},T\in dt)\label{mu3}\\
&  =P_{x_{2}}^{x_{1}}..P_{x_{1}}^{x_{k}}\frac{1_{\{0<t_{1}<...<t_{k}<t\}}}%
{t}e^{-t}dt_{1}...dt_{k}dt
\end{align}

for $k>1$.

Moreover, for one point-loops, $\mu\{p(\xi)=1,\xi_{1}=x_{1},\tau_{1}\in
dt\}=\frac{e^{-t}}{t}dt.$

It is clear on these formulas that for any positive constant $c$, the energy
forms $e$ and $ce$ define the same loop measure.

\section{First properties}

Note that the loop measure is invariant under time reversal.

If $D$ is a subset of $X$, the restriction of $\mu$ to loops contained in $D$,
denoted $\mu^{D}$ is clearly the loop measure induced by the Markov chain
killed at the exit of $D$. This can be called the \textsl{restriction
property}.

Let us recall that this killed Markov chain is defined by the restriction of
$\lambda$ to $D$ and the restriction $P^{D}$ of $P$ to $D^{2}$ (or
equivalently by the restriction $e_{D}$ of the Dirichlet form $e$ to functions
vanishing outside $D$).

As $\int\frac{t^{k-1}}{k!}e^{-t}dt=\frac{1}{k}$, it follows from (\ref{mu3})
that for $k>1$, on based loops,%
\begin{equation}
\mu(p(\xi)=k,\xi_{1}=x_{1},...,\xi_{k}=x_{k})=\frac{1}{k}P_{x_{2}}^{x_{1}%
}...P_{x_{1}}^{x_{k}}. \label{de}%
\end{equation}
In particular, we obtain that, for $k\geq2$
\[
\mu(p=k)=\frac{1}{k}Tr(P^{k})
\]
and therefore, as $Tr(P)=0$, in the transient case:
\begin{equation}
\mu(p>1)=\sum_{2}^{\infty}\frac{1}{k}Tr(P^{k})=-\log(\det(I-P))=\log
(\det(G)\prod_{x}\lambda_{x}) \label{logdet1}%
\end{equation}
since (denoting $M_{\lambda}$ the diagonal matrix with entries $\lambda_{x}$),
we have
\[
{\det(I-P)=\frac{\det(M_{\lambda}-C)}{\det(M_{\lambda})}}%
\]

Note that $\det(G)$ is defined as the determinant of the matrix $G^{x,y}$. It
is the determinant of the matrix representing the scalar product defined on
$\mathbb{R}^{\left|  X\right|  }$ (more precisely, on the space of measures on
$X$) by $G$ in any basis, orthonormal with respect to the natural euclidean
scalar product on $\mathbb{R}^{\left|  X\right|  }$.

Moreover
\[
\int p(l)1_{\{p>1\}}\mu(dl)=\sum_{2}^{\infty}Tr(P^{k})=Tr((I-P)^{-1}P)=Tr(GC)
\]

\section{Loops and pointed loops}

It is clear on formula \eqref{mu1}  that $\mu$ is invariant under the time
shift that acts naturally on based loops.\bigskip

\index{loop}A loop is defined as an equivalence class of based loops for this shift.
Therefore, $\mu$ induces a\textsl{ measure\ on loops also denoted by }$\mu$.

A loop is defined by the discrete loop $\xi^{%
{{}^\circ}%
}$\ formed by the $\xi_{i}$ in circular order, (i.e. up to translation) and
the associated holding times. We clearly have:%
\[
\mu(\xi^{%
{{}^\circ}%
}=(x_{1},x_{2},...,x_{k})^{%
{{}^\circ}%
})=P_{x_{2}}^{x_{1}}...P_{x_{1}}^{x_{k}} \label{gems}%
\]

\index{pointed loops}However, loops are not easy to parametrize, that is why we will work mostly
with based loops or with \textsl{pointed loops}$.$ These are defined as based
loops ending with a jump, or equivalently as loops with a starting point. They
can be parametrized by a based discrete loop and by the holding times at each
point. Calculations are easier if we work with based or pointed loops, even
though we will deal only with functions independent of the base point.\bigskip

The parameters of the pointed loop naturally associated with a based loop are
$\xi_{1},...,\xi_{p}$ and
\[
\tau_{1}+\tau_{p+1}{=\tau_{1}^{\ast},\tau_{i}=\tau_{i}^{\ast},\ 2\leq i\leq p}%
\]
An elementary change of variables, shows the expression of $\mu$ on pointed
loops can be written:%

\begin{equation}
\mu(p=k,\xi_{i}=x_{i},{\tau_{i}^{\ast}\in dt}_{i})=P_{x_{2}}^{x_{1}%
}...P_{x_{1}}^{x_{k}}\frac{t_{1}}{\sum t_{i}}e^{-\sum t_{i}}dt_{1}...dt_{k}.
\label{dd}%
\end{equation}
Trivial ($p=1$) pointed loops and trivial based loops coincide.

Note that loop functionals can be written%
\[
\Phi(l%
{{}^\circ}%
)=\sum1_{\{p=k\}}\Phi_{k}((\xi_{i},\tau_{i}^{\ast}),i=1,...,k)
\]
with $\Phi_{k}$ invariant under circular permutation of the variables
$(\xi_{i},\tau_{i}^{\ast})$.

Then, for non negative $\Phi_{k}$
\[
\int\Phi_{k}(l^{%
{{}^\circ}%
})\mu(dl)=\int\Phi_{k}((x_{i},t_{i})i=1,...,k)P_{x_{2}}^{x_{1}}...P_{x_{1}%
}^{x_{k}}e^{-\sum t_{i}}\frac{t_{1}}{\sum t_{i}}dt_{1}...dt_{k}%
\]
and by invariance under circular permutation, the term $t_{1}$ can be replaced
by any $t_{i}$. Therefore, adding up and dividing by $k$, we get that%

\[
\int\Phi_{k}(l^{%
{{}^\circ}%
})\mu(dl)=\int\frac{1}{k}\Phi_{k}((x_{i},t_{i})i=1,...,k)P_{x_{2}}^{x_{1}%
}...P_{x_{1}}^{x_{k}}e^{-\sum t_{i}}dt_{1}...dt_{k}.
\]

The expression on the right side, applied to any pointed loop functional
defines a different measure on pointed loops, we will denote by $\mu^{\ast}$.
It induces the same measure as $\mu$ on loops.

We see on this expression that conditionally on the discrete loop, the holding
times of the loop are independent exponential variables.%

\begin{equation}
\mu^{\ast}(p=k,\xi_{i}=x_{i},{\tau_{i}^{\ast}\in dt}_{i})=\frac{1}{k}%
\prod_{i\in\mathbb{Z}/p\mathbb{Z}}\frac{C_{\xi_{i},\xi_{i+1}}}{\lambda
_{\xi_{i}}}e^{-t_{i}}dt_{i} \label{ddd}%
\end{equation}

Conditionally on $p(\xi)=k,$ $T$ is a gamma variable of density $\frac{t^{k-1}%
}{(k-1)!}e^{-t}$ on $\mathbb{R}_{+}$ and ${(\frac{{\tau_{i}^{\ast}}}%
{T},\ 1\leq i\leq k)}$ an independent ordered $k$-sample of the uniform
distribution on $(0,T)$ (whence the factor $\frac{1}{t}$). Both are
independent, conditionally on the number of points $p$ of the discrete loop.
We see that $\mu$ on based loops is obtained from $\mu$ on the loops by
choosing the base point uniformly. On the other hand, it induces a choice of
$\xi_{1}$ biased by the size of the $\tau_{i}^{\ast}$'s, different from
$\mu^{\ast}$ for which this choice is uniform (whence the factor $\frac{1}{k}%
$). But we will consider only loop functionals for which $\mu$ and $\mu^{\ast
}$ coincide.

It will be convenient to rescale the holding time at each $\xi_{i}$\ by
$\lambda_{\xi_{i}}$ and set
\[
\widehat{\tau}_{i}=\frac{\tau_{i}^{\ast}}{\lambda_{\xi_{i}}}.
\]

The discrete part of the loop is the most important, though we will see that
to establish a connection with Gaussian fields it is necessary to consider
occupation times. The simplest variables are the number of jumps from $x$ to
$y$, defined for every oriented edge $(x,y)$%
\[
N_{x,y}=\#\{i:\xi_{i}=x,\xi_{i+1}=y\}
\]
(recall the convention $\xi_{p+1}=\xi_{1})$ and%
\[
N_{x}=\sum_{y}N_{x,y}%
\]
Note that $N_{x}=\#\{i\geq1:\xi_{i}=x\}$\ except for trivial one point loops
for which it vanishes.

Then, the measure on pointed loops (\ref{dd}) can be rewritten as:%
\begin{align}
\mu^{\ast}(p  &  =1,\xi=x,\widehat{\tau}\in dt)=e^{-\lambda_{x}t}\frac{dt}%
{t}\text{ and }\label{d}\\
\mu^{\ast}(p  &  =k,\xi_{i}=x_{i},\widehat{\tau}_{i}\in dt_{i})=\frac{1}%
{k}\prod_{x,y}C_{x,y}^{N_{x,y}}\prod_{x}\lambda_{x}^{-N_{x}}\prod
_{i\in\mathbb{Z}/p\mathbb{Z}}\lambda_{\xi_{i}}e^{-\lambda_{\xi_{i}}t_{i}%
}dt_{i}.
\end{align}

\index{bridge measure}Another \textsl{bridge measure }$\mu^{x,y}$ can be defined on paths $\gamma$
from $x$ to $y$:
\[
\mu^{x,y}(d\gamma)=\int_{0}^{\infty}\mathbb{P}_{t}^{x,y}(d\gamma)dt.
\]
Note that the mass of $\mu^{x,y}$ is $G^{x,y}$. We also have, with similar
notations as the one defined for loops, $p$ denoting the number of jumps%
\begin{multline*}
\mu^{x,y}(p(\gamma)=k,\gamma_{T_{1}}=x_{1},...,\gamma_{T_{k-1}}=x_{k-1}%
,T_{1}\in dt_{1},...,T_{k}\in dt_{k},T\in dt)\\
=\frac{C_{x,x_{1}}C_{x_{1},x_{2}}...C_{x_{k-1},y}}{\lambda_{x}\lambda_{x_{1}%
}...\lambda_{y}}1_{\{0<t_{1}<...<t_{k}<t\}}e^{-t}dt_{1}...dt_{k}dt.
\end{multline*}
From now on, we will assume, unless otherwise specified, that we are in the
\emph{transient case.}

For any $x\neq y$ in $X$ and $s\in\lbrack0,1]$, setting $P_{v}^{(s),u}%
=P_{v}^{u}$ if $(u,v)\neq(x,y)$ and $P_{y}^{(s),x}=sP_{y}^{x}$, we can prove
in the same way as (\ref{logdet1})\ that:%
\[
\mu(s^{N_{x,y}}1_{\{p>1\}})=-\log(\det(I-P^{(s)})).
\]
Differentiating in $s=1$, and remenbering that for any invertible matrix
function $M(s)$, $\frac{d}{ds}\log(\det(M(s))=Tr(M^{\prime}(s)M(s)^{-1})$,\ it
follows that:%
\[
\mu(N_{x,y})=[(I-P)^{-1}]_{x}^{y}P_{y}^{x}=G^{x,y}C_{x,y}%
\]
and
\begin{equation}
\mu(N_{x})=\sum_{y}\mu(N_{x,y})=\lambda_{x}G^{x,x}-1 \label{munx}%
\end{equation}
(as $G(M_{\lambda}-C)=Id$).

\begin{exercise}
Show that more generally
\[
\mu(N_{x,y}(N_{x,y}-1)...(N_{x,y}-k+1))=(k-1)!(G^{x,y}C_{x,y})^{k}.
\]

Hint: Show that if $M^{\prime\prime}(s)$ vanishes,
\[
\frac{d^{n}}{ds^{n}}\log(\det(M(s)))=(-1)^{n-1}(n-1)!Tr((M^{\prime
}(s)M(s)^{-1})^{n}).
\]
\end{exercise}

\begin{exercise}
Show that more generally, if $x_{i},y_{i}$ are $n$ distinct oriented edges:%

\[
\mu(\prod N_{x_{i},y_{i}})=\prod C_{x_{i},y_{i}}\frac{1}{n}\sum_{\sigma
\in\mathcal{S}_{n}}\prod G^{y_{\sigma(i)},x_{\sigma(i+1)}}%
\]

Hint: Introduce $[P^{(s_{1},...,s_{n})}]_{y}^{x}$ equal to $P_{y}^{x}$ if
$(x,y)\neq(x_{i},y_{i})$ for all $i$, and equal to $s_{i}P_{y_{i}}^{x_{i}}$ if
$(x,y)=(x_{i},y_{i})$.
\end{exercise}

We finally note that if $C_{x,y}>0$, any path segment on the graph starting at
$x$ and ending at $y$ can be naturally extended into a loop by adding a jump
from $y$ to $x$. We have the following

\begin{proposition}
\label{bridg}
 For $C_{x,y}>0$, the natural extension of $\mu^{x,y}$ to loops
coincides with $\frac{N_{y,x}(l)}{C_{x,y}}\mu(dl)$.
\end{proposition}

\begin{proof}
The first assertion follows from the formulas, noticing that a loop $l$\ can
be associated to $N_{y,x}(l)$ distinct bridges from $x$ to $y$, \ obtained by
''cutting'' one jump from $y$ to $x$.
\end{proof}

Note that a) shows that the loop measure induces bridge measures $\mu^{x,y}$
when ${C_{x,y}>0}$. If $C_{x,y}$ vanishes, an arbitrarily small positive
perturbation creating a non vanishing conductance between $x$ and $y$ allows
to do it. More precisely, denoting by $e^{(\varepsilon)}$ the energy form
equal to $e$ except for the additional conductance $C_{x,y}^{(\varepsilon
)}=\varepsilon$, \ $\mu^{x,y}$ can be represented as $\frac{d}{d\varepsilon
}\mu^{e^{(\varepsilon)}}|_{\varepsilon=0}$.

\section{Occupation field\label{Occup}}

\index{occupation field, local times}To each loop $l^{%
{{}^\circ}%
}$ we associate local times, i.e. an occupation field $\{\widehat{l_{x}},x\in
X\}$ defined by
\[
\widehat{l}^{x}=\int_{0}^{T(l)}1_{\{l(s)=x\}}\frac{1}{\lambda_{l(s)}}%
ds=\sum_{i=1}^{p(l)}1_{\{\xi_{i}=x\}}\widehat{\tau_{i}}%
\]
for any representative $l=(\xi_{i},\tau_{i}^{\ast})$ of $l^{\circ}$.

For a path $\gamma$, $\widehat{\gamma}$ is defined in the same way.\newline
Note that%
\begin{equation}
\mu((1-e^{-\alpha\widehat{l}^{x}})1_{\{p=1\}})=\int_{0}^{\infty}%
e^{-t}(1-e^{-\frac{\alpha}{\lambda_{x}}t})\frac{dt}{t}=\log(1+\frac{\alpha
}{\lambda_{x}}). \label{triv}%
\end{equation}
The proof goes by expanding $1-e^{-\frac{\alpha}{\lambda_{x}}t}$ before the
integration, assuming first that $\alpha$ is small and then by analyticity of
both members, or more elegantly, noticing that $\int_{a}^{b}(e^{-cx}%
-e^{-dx})\frac{dx}{x}$ is symmetric in $(a,b)$\ and $(c,d)$, by Fubini's theorem.

In particular, $\mu(\widehat{l}^{x}1_{\{p=1\}})=\frac{1}{\lambda_{x}}$.

From formula \eqref{ddd}  , we get easily that the joint conditional
distribution of $(\widehat{l}^{x},\ x\in X)$ given $(N_{x},\ x\in X)$ is a
product of gamma distributions. In particular, from the expression of the
moments of a gamma distribution, we get that for any function $\Phi$ of the
discrete loop and $k\geq1$,
\[
\mu((\widehat{l}^{x})^{k}1_{\{p>1\}}\Phi)=\lambda_{x}^{-k}\mu((N_{x}%
+k-1)...(N_{x}+1)N_{x}\Phi).
\]
In particular, by \eqref{munx}  $\mu(\widehat{l}^{x})=\frac{1}{\lambda_{x}%
}[\mu(N_{x})+1]=G^{x,x}$.

Note that functions of $\widehat{l}$ are not the only functions naturally
defined on the loops. \index{multiple local times}Other such variables of interest are, for $n\geq2$, the
multiple local times, defined as follows:%
\[
\widehat{l}^{x_{1},...,x_{n}}=\sum_{j=0}^{n-1}\int_{0<t_{1}<...<t_{n}%
<T}1_{\{l(t_{1})=x_{1+j},...,l(t_{n-j})=x_{n},...,l(t_{n})=x_{j}\}}%
\prod\frac{1}{\lambda_{x_{i}}}dt_{i}.
\]
It is easy to check that, when the points $x_{i}$ are distinct,
\begin{equation}
\widehat{l}^{x_{1},...,x_{n}}=\sum_{j=0}^{n-1}\sum_{1\leq i_{1}<...<i_{n}\leq
p(l)}\prod_{l=1}^{n}1_{\{\xi_{i_{l}}=x_{l+j}\}}\widehat{\tau_{i_{l}}}.
\label{tamul}%
\end{equation}
Note that in general $\widehat{l}^{x_{1},...,x_{k}}$ cannot be expressed in
terms of $\widehat{l}$, but
\[
\widehat{l}^{x_{1}}...\widehat{l}^{x_{n}}=\frac{1}{n}\sum_{\sigma
\in\mathcal{S}_{n}}\widehat{l}^{x_{\sigma(1)},...,x_{\sigma(n)}}.
\]

In particular, $\widehat{l}^{x,...,x}=\frac{1}{(n-1)!}[\widehat{l}^{x}]^{n}$.
It can be viewed as a $n$-th self intersection local time.

One can deduce from the definitions of $\mu$\ the following:

\begin{proposition}
$\mu(\widehat{l}^{x_{1},...,x_{n}})=G^{x_{1},x_{2}}G^{x_{2},x_{3}}%
...G^{x_{n},x_{1}}.$

In particular, $\mu(\widehat{l}^{x_{1}}...\widehat{l}^{x_{n}})=\frac{1}{n}%
\sum_{\sigma\in\mathcal{S}_{n}}G^{x_{\sigma(1)},x_{\sigma(2)}}G^{x_{\sigma
(2)},x_{\sigma(3)}}...G^{x_{\sigma(n)},x_{\sigma(1)}}.$
\end{proposition}

\begin{proof}
Let us denote $\frac{1}{\lambda_{y}}$ $[P_{t}]_{y}^{x}$ by $p_{t}^{x,y}$ or
$p_{t}(x,y)$.\ From the definition of $\widehat{l}^{x_{1},...,x_{n}}$ and
$\mu$, $\mu(\widehat{l}^{x_{1},...,x_{n}})$ equals:
\[
\sum_{x}\lambda_{x}\sum_{j=0}^{n-1}\int\int_{\{0<t_{1}<...<t_{n}<t\}}%
\frac{1}{t}p_{t_{1}}(x,x_{1+j})\ldots p_{t-t_{n}}(x_{n+j},x)\prod dt_{i}dt.
\]
where sums of indices $k+j$ are computed $\operatorname{mod}(n)$. By the
semigroup property, it equals
\[
\sum_{j=0}^{n-1}\int\int_{\{0<t_{1}<...<t_{n}<t\}}\frac{1}{t}p_{t_{2}-t_{1}%
}(x_{1+j},x_{2+j})\ldots p_{t_{1}+t-t_{n}}(x_{n+j},x_{1+j})\prod dt_{i}dt.
\]

Performing the change of variables $v_{2}=t_{2}-t_{1},.., v_{n}=t_{n}%
-t_{n-1},v_{1}=t_{1}+t-t_{n}$, and $v=t_{1}$, we obtain:
\begin{align*}
\sum_{j=0}^{n-1}\int_{\{0<v<v_{1},0<v_{i}\}}  &  \frac{1}{v_{1}+...+v_{n}%
}p_{v_{2}}(x_{1+j},x_{2+j})\ldots p_{v_{1}}(x_{n+j},x_{1+j})\prod dv_{i}dv\\
&  =\sum_{j=0}^{n-1}\int_{\{0<v_{i}\}}\frac{v_{1}}{v_{1}+...+v_{n}}p_{v_{2}%
}(x_{1+j},x_{2+j})\ldots p_{v_{1}}(x_{n+j},x_{1+j})\prod dv_{i}\\
&  =\sum_{j=1}^{n}\int_{\{0<v_{i}\}}\frac{v_{j}}{v_{1}+...+v_{n}}p_{v_{2}%
}(x_{1},x_{2})\ldots p_{v_{1}}(x_{n},x_{1})\prod dv_{i}\\
&  =\int_{\{0<v_{i}\}}p_{v_{2}}(x_{1},x_{2})\ldots p_{v_{1}}(x_{n},x_{1})\prod
dv_{i}\\
&  =G^{x_{1},x_{2}}G^{x_{2},x_{3}}...G^{x_{n},x_{1}}.
\end{align*}
Note that another proof can be derived from formula \eqref{tamul}  .
\end{proof}

\begin{exercise}
\label{shuf}(Shuffle product) Given two positive integers $n>k$, let $\mathcal{P}_{n,k}%
$\ be the family of partitions of $\{1,2,...n\}$\ into $k$ consecutive non
empty intervals $I_{l}=(i_{l},i_{l}+1,...,i_{l+1}-1)$\ with $i_{1}%
=1<i_{2}<...<i_{k}<i_{k+1}=n+1$.\newline Show that%
\[
\widehat{l}^{x_{1},...,x_{n}}\widehat{l}^{y_{1},...,y_{m}}=\sum_{j=0}%
^{m-1}\sum_{k=1}^{\inf(n,m)}\sum_{I\in\mathcal{P}_{n,k}}\sum_{J\in
\mathcal{P}_{m,k}}\widehat{l}^{x_{I_{1}},y_{j+J_{1}},x_{I_{2}},...y_{j+J_{k}}}%
\]
where for example the term $y_{j+J_{1}}$ appearing in the upper index should
be read as ${j+j_{1}}$, \ldots, ${j+j_{2}-1}$.
\end{exercise}

Similarly, we can define $N_{(x_{1},y_{1}),...(x_{n},y_{n)}}$ to be
\[
\sum_{j=0}^{n-1}\sum_{1\leq i_{1}<\ldots<i_{n}\leq p(l)}\prod_{l=1}%
^{n}1_{\{\xi_{i_{l}}=x_{l+j},\xi_{i_{l}+1}=y_{l+j}\}}.
\]

If $(x_{i},y_{i})=(x,y)$ for all $i$, it equals $\frac{N_{x,y}(N_{x,y}%
-1)...(N_{x,y}-n+1)}{(n-1)!}.$

Notice that
\[
\prod N_{(x_{i},y_{i})}=\frac{1}{n}\sum_{\sigma\in\mathcal{S}_{n}%
}N_{(x_{\sigma(1)},y_{\sigma(1)}),...(x_{\sigma(n)},y)}.
\]

Then we have the following:

\begin{proposition}
$\int N_{(x_{1},y_{1}),...,(x_{n},y_{n})}(l)\mu(dl)=\Big(          \prod
C_{x_{i},y_{i}}\Big)          G^{y_{1},x_{2}}G^{y_{2},x_{3}}...G^{y_{n},x_{1}}.$
\end{proposition}

The proof is left as exercise.

\begin{exercise}
\label{self}For $x_{1}=x_{2}=...=x_{k}$, we could define different self
intersection local times
\[
\widehat{l}^{x,(k)}=\sum_{1\leq i_{1}<..<i_{k}\leq p(l)}\prod_{l=1}%
^{k}1_{\{\xi_{i_{l}}=x\}}\widehat{\tau_{i_{l}}}%
\]
which vanish on $N_{x}<k$. Note that
\[
\widehat{l}^{x,(2)}=\frac{1}{2}((\widehat{l}^{x})^{2}-\sum_{i=1}%
^{p(l)}1_{\{\xi_{i}=x\}}(\widehat{\tau_{i}})^{2}.
\]

\begin{enumerate}
\item For any function $\Phi$ of the discrete loop, show that
\[
\mu(\widehat{l}^{x,2}\Phi)=\lambda_{x}^{-2}\mu\Big(          \frac{N_{x}%
(N_{x}-1)}{2}1_{\{N_{x}\geq2\}}\Phi\Big)          .
\]

\item More generally prove in a similar way that
\[
\mu(\widehat{l}^{x,(k)}\Phi)=\lambda_{x}^{-k}\mu\Big(          \frac{N_{x}%
(N_{x}-1)...(N_{x}-k+1)}{k!}1_{\{N_{x}\geq k\}}\Phi\Big)          .
\]
\end{enumerate}
\end{exercise}

Let us come back to the occupation field to compute its Laplace transform.
From the Feynman-Kac formula, it comes easily that, denoting $M_{\frac{\chi
}{\lambda}}$ the diagonal matrix with coefficients $\frac{\chi_{x}}%
{\lambda_{x}}$
\[
\mathbb{P}_{t}^{x,x}(e^{-\left\langle \widehat{l},\chi\right\rangle
}-1)=\frac{1}{\lambda_{x}}\Big(          \exp(t(P-I-M_{_{\frac{\chi}{\lambda}%
}}))_{x}^{x}-\exp(t(P-I))_{x}^{x}\Big)          .
\]
Integrating in $t$ after expanding, we get from the definition of $\mu$ (first
for $\chi$ small enough):%
\begin{align*}
\int(e^{-\left\langle \widehat{l},\chi\right\rangle }-1)d\mu(l) =  &
\sum_{k=1}^{\infty}\int_{0}^{\infty}[Tr((P-M_{_{\frac{\chi}{\lambda}}}%
)^{k})-Tr((P)^{k})]\frac{t^{k-1}}{k!}e^{-t}dt\\
=  &  \sum_{k=1}^{\infty}\frac{1}{k}[Tr((P-M_{_{\frac{\chi}{\lambda}}}%
)^{k})-Tr((P)^{k})]\\
=  &  -Tr(\log(I-P+M_{_{\frac{\chi}{\lambda}}}))+Tr(\log(I-P)).
\end{align*}
Hence, as $Tr(\log)=\log(\det)$%
\begin{align*}
\int(e^{-\left\langle \widehat{l},\chi\right\rangle }-1)d\mu(l)  &  =\log
[\det(-L(-L+M_{\chi/\lambda})^{-1})]\\
&  =-\log\det(I+VM_{\frac{\chi}{\lambda}})=\log\det(I+GM_{\chi})
\end{align*}
which now holds for all non negative $\chi$ as both members are analytic in
$\chi$. Besides, by the ''resolvent'' equation (\ref{Res}):%
\begin{equation}
\det(I+GM_{\chi})^{-1}=\det(I-G_{\chi}M_{\chi})=\frac{\det(G_{\chi})}{\det
(G)}. \label{F1}%
\end{equation}
Note that $\det(I+GM_{\chi})=\det(I+M_{\sqrt{\chi}}GM_{\sqrt{\chi}})$ and
$\det(I-G_{\chi}M_{\chi})=\det(I-M_{\sqrt{\chi}}G_{\chi}M_{\sqrt{\chi}})$, so
we can deal with symmetric matrices. Finally we have

\begin{proposition}
\label{LAPLT}$\mu(e^{-\left\langle \widehat{l},\chi\right\rangle }%
-1)=-\log(\det(I+M_{\sqrt{\chi}}GM_{\sqrt{\chi}}))=\log(\frac{\det(G_{\chi}%
)}{\det(G)})$
\end{proposition}

Note that in particular $\mu(e^{-t\widehat{l}^{x}}-1)=-\log(1+tG^{x,x})$.
Consequently, the image measure of $\mu$ by $\widehat{l}^{x}$ is
$\displaystyle          {1_{\{s>0\}}\frac{1}{s}\exp(-\frac{s}{G^{x,x}})ds}$.

Considering the Laguerre-type polynomials $D_{k}$ with generating function
\[
\sum_{1}^{\infty}t^{k}D_{k}(u)=e^{\frac{ut}{1+t}}-1
\]
and setting $\sigma_{x}=G^{x,x}$, we have:

\begin{proposition}
The variables $\frac{1}{\sqrt{k}}\sigma_{x}^{k}D_{k}(\frac{\widehat{l}^{x}%
}{\sigma_{x}})$ are orthonormal in $L^{2}(\mu)$ for $k>0$, and more generally%
\[
\mathbb{E}(\sigma_{x}^{k}D_{k}(\frac{\widehat{l}^{x}}{\sigma_{x}})\sigma
_{y}^{j}D_{j}(\frac{\widehat{l}^{y}}{\sigma_{y}}))=\frac{1}{k}\delta
_{k,j}(G^{x,y})^{2k}.
\]
\end{proposition}

\begin{proof}
By proposition \ref{LAPLT} ,%
\begin{multline*}
\int(1-e^{\frac{\widehat{l}^{x}t}{1+\sigma_{x}t}})(1-e^{\frac{\widehat{l}%
^{y}s}{1+\sigma_{y}s}})\mu(dl)\\
=\log(1-\frac{\sigma_{x}t}{1+\sigma_{x}t})+\log(1-\frac{\sigma_{y}s}%
{1+\sigma_{y}s})-\log\det\Big(
\begin{array}
[c]{cc}%
1-\frac{\sigma_{x}t}{1+\sigma_{x}t} & -\frac{tG^{x,y}}{1+\sigma_{x}t}\\
-\frac{sG^{x,y}}{1+\sigma_{y}s} & 1-\frac{\sigma_{y}s}{1+\sigma_{y}s}%
\end{array}
\Big) \\
=-\log(1-st(G^{x,y})^{2}).
\end{multline*}

The proposition follows by expanding both sides in powers of $s$ and $t$, and
identifying the coefficients.
\end{proof}

Note finally that if $\chi$ has support in $D$, by the restriction property%
\[
\mu(1_{\{\widehat{l(}X\backslash D)=0\}}(e^{-<\widehat{l},\chi>}%
-1))=-\log(\det(I+M_{\sqrt{\chi}}G^{D}M_{\sqrt{\chi}}))=\log\Big(
\frac{\det(G_{\chi}^{D})}{\det(G^{D})}\Big)          .
\]
Here the determinants are taken on matrices indexed by $D$ and $G^{D}$ denotes
the Green function of the process killed on leaving $D$.

For paths we have $\mathbb{P}_{t}^{x,y}(e^{-\left\langle \widehat{l}%
,\chi\right\rangle })=\frac{1}{\lambda_{y}}\exp(t(L-M_{_{\frac{\chi}{\lambda}%
}}))_{x,y}$. Hence
\[
\mu^{x,y}(e^{-\left\langle \widehat{\gamma},\chi\right\rangle })=\frac{1}%
{\lambda_{y}}((I-P+M_{\chi/\lambda})^{-1})_{x,y}=[G_{\chi}]^{x,y}.
\]
In particular, note that from the resolvent equation (\ref{Res}), we get that%
\[
G^{y,x}=[G_{\varepsilon\delta_{x}}]^{y,x}+\varepsilon\lbrack G_{\varepsilon
\delta_{x}}]^{y,x}G^{x,x}.
\]
Hence $\frac{[G_{\varepsilon\delta_{x}}]^{y,x}}{G^{y,x}}=\frac{1}%
{1+\varepsilon G^{x,x}}$ and therefore, we obtain:

\begin{proposition}
\label{mul}Under the probability $\frac{\mu^{y,x}}{G^{y,x}}$, $\widehat{l}%
_{x}$ follows an exponential distribution of mean $G^{x,x}$.
\end{proposition}

Also $\mathbb{E}^{x}(e^{-\left\langle \widehat{\gamma},\chi\right\rangle
})=\sum_{y}[G_{\chi}]^{x,y}\kappa_{y}$ i.e. $[G_{\chi}\kappa]^{x}$.

Finally, let us note that a direct calculation shows the following result,
analogous to proposition \ref{bridg}\ in which the case $x=y$ was left aside.

\begin{proposition}
\label{bridg2}On loops passing through $x$, $\mu^{x,x}(dl)=\widehat{l}^{x}\mu(dl).$
\end{proposition}

An alternative way to prove the proposition is to check it on multiple local
times, using exercise \ref{shuf}. It can be shown that the algebra formed by
linear combinations of multiple local times generates the loop $\sigma$-field.
Indeed, the discrete loop can be recovered by taking the multiple local time
it indexes and noting it is the unique one of maximal index length among non
vanishing multiple local times indexed by multiplets in which consecutive
points are distinct. Then it is easy to get the holding times as the product
of any of their powers can be obtained from a multiple local time.

\begin{remark}
Propositions \ref{bridg} and \ref{bridg2} can be generalized: For example, if
$x_{i}\;$are $n$ points, $\widehat{l}^{x_{1},...,x_{n}}\mu(dl)$ can be
obtained as the the image by circular concatenation of the product of the
bridge measures $\mu^{x_{i},x_{i+1}}(dl)$ and $\prod\widehat{l}^{x_{i}}%
\mu(dl)$ can be obtained as the sum of the images, by concatenation in all
circular orders, of the product of the bridge measures $\mu^{y_{\sigma
(i)},x_{\sigma(i+1)}}(dl)$. If $(x_{i},y_{i})$ are $n$ oriented edges,
$\prod\frac{N_{x_{i},y_{i}}(l)}{C_{x_{i},y_{i}}}\mu(dl)$ can be obtained as
the sum of the images, by concatenation in all circular orders $\sigma$, of
the product of the bridge measures $\mu^{y_{\sigma(i)},x_{\sigma(i+1)}}%
(dl)$.One can also evaluate expressions of the form $\prod\widehat{l}^{z_{j}%
}\prod\frac{N_{x_{i},y_{i}}(l)}{C_{x_{i},y_{i}}}\mu(dl)$ as a sum of images,
by concatenation in all circular orders, of a product of bridge measures .
\end{remark}

\section{Wreath products}

The following construction gives an interesting information about the number
of distinct points visited by the loop, which is more difficult to evaluate
than the occupation measure.

Associate to each point $x$ of $X$ an integer $n_{x}$. Let $Z$ be the product
of all the groups $\mathbb{Z}/n_{x}\mathbb{Z}$. On the wreath product space
$X\times Z$, define a set of conductances $\widetilde{C}_{(x,z),(x^{\prime
},z^{\prime})}$ by:%

\[
\widetilde{C}_{(x,z),(x^{\prime},z^{\prime})}=\frac{1}{n_{x}n_{x^{\prime}}%
}C_{x,x^{\prime}}\prod_{y\neq x,x^{\prime}}1_{\{z_{y}=z_{y}^{\prime}\}}%
\]
and set $\widetilde{\kappa}_{(x,z)}=\kappa_{x}$. This means in particular that
in the associated Markov chain, the first coordinate is an autonomous Markov
chain on $X$ and that in a jump, the $Z$-configuration can be modified only at
the point from which or to which the first coordinate jumps.

Denote by $\widetilde{e}$ the corresponding energy form. Note that
$\widetilde{\lambda}_{(x,z)}=\lambda_{x}$.

Then, denoting $\widetilde{\mu}$ the loop measure and $\widetilde{P}$ the
transition matrix on $X\times Z$ defined by $\widetilde{e}$, we have the following

\begin{proposition}
$\prod_{x\in X}n_{x}\int1_{\{p>1\}}\prod_{x,\;N_{x}(l)>0}\frac{1}{n_{x}}%
\mu(dl)=\widetilde{\mu}(p>1)=-\log(\det(I-\widetilde{P}))$. In particular, if
$n_{x}=n$ for all $x$,
\[
n^{\left|  X\right|  }\int1_{\{p>1\}}n^{-\#\{x,\;N_{x}(l)>0\}}\mu
(dl)=\widetilde{\mu}(p>1)=-\log(\det(I-\widetilde{P})).
\]
\end{proposition}

\begin{proof}
Each time the Markov chain on $X\times Z$ defined by $\widetilde{e}$\ jumps
from a point above $x$ to a point above $y$, $z_{x}$ and $z_{y}$ are resampled
according to the uniform distribution on $\mathbb{Z}/n_{x}\mathbb{Z\times
Z}/n_{y}\mathbb{Z}$, while the other indices $z_{w}$ are unchanged. It follows
that
\[
[\widetilde{P}^{k}]_{(x,z)}^{(x,z)}=\sum_{x_{1},...,x_{k-1}}P_{x_{1}}%
^{x}P_{x_{2}}^{x_{1}}...P_{x}^{x_{k-1}}\prod_{y\in\{x,x_{1},...,x_{k-1}%
\}}\frac{1}{n_{y}}.
\]
Note that in the set $\{x,x_{1},...,x_{k-1}\}$, distinct points are counted
only once, even if the path visit them several times. There are $\prod_{x\in
X}n_{x}$ possible values for $z$. The detail of the proof is left as an exercise.
\end{proof}

In the case where $X$ is a group and $P$ defines a random walk, $\widetilde
{P}$ is associated with a random walk on $X\times Z$ equipped with its wreath
product structure (Cf \cite{PitSal}).

\section{Countable spaces}

The assumption of finiteness of $X$ can of course be relaxed. On countable
spaces, the previous results extend easily under spectral gap conditions. In
the transient case we consider here, the Dirichlet space $\mathbb{H}$\ is the
space of all functions $f$\ with finite energy $e(f)$ which are limits in
energy norm of functions with finite support, and the energy defines a
Hilbertian scalar product on $\mathbb{H}$.

The energy of a measure is defined as $\sup_{f\in\mathbb{H}}\frac{\mu(f)^{2}%
}{e(f)}$. Finitely supported measures have finite energy. Measures of finite
energy are elements of the dual $\mathbb{H}^{\ast}$ of the Dirichlet space.
The potential $G\mu$ is well defined\ for all finite energy measures $\mu$, by
the identity $e(f,G\mu)=\left\langle f,\mu\right\rangle $, valid for all $f$
in the Dirichlet space. The energy of the measure $\mu$ equals $e(G\mu
)=\left\langle G\mu,\mu\right\rangle $ (see \cite{Fukutak}\ for more information).

Most important examples of countable graphs are the non ramified covering of
finite graphs (Recall that non ramified means that the projection is locally
one to one, i.e. that the projection on $X$ of each vertex $v$\ of the
covering space has the same number of incident edges as $v$ ). Consider a non
ramified covering graph $(Y,F)$ defined by a normal subgroup $H_{x_{0}}$\ of
$\Gamma_{x_{0}}$. The conductances $C$ and the measure $\lambda$ can be lifted
in an obvious way to $Y$\ as $H_{x_{0}}\backslash\Gamma_{x_{0}}$-periodic
functions but the associated Green function $\widehat{G}$\ or semigroup\ are
non trivial. By applying $M_{\lambda}-C$, it is easy to check the following:

\begin{proposition}
\label{greenuniv} $G^{x,y}=\sum_{\gamma\in H_{x_{0}}\backslash\Gamma_{x_{0}}%
}\widehat{G}^{i(x),\gamma(i(y))}$ for any section $i$ of the canonical
projection from $Y$ onto $X$.
\end{proposition}

Let us consider the universal covering (then $H_{x_{0}}\ $is trivial). It is
easy to check it will be transient even in the recurrent case\ as soon as
$(X,E)$ is not circular.

The expression of the Green function $\widehat{G}$ on a universal covering can
be given exactly when it is a regular tree, i.e. in the regular graph case. In
fact a more general result can be proved as follows:

Given a graph $(X,E)$, set $d_{x}=\sum_{y}1_{\{x,y\}\in E}$ (degree or valency
of the vertex $x$), $D_{x,y}=d_{x}\delta_{x,y}$ and denote $A_{x,y}$ the
incidence matrix $1_{E}(\{x,y\})$.

Consider the Green function associated with $\lambda_{x}=(d_{x}-1)u+\frac{1}%
{u}$, with $0<u<\inf(\frac{1}{d_{x}-1},x\in X)$\ and for $\{x,y\}\in E$,
$C_{x,y}=1$.

\begin{proposition}
\label{greencov} On the universal covering $\mathfrak{T}_{x_{0}}$%
,$\ \widehat{G}$ $^{x,y}=u^{d(x,y)}\frac{u}{1-u^{2}}$.
\end{proposition}

\begin{proof}
Note first that as $\frac{1}{u}>d_{x}-1$, $\kappa_{x}$ is positive for all
$x$. Then $\widehat{G}=(M_{\lambda}-C)^{-1}$ can be written $\widehat
{G}=[u^{-1}I+(D-I)u-A]^{-1}$.Moreover, since we are on a tree,
\[
\sum_{x}A_{z,x}u^{d(x,y)}=(d_{z}-1)u^{d(z,y)+1}+u^{d(z,y)-1}%
\]
for $z\neq y$, hence $\sum_{x}(\lambda_{z}\delta_{x}^{z}-A_{z,x})u^{d(x,y)}=0$
for $z\neq y$ and one checks it equals $\frac{1}{u}-u$ for $z=y$.
\end{proof}

It follows from proposition \ref{greenuniv} that for any section $i$ of the
canonical projection from $\mathfrak{T}_{x_{0}}$ onto $X$,
\[
\sum_{\gamma\in\Gamma_{x_{0}}}u^{d(i(x),\gamma(i(y)))}=(\frac{1}{u}%
-u)G^{x,y}.
\]

\section{Zeta functions for discrete loops\label{pres}}

We present briefly the terminology of symbolic dynamics (see for example
\cite{ParPol}) in this simple framework: Setting $f(x_{0},x_{1},...,x_{n}%
,...)=\log(P_{x_{0},x_{1}})$, $P$ induces the Ruelle operator $L_{f}%
$\ associated with $f$.

The pressure\ is defined as the logarithm of the highest eigenvalue $\beta
$\ of $P$. It is associated with a unique positive eigenfunction $h$
(normalized in $L^{2}(\lambda)$), by Perron Frobenius theorem. Note that
$Ph=\beta h$ implies $\lambda hP=\beta\lambda h$ by duality and that in the
recurrent case, the pressure vanishes and $h=\frac{1}{\sqrt{\lambda(X)}}$.

In continuous time, the lowest eigenvalue of $-L$ i.e. $1-\beta$ plays the
role of the pressure

The equilibrium measure associated with $f$, $m=h^{2}\lambda$\ is the law of
the stationnary Markov chain defined by the transition probability
$\frac{1}{\beta h_{x}}P_{y}^{x}h_{y}$.

If $P1=1$, i.e. $\kappa=0$, we can consider a Feynman Kac type perturbation
${P^{(\varepsilon\kappa)}=PM_{\frac{\lambda}{\lambda+\varepsilon\kappa}}}$,
with $\varepsilon\downarrow0$ and $\kappa$ a positive measure. Perturbation
theory (Cf for example \cite{Kato}) shows that $\beta^{(\varepsilon\kappa
)}-1=\frac{1}{\lambda(X)}\sum_{x}\frac{\lambda_{x}}{1+\varepsilon\kappa_{x}%
}-1+o(\varepsilon)=-\frac{\varepsilon\kappa(X)}{\lambda(X)}+o(\varepsilon)$
and that $h^{(\varepsilon k)}=\frac{1}{\sqrt{\lambda(X)}}+o(\varepsilon)$.

We deduce from that the asymptotic behaviour of
\[
\int(e^{-\varepsilon\left\langle \widehat{l},\chi\right\rangle }%
-1)d\mu^{(\varepsilon\kappa)}(l)=\log(\det(I-P^{(\varepsilon\kappa)}%
))-\log(\det(I-P^{(\varepsilon(\kappa+\chi))}))
\]
which is equivalent to $-\log(1-\beta^{(\varepsilon(\kappa+\chi))})$
$+\log(1-\beta^{(\varepsilon\kappa)})$ and therefore to $\log(\frac{\kappa
(X)}{\kappa(X)+\chi(X)})$.

The study of relations between the loop measure $\mu$ and the zeta function
${(\det(I-sP))^{-1}}$ and more generally $(\det(I-M_{f}P))^{-1}$ with $f$ a
function on $[0,1]$ can be done in the context of discrete loops.
\[
\exp\Big(          \sum_{\substack{\text{based}\\\text{discrete loops}%
}}\frac{1}{p(\xi)}s^{p(\xi)}\mu(\xi)\Big)          =(\det(I-sP))^{-1}%
\]
can be viewed as a type of zeta function defined for $s\in\lbrack0\;1/\beta)$

Primitive non trivial (based) discrete loops are defined as discrete based
loops which cannot be obtained by the concatenation of $n\geq2$ identical
based loops. Loops are primitive iff they are classes of primitive based loops.

The zeta function has an Euler product expansion: if we denote by $\xi^{\circ
}$ this discrete loop defined by the based discrete loop $\xi$, and set, for
$\xi=(\xi_{1},...,\xi_{k})$, $\mu(\xi^{\circ})=P_{\xi_{2}}^{\xi_{1}}P_{\xi
_{3}}^{\xi_{2}}....P_{\xi_{1}}^{\xi_{k}}$, it can be seen, by taking the logarithm, that:
\[
(\det(I-sP))^{-1}=\exp\Big(          \sum_{_{_{\substack{\text{based}%
\\\text{discrete loops}}}}}\frac{1}{p(\xi)}s^{p(\xi)}\mu(\xi)\Big)
=\prod_{_{\substack{\text{primitive}\\\text{discrete loops}}}}\Big(     1-\int
s^{p(\xi^{%
{{}^\circ}%
})}\mu(\xi^{%
{{}^\circ}%
})\Big)          ^{-1}%
\]

\chapter{Geodesic loops}

\section{Reduction\label{secred}}

\index{reduced path}Given any finite path $\omega$ with starting point $x_{0}$, the reduced path
$\omega^{R}$ is defined as the geodesic arc defined by the endpoint of the
lift of $\omega$ to $\mathfrak{T}_{x_{0}}$.

Tree-contour-like based loops can be defined as discrete based loops whose
lift to the universal covering are still based loops. Each link is followed
the same number of times in opposite directions (backtracking). The reduced
path $\omega^{R}$\ can equivalently be obtained by removing all
tree-contour-like based loops imbedded into it. In particular each loop $l$
based at $x_{0}$ defines an element $l^{R}$ in $\Gamma_{x_{0}}$.%

\begin{center}
\includegraphics[
height=3.0061in,
width=4.7046in
]%
{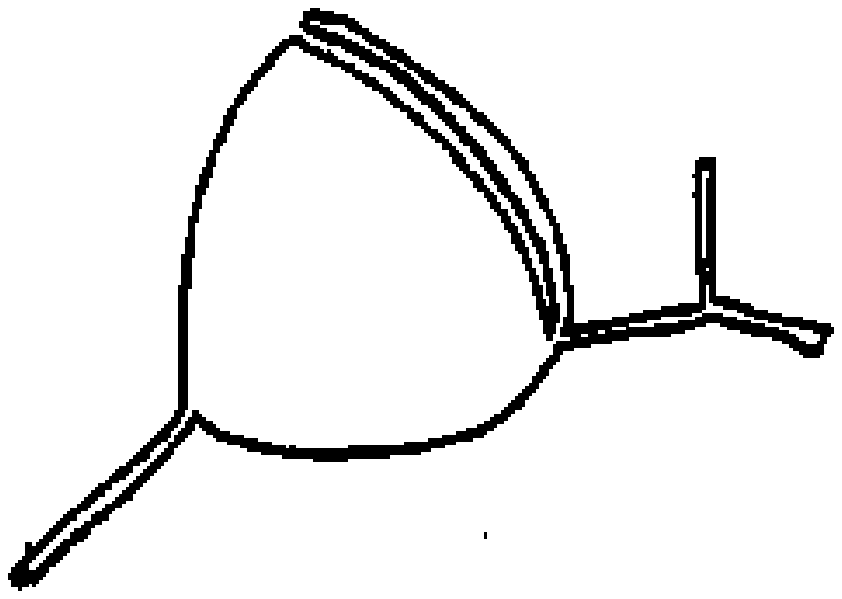}%
\\
Based loop
\end{center}

\index{loop erasure}This procedure is an example of \emph{loop erasure}. In any graph, given a
path $\omega$, the loop erased path $\omega^{LE}$\ is defined by removing
progressively all based loops imbedded in the path, starting from \ the
origin. It produces a self avoiding path (and we see geodesics in
$\mathfrak{T}_{x_{0}}$ are self avoiding paths). Hence any non ramified
covering defines a specific reduction operation by composition of lift, loop
erasure, and projection.%

\begin{center}
\includegraphics[
height=2.6429in,
width=3.8631in
]%
{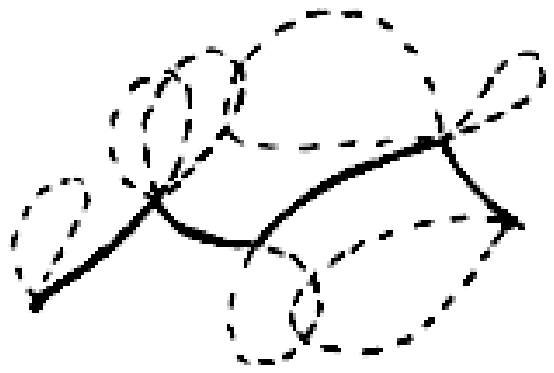}%
\\
Loop erasure
\end{center}

\section{Geodesic loops and conjugacy classes}

Then, we can consider loops i.e. equivalence classes of based loops under the
natural shift.

Geodesic loops are of particular interest. Note their based loops
representatives have to be ''tailess'': If $\gamma$ is a geodesic based loop,
with $\left|  \gamma\right|  =n$, the tail of $\gamma$ is defined as
$\gamma_{1}\gamma_{2}...\gamma_{i}\gamma_{i-1}...\gamma_{1}$ if $i=\sup(j,\gamma_{1}\gamma_{2}%
...\gamma_{j}=\gamma_{n}\gamma_{n-1}...\gamma_{n-j+1})$. The associated
geodesic loop is obtained by removing the tail.

The geodesic loops are clearly in bijection with the set of conjugacy classes
of the fundamental group. Indeed, if we fix a reference point $x_{0}$, a geodesic
loop defines the conjugation class formed of the elements of $\Gamma_{x_{0}}$
obtained by choosing a base point on the loop and a geodesic segment linking
it to $x_{0}$. Any non trivial element of $\Gamma_{x_{0}}$ can be obtained in
this way.

Given a loop, there is a canonical geodesic loop associated with it. It is
obtained by removing all tails imbedded in it. It can be done by removing one
by one all tail edges (i.e. pairs of consecutive inverse oriented edges of the
loop). Note that after removal of a tail edge, another tail edge cannot
disappear, and that new tail edges appear during this process.%

\begin{center}
\includegraphics[
height=5.2157in,
width=3.7663in
]%
{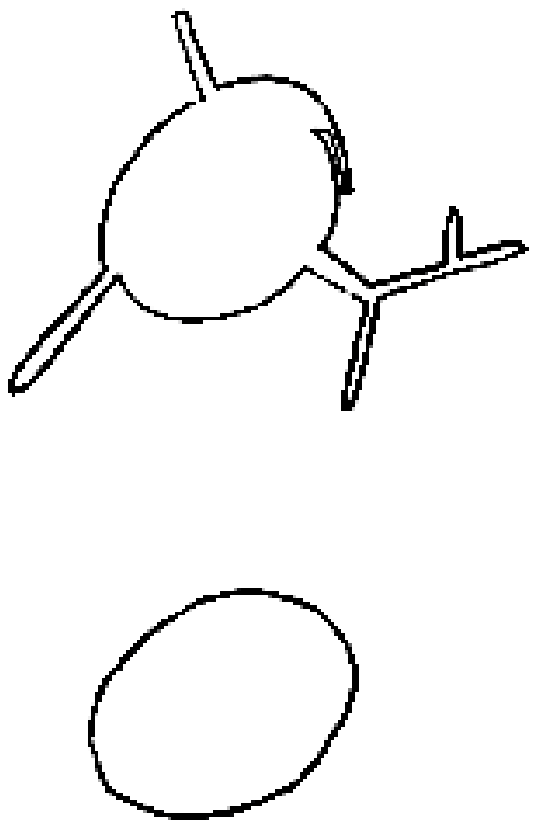}%
\\
Loop and associated geodesic loop
\end{center}

A closed geodesic based at $x_{0}$ is called primitive if it cannot be
obtained as the concatenation of several identical closed geodesic, i.e. if it
is not a non trivial power in $\Gamma_{x_{0}}$. This property is clearly
stable under conjugation. Let $\mathfrak{P}$\ be corresponding set of
primitive geodesic loops. They represent conjugacy classes of primitive
elements of $\Gamma$ (see \cite{Starter}).

\section{Geodesics and boundary}

Geodesics lines (half-lines) on a graph are defined as paths without
backtracking indexed by $\mathbb{Z}$ ($\mathbb{N}$).

Paths and in particular geodesics can be defined on $(X,E)$ or on a universal
cover $\mathfrak{T}$ and lifted or projected on any intermediate covering
space. Two geodesic half lines are said to be confluent if their intersection
is a half line.

Let us now take the point of view described in remark \ref{intrinsic}%
. Equivalence classes of geodesics half lines of $\mathfrak{T}$ for the
confluence relation define the boundary $\partial\mathfrak{T}$ of
$\mathfrak{T}$. A geodesic half-line on $\mathfrak{T}$\ can therefore be
defined by two points: its origin $Or$\ and the boundary point $\theta
$\ towards which it converges. It projects on a geodesic half-line on
$(X,E)=\Gamma\backslash\mathfrak{T}$. The set of geodesic half lines on
$(X,E)$\ is identified with $\Gamma\backslash(\mathfrak{T}\times$
$\partial\mathfrak{T})$ which projects canonically onto $X$.

There is a natural $\sigma$-field on the boundary generated by cylinder sets
$B_{g}$\ defined by half geodesics starting with a given oriented edge $g$.

Given any point $x_{0}$ in $\mathfrak{T}$, \emph{assuming in this subsection
that }$\kappa=0$, one can define a probability measure on the boundary called
the harmonic measure and denoted $\nu_{x_{0}}$: $\nu_{x_{0}}(B_{g})$ is the
probability that the lift of the $P$-Markov chain starting at $x_{0}$ hits
$g^{+}$ after its last visit to $g^{-}$.

Note that $\Gamma$ acts on the boundary in such a way that $\gamma^{\ast}%
(\nu_{x})=\nu_{\gamma(x)}$, for all $\gamma\in\Gamma$ and $x\in\mathfrak{T}$.
This harmonic measure induces a probability on the fiber above $\Gamma x$ in
$\Gamma\backslash(\mathfrak{T}\times\partial\mathfrak{T)}$, i.e. on half
geodesics starting at the projection of $x$ on $X$.

Clearly, in the case of a regular graph, as the universal covering is a
$r-$regular tree, $\nu_{x_{0}}(B_{g})=\frac{1}{r}\frac{1}{(r-1)^{d(x_{0},g-)}%
}$ where $d$ denotes the distance in the tree. When conductances are all
equal, $\nu_{x_{0}}(B_{g})$ can also be computed but is in general distinct
from the visibility measure from $x_{0}$, $\nu_{x_{0}}^{vis}(B_{g})$, defined
as $\frac{1}{d_{x_{0}}}\prod\frac{1}{d_{x_{i}}-1}$, $x_{1},x_{2},...x_{i},...$
being the points of the geodesic segment linking $x_{0}$ to $g^{-}$.
$\nu_{x_{0}}^{vis}$ is also a probability on $\partial\mathfrak{T}$.

There is an obvious canonical shift acting on half geodesics.

Note also that $\sum_{x\in\mathfrak{T}}\delta_{x}^{Or}\nu_{x}^{vis}(d\theta)$
is a shift-invariant and $\Gamma$-invariant measure on the set of
half-geodesics of $\mathfrak{T}$.

It can be shown it induces a canonical shift invariant and $\Gamma$-invariant
probability\ on half geodesics on $X$ obtained by restricting the sum to any
fundamental domain and normalizing by $\left|  X\right|  $. It is independent
of the choice of the domain.

\section{Closed geodesics and associated Zeta function}

Recall that $\mathfrak{P}$\ denotes the set of primitive geodesic loops.

Ihara's zeta function $IZ(u)$ is defined for $0\leq u<1$\ as
\[
IZ(u)=\prod_{\gamma\in\mathfrak{P}}(1-u^{p(\gamma)})^{-1}%
\]
It depends only on the graph.

Note that
\[
u\frac{\frac{d}{du}IZ(u)}{IZ(u)}=\sum_{\gamma\in\mathfrak{P}}\frac{p(\gamma
)u^{p(\gamma)}}{1-u^{p(\gamma)}}=\sum_{\gamma\in\mathfrak{P}}\sum
_{n=1}^{\infty}p(\gamma)u^{np(\gamma)}=\sum_{m=2}^{\infty}N_{m}u^{m}%
\]
where $N_{m}$ denotes the number of \emph{tailess} geodesic based loops of
length $m$. Indeed, each primitive geodesic loop $\gamma$\ traversed $n$ times
still induces $p(\gamma)$ distinct tailess geodesic based loops. Therefore
$IZ(u)$ can also be written\ as $\exp\big(          \sum_{m=2}^{\infty
}\frac{N_{m}u^{m}}{m}\big)          .$

Similarly, one can define $\Pi\Gamma$ to be the set of primitive elements of
the fundamental group $\Gamma$ and the $\Gamma$-zeta function to be:%
\[
\Gamma Z(u)=\prod_{\gamma\in\Pi\Gamma}(1-u^{p(\gamma)})^{-1}.
\]
Note that $u\frac{\frac{d}{du}\Gamma Z(u)}{\Gamma Z(u)}=\sum_{2}^{\infty}%
L_{m}u^{m}$ where $L_{m}$ denotes the number of geodesic based loops of length
$m$. $\Gamma Z(u)$ can also be written\ as $\exp(\sum_{m=2}^{\infty
}\frac{L_{m}u^{m}}{m})$. Recall that $A$ denotes the adjacency matrix of the
graph, and $D$ the diagonal matrix whose entries are given by the degrees of
the vertices.

Assume now that $0<u<\inf(\frac{1}{d_{x}-1},x\in X)$. We will use again the
Green function associated with $\lambda_{x}=(d_{x}-1)u+\frac{1}{u}$\ and, for
$\{x,y\}\in E$, $C_{x,y}=1$.

\begin{theorem}
\begin{enumerate}
\item[a)] $\sum_{2}^{\infty}L_{m}u^{m}=(1-u^{2})Tr([I+(D-I)u^{2}%
-uA]^{-1})-\left|  X\right|  .$

\item[b)] $IZ(u)=(1-u^{2})^{-\chi}\det(I-uA+u^{2}(D-I))^{-1}$ where $\chi=$
denotes the Euler number $\left|  E\right|  -\left|  X\right|  $ of the graph.
\end{enumerate}
\end{theorem}

\begin{proof}
We adapt the approach of Stark-Terras (\cite{Starter})

\begin{enumerate}
\item[a)] As geodesic loops based in $x_{0}$ are in bijection with
$\Gamma_{x_{0}}$, it follows from proposition \ref{greencov}\ and
\ref{greenuniv}\ that

$\left|  V\right|  +\sum_{2}^{\infty}L_{m}u^{m}=(\frac{1}{u}-u)Tr(G)=(1-u^{2}%
)Tr([I+(D-I)u^{2}-uA]^{-1})$

\item[b)] Given a geodesic loop $l$ (possibly empty) and a base point $y$ of
$l$, let $S_{x,y,l}$ be the sum of the coefficients $u^{p(\delta)}$, where
$\delta$ varies on all geodesic loops based at $x$ composed with $l$ and a
tail ending at $y$. If $x=y$, we have $S_{y,y,l}=u^{p(l)}$. Set $S_{x,l}%
=\sum_{y}S_{x,y,l}$.

Clearly, for any section $i$ of the canonical projection from $\mathfrak{T}%
_{x}$ onto $X$,
\[
\sum_{y,l}S_{x,y,l}=\sum_{\gamma\in\Gamma_{x}-\{I\}}u^{d(i(x),\gamma
i(x))}=(\frac{1}{u}-u)G^{x,x}-1.
\]
On the other hand, considering first the tailess case, then the case where the
tail has length $1$, and finally decomposing the case where the tail has
length at least two according to the position of the point of the tail next to
$x$, (denoted $x^{\prime}$), we obtain the expression:
\begin{align*}
\sum_{x}S_{x,y,l}  &  =u^{p(l)}+(d_{y}-2)u^{p(l)+2}+\sum_{x^{\prime}\neq
y}(d_{x^{\prime}}-1)u^{2}S_{x^{\prime},y,l}\\
&  =u^{p(l)}-u^{p(l)+2}+\sum_{x}(d_{x}-1)u^{2}S_{x,y,l}%
\end{align*}
summing in $y$, it comes that
\[
\sum_{x}S_{x,l}=p(l)(u^{p(l)}-u^{p(l)+2})+\sum_{x}(d_{x}-1)u^{2}S_{x,l}.
\]
Then, summing on all geodesic loops $l$
\[
\sum_{x}((\frac{1}{u}-u)G^{x,x}-1)=(1-u^{2})(\sum N_{m}u^{m})+\sum_{x}%
(d_{x}-1)u^{2}((\frac{1}{u}-u)G^{x,x}-1).
\]
Therefore,
\begin{align*}
\sum N_{m}u^{m}=  &  Tr((I-u^{2}(D-I))([I+(D-I)u^{2}-uA]^{-1}-(1-u^{2}%
)^{-1}I))\\
=  &  Tr(I-(2u^{2}(D-I)-uA)[I+(D-I)u^{2}-uA]^{-1}\\
&  -(1-u^{2})^{-1}(I-u^{2}(D-I))\\
=  &  Tr((2u^{2}(D-I)-uA)[I+(D-I)u^{2}-uA]^{-1}+(1-u^{2})^{-1}(u^{2}(D-2I)).
\end{align*}
To conclude note that
\[
\frac{d}{du}\log(\det(I+(D-I)u^{2}-uA))=Tr((2u(D-I)-A)[I+(D-I)u^{2}-uA]^{-1})
\]
and that $u^{2}Tr(D-2I)=2u^{2}\chi$.
\end{enumerate}
\end{proof}

\paragraph{An alternative proof}

Other proofs can be found in the litterature, especially the following one due
to Kotani-Sunada (\cite{Kotasun}):\newline On the line\ graph, we define a
transfer operator $Q$ by $Q_{(x,y)}^{(y^{\prime},z)}=\delta_{y}^{y^{\prime}%
}1_{\{z\neq x\}}$. Then, as $\log(\det((I-uQ)^{-1})=\sum\frac{u^{n}}%
{n}Tr(Q^{n})$ and $Tr(Q^{n})=N_{n}$, we have%
\[
IZ(u)=\det((I-uQ)^{-1}%
\]
Define the linear map $T$, from $\mathbb{A}$ to functions on $X$ by
$T\alpha(x)=\sum_{y,\{x,y\}\in E}\alpha(x,y)$. Define a linear transformation
$\tau$ on $\mathbb{A}$ by $\tau\alpha(e)=\alpha(-e)$. Define $S$ the linear
map from functions on $X$ to $\mathbb{A}$\ defined by $Sf(x,y)=f(y)$. Note
that $T\tau S=D$, $TS=A$, and $Q=-\tau+ST$.\newline Then, for any scalar $u$,
$(I-u\tau)(I-uQ)=(1-u^{2})I-(I-u\tau)uST$ and \
\begin{equation}
(I-uQ)(I-u\tau)=(1-u^{2})I-ST(I-u\tau). \label{dec1}%
\end{equation}
Therefore $T(I-u\tau)(I-uQ)=((1-u^{2})T-uT(I-u\tau)ST)=(I+u^{2}(D-I)-uA)T$ and%
\[
T(I-u\tau)(I-uQ)(I-u\tau)=(I+u^{2}(D-I)-uA)T(I-u\tau).
\]
Moreover $(I-uQ)(I-u\tau)S=S((1-u^{2})I-uT(I-u\tau)S)$ and%
\begin{equation}
(I-uQ)(I-u\tau)S=S(I+u^{2}(D-I)-uA) . \label{dec2}%
\end{equation}
It follows from these two last identities that $\operatorname{Im}(S)$ and
$Ker(T(I-u\tau))$ are stable under $(I-uQ)(I-u\tau)$.

Note that $S$ is the dual of $-T\tau$: Indeed, for any function $f$ on
vertices and $\alpha$ on oriented edges,
\[
\sum_{(x,y)\in E^{O}}\alpha(x,y)Sf(x,y)=\sum_{(x,y)\in E^{O}}\alpha
(x,y)f(y)=\sum_{y}T\tau\alpha(y)f(y).
\]
Therefore, $\dim(\operatorname{Im}(S))+\dim(Ker(T\tau))=2\left|  E\right|  .$

Note also that $\dim(Ker(T\tau))=\dim(Ker(T))=\dim(Ker(T(I-u\tau)))$ (as $u<1$).

Moreover, except for a finite set of $u$'s, $\operatorname{Im}(S)\cap
Ker(T(I-u\tau))=\{0\}$.\newline Indeed $T(I-u\tau)S=A-uD$ which is invertible,
except for a finite set of $u$'s.

Note that (\ref{dec1}) implies that $(I-uQ)(I-u\tau)$ equals $(1-u^{2})I$ on
$Ker(T(I-u\tau)$ and that (\ref{dec2}) implies it equals $S(I+u^{2}%
(D-I)-uA)S^{-1}$ on $\operatorname{Im}(S)$.

It comes that:%
\[
\det((I-u\tau)(I-uQ))=(1-u^{2})^{2\left|  E\right|  -\left|  X\right|  }%
\det(I+u^{2}(D-I)-uA)
\]

On the other hand, $\det((I-u\tau))=(1-u^{2})^{\left|  E\right|  }$, which
allows to conclude.

\chapter{Poisson process of loops}

\section{Definition}

\index{Poissonian loop ensembles}Still following the idea of \cite{LW}, which was already implicitly in germ in
\cite{Symanz}, define, for all positive $\alpha$,\ the Poissonian ensemble of
loops $\mathcal{L}_{\alpha}$ with intensity $\alpha\mu$.

Note also that these Poissonian ensembles can be considered for fixed $\alpha$ or
as a point process of loops indexed by the ''time'' $\alpha$. In that case,
$\mathcal{L}_{\alpha}$ is an increasing set of loops with stationnary
increments. We will denote by $\mathcal{LP}$ the associated Poisson point
process of intensity $\mu(dl)\otimes Leb(d\alpha)$ ($Leb$ denoting Lebesgue
measure on the positive half-line). It is formed by a countable set of pairs
$(l_{i},\alpha_{i})$\ formed by a loop and a time.

We denote by $\mathbb{P}$ its distribution.

Recall that for any functional $\Phi$ on the loop space, vanishing on loops of
arbitrary small length,
\[
\mathbb{E}(e^{i\sum_{l\in\mathcal{L}_{\alpha}}\Phi(l)})=\exp(\alpha
\int(e^{i\Phi(l)}-1)\mu(dl))
\]
and for any positive functional $\Psi$ on the loops space,%
\begin{equation}
\mathbb{E}(e^{-\sum_{l\in\mathcal{L}_{\alpha}}\Psi(l)})=\exp(\alpha
\int(e^{-\Psi(l)}-1)\mu(dl)) \label{Poilapl}%
\end{equation}
It follows that if $\Phi$ is $\mu$-integrable, $\sum_{l\in\mathcal{L}_{\alpha
}}\Phi(l)$ is integrable and
\[
\mathbb{E}(\sum_{l\in\mathcal{L}_{\alpha}}\Phi(l))=\int\Phi(l)\alpha\mu(dl).
\]
And if in addition $\Phi^{2}$ is $\mu$-integrable, $\sum_{l\in\mathcal{L}_{\alpha}}%
\Phi(l)$ is square-integrable and
\[
\mathbb{E(}(\sum_{l\in\mathcal{L}_{\alpha}}\Phi(l)))^{2}=\int\Phi^{2}%
(l)\alpha\mu(dl)+(\int\Phi(l)\alpha\mu(dl))^{2}.
\]

Recall also ''Campbell formula'' (Cf formula 3-13 in \cite{King}): For any
system of non negative or $\mu$-integrable\ loop functionals $F_{i}$,%

\begin{equation}
\mathbb{E}\Big(  \sum_{l_{1}\neq l_{2}...\neq l_{k}\in\mathcal{L}_{\alpha}%
}\prod F_{i}(l_{i})\Big)  =\prod_{1}^{k}\alpha\mu(F_{i}) \label{campb}%
\end{equation}

Note the same results hold for functionals of $\mathcal{LP}$.

Of course, $\mathcal{L}_{\alpha}$ includes trivial loops. The periods
$\tau_{l}$\ of the trivial loops based at any point $x$ form a Poisson process
of intensity $\alpha\frac{e^{-t}}{t}$. It follows directly from this (
\cite{Pitm} and references therein) that we have the following

\begin{proposition}
The sum of these periods $\sum\tau_{l}$ and the set of ''frequencies''
$\frac{\tau_{l}}{\sum\tau_{l}}$ (in decreasing order) are independent and
follow repectively a $\Gamma(\alpha)$ and a $Poisson-Dirichlet(0,\alpha)$ distribution.\bigskip
\end{proposition}

Note that by the restriction property, $\mathcal{L}_{\alpha}^{D}%
=\{l\in\mathcal{L}_{\alpha},l\subseteq D\}$\ is a Poisson process of loops
with intensity $\mu^{D}$, and that $\mathcal{L}_{\alpha}^{D}$ \ is independent
of $\mathcal{L}_{\alpha}\backslash\mathcal{L}_{\alpha}^{D}$.

We denote by $\mathcal{DL}_{\alpha}$ the set of non trivial discrete loops in
$\mathcal{L}_{\alpha}$. Then, $\mathbb{P(}\#\mathcal{DL}_{\alpha
}=k)=e^{-\alpha\mu(p>1)}\frac{\mu(p>1)^{k}}{k!}$ and conditionally to their
number, the discrete loops are independently sampled according to
$\frac{1}{\mu(p>1)}\mu1_{\{p>1\}}$. In particular, if $l_{1},l_{2},...,l_{k}$
are distinct discrete loops%
\begin{multline*}
\mathbb{P(}\mathcal{DL}_{\alpha}=\{l_{1},l_{2},...,l_{k}\})=e^{-\alpha
\mu(p>1)}\alpha^{k}\mu(l_{1})...\mu(l_{k})\\
=\alpha^{k}[\frac{\det(G)}{\prod_{x}\lambda_{x}}]^{\alpha}\prod_{x,y}%
C_{x,y}^{\sum_{1}^{k}N_{x,y}(l_{i})}\prod_{x}\lambda_{x}^{-\sum_{1}^{k}%
N_{x}(l_{i})}.
\end{multline*}

The general result (when the $l_{i}$'s are not necessarily distinct) follows
from the multinomial distibution.

\index{occupation field, local times}We can associate to $\mathcal{L}_{\alpha}$ a $\sigma$-finite measure (in fact
as we will see, it is finite when $X$ is finite, and more generally if $G$ is
trace class) called local time or occupation field
\[
\widehat{\mathcal{L}_{\alpha}}=\sum_{l\in\mathcal{L}_{\alpha}}\widehat{l}.
\]

Then, for any non-negative measure $\chi$ on $X$
\[
\mathbb{E}(e^{-\left\langle \widehat{\mathcal{L}_{\alpha}},\chi\right\rangle
})=\exp\Big(          \alpha\int(e^{-\left\langle \widehat{l},\chi
\right\rangle }-1)d\mu(l)\Big)          .
\]
and therefore by proposition \ref{LAPLT} we have

\begin{corollary}
\label{lapl} $\mathbb{E}(e^{-\left\langle \widehat{\mathcal{L}_{\alpha}}%
,\chi\right\rangle })=\det(I+M_{\sqrt{\chi}}GM_{\sqrt{\chi}})^{-\alpha
}=(\frac{\det(G_{\chi})}{\det(G)})^{\alpha}.$
\end{corollary}

\bigskip

Many calculations follow from this result.

Note first that $\mathbb{E}(e^{-t\widehat{\mathcal{L}_{\alpha}}^{x}%
})=(1+tG^{x,x})^{-\alpha}$. Therefore $\widehat{\mathcal{L}_{\alpha}}^{x}$
follows a gamma distribution $\Gamma(\alpha,G^{x,x})$, with density
$1_{\{x>0\}}\frac{e^{-\frac{x}{G^{xx}}}}{\Gamma(\alpha)}\frac{x^{\alpha-1}%
}{(G^{xx})^{\alpha}}$ (in particular, an exponential distribution of mean
$G^{x,x}$\ for $\alpha=1$, as $\widehat{l}^{x}$ under $\frac{\mu^{y,x}%
}{G^{y,x}}$). When we let $\alpha$ vary as a time parameter, we get a family
of gamma subordinators, which can be called a ''multivariate gamma
subordinator''\footnote{A subordinator is an increasing Levy process. See for
example reference \cite{Bert}.}.

We check in particular that $\mathbb{E}(\widehat{\mathcal{L}_{\alpha}}%
^{x})=\alpha G^{x,x}$ which follows directly from $\mu(\widehat{l}%
_{x})=G^{x,x}$.

\begin{exercise}
\label{poidir}If $\mathcal{L}_{\alpha}=\{l_{i}\}$, check that the set of
''frequencies'' $\frac{\widehat{l}_{i}^{x}}{\widehat{\mathcal{L}_{\alpha}}%
^{x}}$ follows a Poisson-Dirichlet distribution of parameters $(0,\alpha)$.
\end{exercise}

Hint: use the $\mu$-distribution of $\widehat{l}^{x}$.\medskip\newline Note
also that for $\alpha>1$,
\[
{\mathbb{E}((1-\exp(-\frac{\widehat{\mathcal{L}_{\alpha}}^{x}}{G^{x,x}}%
))^{-1})=\zeta(\alpha)}.
\]
For two points, it follows easily from corollary \ref{lapl} that:%
\[
\mathbb{E}(e^{-t\widehat{\mathcal{L}_{\alpha}}^{x}}e^{-s\widehat
{\mathcal{L}_{\alpha}}^{y}})=((1+tG^{x,x})(1+sG^{y,y})-st(G^{x,y}%
)^{2})^{-\alpha}%
\]
This allows to compute the joint density of $\widehat{\mathcal{L}_{\alpha}%
}^{x}$ and $\widehat{\mathcal{L}_{\alpha}}^{y}$ in terms of Bessel and Struve functions.

\bigskip

We can condition the loop configuration on the set of associated non trivial
discrete loops by using the restricted $\sigma$-field $\sigma(\mathcal{DL}%
_{\alpha})$\ which contains the variables $N_{x,y}$.\ We see from
\eqref{triv}  and \eqref{d}  \ that%

\[
\mathbb{E}\Big(          e^{-\left\langle \widehat{\mathcal{L}_{\alpha}}%
,\chi\right\rangle }|\mathcal{DL}_{\alpha}\Big)          =\prod_{x}%
(\frac{\lambda_{x}}{\lambda_{x}+\chi_{x}})^{N_{x}^{(\alpha)}+1}%
\]
The distribution of $\{N_{x}^{(\alpha)},x\in X\}$ follows easily, from
corollary \ref{lapl}\ in terms of generating functions:
\begin{equation}
\mathbb{E}(\prod_{x}s_{x}^{N_{x}^{(\alpha)}+1})=\det(\delta_{x,y}%
+\sqrt{\frac{\lambda_{x}(1-s_{x})}{s_{x}}}G_{x,y}\sqrt{\frac{\lambda
_{y}(1-s_{y})}{s_{y}}})^{-\alpha} \label{GF}%
\end{equation}
so that the vector of components $N_{x}^{(\alpha)}$ follows a multivariate
negative binomial distribution (see for example \cite{VJ2}).

It follows in particular that $N_{x}^{(\alpha)}$\ follows a negative binomial
distribution of parameters $-\alpha$ and $\frac{1}{\lambda_{x}G^{xx}}$. Note
that for $\alpha=1$, $N_{x}^{(1)}+1$ follows a geometric distribution of
parameter $\frac{1}{\lambda_{x}G^{xx}}$.

Note finally that in the recurrent case, with the setting and the notations of
subsection \ref{pres}, denoting $\mathcal{L}_{\alpha\varepsilon}%
^{(\varepsilon\kappa)}$ the Poisson process of loops of intensity
$\varepsilon\alpha\mu^{(\varepsilon\kappa)}$, we get that the associated
occupation field converges in distribution towards a random constant following
a Gamma distribution.

Let us recall one important property of Poisson processes.

\begin{proposition}
\label{camca}Given any bounded functional $\Phi$ on loops configurations and any integrable loop functional $F$, we have:%
\[
\mathbb{E}(\sum_{l\in\mathcal{L}_{\alpha}}F(l)\Phi(\mathcal{L}_{\alpha}%
))=\int\mathbb{E}(\Phi(\mathcal{L}_{\alpha}\cup\{l\}))\alpha F(l)\mu(dl).
\]
\end{proposition}

\begin{proof}
This is proved by considering first for $\Phi(\mathcal{L}_{\alpha})$ the
functionals of the form $\sum_{l_{1}\neq l_{2}...\neq l_{q}\in\mathcal{L}%
_{\alpha}}\prod_{1}^{q}G_{j}(l_{j}))$ (with $G_{j}$ bounded and $\mu
$-integrable) which span an algebra separating distinct configurations and
applying formula \eqref{campb}  : Then, the common value of both members is
$\alpha^{q}\sum_{1}^{q}\mu(FG_{j})\prod_{l\neq j}\mu(G_{l})+\alpha^{q+1}%
\mu(F)\prod_{1}^{q}\mu(G_{j})$
\end{proof}

\begin{exercise}
Give an alternative proof of this proposition using formula \eqref{Poilapl}  .
\end{exercise}

The above proposition applied to $F(l)=\widehat{l}^{x}$,$N_{x,y}^{(\alpha)}$
and propositions \ref{bridg} and \ref{bridg2}\ yield the following:

\begin{corollary}
\label{camcam}%
\[
\mathbb{E}(\Phi(\mathcal{L}_{\alpha})\widehat{\mathcal{L}_{\alpha}}%
^{x})=\alpha\int\mathbb{E}(\Phi(\mathcal{L}_{\alpha}\cup\{\gamma
\}))\widehat{\gamma}^{x}\mu(d\gamma)=\alpha\int\mathbb{E}(\Phi(\mathcal{L}%
_{\alpha}\cup\{\gamma\}))\mu^{x,x}(d\gamma)
\]
and if $x\neq y$
\[
\mathbb{E}(\Phi(\mathcal{L}_{\alpha})N_{x,y}^{(\alpha)})=\alpha\int
\mathbb{E}(\Phi(\mathcal{L}_{\alpha}\cup\{\gamma\}))N_{x,y}(\gamma)\mu
(d\gamma)=\alpha C_{x,y}\int\mathbb{E}(\Phi(\mathcal{L}_{\alpha}\cup
\{\gamma\}))\mu^{x,y}(d\gamma).
\]
\end{corollary}

\begin{remark}
Proposition \ref{camca} and corollary \ref{camcam}\ can be easily generalized
to functionals of the Poisson process $\mathcal{LP}$.
\end{remark}

\begin{exercise}
Generalize corollary \ref{camcam} to $\widehat{\mathcal{L}_{\alpha}}%
^{x}\widehat{\mathcal{L}_{\alpha}}^{y}$, for $x\neq y$.
\end{exercise}

\section{Moments and polynomials of the occupation field\label{momts}}

It is easy to check (and well known from the properties of the gamma
distributions) that the moments of $\widehat{\mathcal{L}_{\alpha}}^{x}$ are
related to the factorial\ moments of $N_{x}^{(\alpha)}$\ :%
\[
\mathbb{E}((\widehat{\mathcal{L}_{\alpha}}^{x})^{k}|\mathcal{DL}_{\alpha
})=\frac{(N_{x}^{(\alpha)}+k)(N_{x}^{(\alpha)}+k-1)...(N_{x}^{(\alpha)}%
+1)}{k!\lambda_{x}^{k}}%
\]

\begin{exercise}
Denoting $\mathcal{L}_{\alpha}^{+}$ the set of non trivial loops in
$\mathcal{L}_{\alpha}$, define
\[
\widehat{\mathcal{L}_{\alpha}}^{x,(k)}=\sum_{m=1}^{k}\sum_{k_{1}+...+k_{m}%
=k}\sum_{l_{1}\neq l_{2}...\neq l_{m}\in\mathcal{L}_{\alpha}^{+}}\prod
_{j=1}^{m}\widehat{l_{j}}^{x,(k_{j})}.
\]
Deduce from exercise \ref{self} that $\mathbb{E}(\widehat{\mathcal{L}_{\alpha
}}^{x,(k)}|\mathcal{DL}_{\alpha})=\frac{1}{k!\lambda_{x}^{k}}1_{\{N_{x}\geq
k\}}(N_{x}^{(\alpha)}-k+1)...(N_{x}^{(\alpha)}-1)N_{x}^{(\alpha)}$
\end{exercise}

\bigskip

\index{Laguerre polynomials}It is well known that Laguerre polynomials $L_{k}^{(\alpha-1)}$ with
generating function
\[
\sum_{0}^{\infty}t^{k}L_{k}^{(\alpha-1)}(u)=\frac{e^{-\frac{ut}{1-t}}%
}{(1-t)^{\alpha}}%
\]
are orthogonal for the $\Gamma(\alpha)$ distribution. They have mean zero and
variance $\frac{\Gamma(\alpha+k)}{k!}$. Hence if we set $\sigma_{x}=G^{x,x}%
$and $P_{k}^{\alpha,\sigma}(x)=(-\sigma)^{k}L_{k}^{(\alpha-1)}(\frac{x}%
{\sigma})$, the random variables $P_{k}^{\alpha,\sigma_{x}}(\widehat
{\mathcal{L}_{\alpha}}^{x})$ are orthogonal with mean $0$\ and variance
$\sigma^{2k}\frac{\Gamma(\alpha+k)}{k!}$, for $k>0$.

Note that $P_{1}^{\alpha,\sigma_{x}}(\widehat{\mathcal{L}_{\alpha}}%
^{x})=\widehat{\mathcal{L}_{\alpha}}^{x}-\alpha\sigma_{x}=\widehat
{\mathcal{L}_{\alpha}}^{x}-\mathbb{E}(\widehat{\mathcal{L}_{\alpha}}^{x})$.
\emph{It will be denoted }$\widetilde{\mathcal{L}_{\alpha}}^{x}$.

Moreover, we have $\sum_{0}^{\infty}t^{k}P_{k}^{\alpha,\sigma}(u)=\sum(-\sigma
t)^{k}L_{k}^{(\alpha-1)}(\frac{u}{\sigma})=\frac{e^{\frac{ut}{1+\sigma t}}%
}{(1+\sigma t)^{\alpha}}$

Note that by corollary \ref{lapl},
\begin{multline*}
\mathbb{E}(\frac{e^{\frac{\widehat{\mathcal{L}_{\alpha}}^{x}t}{1+\sigma_{x}t}%
}}{(1+\sigma_{x}t)^{\alpha}}\frac{e^{\frac{\widehat{\mathcal{L}_{\alpha}}%
^{y}s}{1+\sigma_{y}s}}}{(1+\sigma_{y}s)^{\alpha}})\\
=\frac{1}{(1+\sigma_{x}t)^{\alpha}(1+\sigma_{y}s)^{\alpha}}((1-\frac{\sigma
_{x}t}{1+\sigma_{x}t})(1-\frac{\sigma_{y}s}{1+\sigma_{y}s})-\frac{t}%
{1+\sigma_{x}t}\frac{s}{1+\sigma_{y}s}(G^{x,y})^{2})^{-\alpha}\\
=(1-st(G^{x,y})^{2})^{-\alpha}.
\end{multline*}
Therefore, we get, by developping in entire series in $(s,t)$ and identifying
the coefficients:
\begin{equation}
\mathbb{E}(P_{k}^{\alpha,\sigma_{x}}(\widehat{\mathcal{L}_{\alpha}}^{x}%
),P_{l}^{\alpha,\sigma_{y}}(\widehat{\mathcal{L}_{\alpha}}^{y}))=\delta
_{k,l}(G^{x,y})^{2k}\frac{\alpha(\alpha+1)...(\alpha+k-1)}{k!} \label{corr}%
\end{equation}
Let us stress the fact that $G^{x,x}$ and $G^{y,y}$ do not appear on the right
hand side of this formula. This is quite important from the renormalisation
point of view, as we will consider in the last section the two dimensional
Brownian motion for which the Green function diverges on the diagonal.

More generally one can prove similar formulas for products of higher order.

It should also be noted that if we let $\alpha$ increase, $(1+\sigma_{x}t)^{-\alpha}\exp({\frac{\widehat
{\mathcal{L}_{\alpha}}^{x}t}{1+\sigma_{x}t}})$ and $P_{k}^{\alpha,\sigma_{x}%
}(\widehat{\mathcal{L}_{\alpha}}^{x})$ are $\sigma(\mathcal{L}_{\alpha}%
)$-martin\-gales with expectations respectively equal to $1$ and $0$.

Note that since $G_{\chi}M_{\chi}$\ is a contraction, from determinant
expansions given in \cite{VJ1} and \cite{VJ2}, we have%
\begin{equation}
\det(I+M_{\sqrt{\chi}}GM_{\sqrt{\chi}})^{-\alpha}=1+\sum_{k=1}^{\infty
}\frac{(-1)^{k}}{k!}\sum\chi_{i_{1}}...\chi_{i_{k}}Per_{\alpha}(G_{i_{l}%
,i_{m}},1\leq l,m\leq k). \label{expan}%
\end{equation}
\index{ $\alpha$-permanent}The $\alpha$-permanent $Per_{\alpha}$ is defined as $\sum_{\sigma
\in\mathcal{S}_{k}}\alpha^{m(\sigma)}G_{i_{1},i_{\sigma(1)}}...G_{i_{k}%
,i_{\sigma(k)}}$ with $m(\sigma)$ denoting the number of cycles in $\sigma$.
Then, from corollary \ref{lapl}, it follows that:%
\[
\mathbb{E}(\left\langle \widehat{\mathcal{L}_{\alpha}},\chi\right\rangle
^{k})=\sum\chi_{i_{1}}...\chi_{i_{k}}Per_{\alpha}(G_{i_{l},i_{m}},1\leq
l,m\leq k).
\]
Note that an explicit form for the multivariate negative binomial
distribution, and therefore, a series expansion for the density of the
multivariate gamma distribution, follows directly (see \cite{VJ2}) from this
determinant expansion.

It is actually not difficult to give a direct proof of this result. Thus, the
Poisson process of loops provides a natural probabilistic proof and
interpretation of this combinatorial identity (see \cite{VJ2} for an
historical view of the subject).

\bigskip We can show in fact that:

\begin{proposition}
For any $(x_{1},...x_{k})$ in $X^{k}$, $\mathbb{E}(\widehat{\mathcal{L}%
_{\alpha}}^{x_{1}}...\widehat{\mathcal{L}_{\alpha}}^{x_{k}})=Per_{\alpha
}(G^{x_{l},x_{m}},1\leq l,m\leq k)$
\end{proposition}

\begin{proof}
The cycles of the permutations in the expression of $Per_{\alpha}$\ are
associated with point configurations on loops. We obtain the result by summing
the contributions of all possible partitions of the points $i_{1}...i_{k}$
into a finite set of distinct loops.\ We can then decompose again the
expression according to ordering of points on each loop. We can conclude by
using the formula $\mu(\widehat{l}^{x_{1},...,x_{m}})=G^{x_{1},x_{2}}%
G^{x_{2},x_{3}}...G^{x_{m},x_{1}}$ and Campbell formula \eqref{campb}  .
\end{proof}

\bigskip

\begin{remark} \label{stirl}
We can actually, in the special case $i_{1}=i_{2}=...=i_{k}=x$, check this
formula in in a different way. From the moments of the Gamma distribution, we
have that $\mathbb{E}((\widehat{\mathcal{L}_{\alpha}}^{x})^{n})=(G^{x,x}%
)^{n}\alpha(\alpha+1)...(\alpha+n-1)$ and the $\alpha$-permanent can be
written $\sum_{1}^{n}d(n,k)\alpha^{k}$ where the coefficients $d(n,k)$ are the
numbers of $n-$permutations with $k$ cycles (Stirling numbers of the first
kind). One checks that $d(n+1,k)=nd(n,k)+d(n,k-1)$.
\index{Stirling numbers}
\end{remark}

\bigskip

Let $\mathcal{S}_{k}^{0}$ be the set of permutations of $k$\ elements without
fixed point. They correspond to configurations without isolated points.

Set $Per_{\alpha}^{0}(G^{i_{l},i_{m}},1\leq l,m\leq k)=\sum_{\sigma
\in\mathcal{S}_{k}^{0}}\alpha^{m(\sigma)}G^{i_{1},i_{\sigma(1)}}%
...G^{i_{k},i_{\sigma(k)}}$. Then an easy calculation shows that:

\begin{corollary}
\label{mom}$\mathbb{E}(\widetilde{\mathcal{L}_{\alpha}}^{i_{1}}...\widetilde
{\mathcal{L}_{\alpha}}^{i_{k}})=Per_{\alpha}^{0}(G^{i_{l},i_{m}},1\leq l,m\leq k)\!$
\end{corollary}

\begin{proof}
Indeed, the expectation can be written
\[
\sum_{p\leq k}\sum_{I\subseteq\{1,...k\},\left|  I\right|  =p}(-1)^{k-p}%
\prod_{l\in I^{c}}G^{i_{l},i_{l}}Per_{\alpha}(G^{i_{a},i_{b}},a,b\in I)
\]
and
\[
Per_{\alpha}(G^{i_{a},i_{b}},a,b\in I)=\sum_{J\subseteq I}\prod_{j\in
I\backslash J}G^{j,j}Per_{\alpha}^{0}(G^{i_{a},i_{b}},a,b\in J).
\]
Then, expressing $\mathbb{E}(\widetilde{\mathcal{L}_{\alpha}}^{i_{1}%
}...\widetilde{\mathcal{L}_{\alpha}}^{i_{k}})$ in terms of $Per_{\alpha}^{0}%
$'s, we see that if $J\subseteq\{1,...k\}$, ${\left|  J\right|  <k}$, the
coefficient of $Per_{\alpha}^{0}(G^{i_{a},i_{b}},a,b\in J)$ is ${\sum
_{I,I\supseteq J}(-1)^{k-\left|  I\right|  }\prod_{j\in J^{c}}G^{i_{j},i_{j}}%
}$ which vanishes as $(-1)^{-\left|  I\right|  }=(-1)^{\left|  I\right|
}=(-1)^{\left|  J\right|  }(-1)^{\left|  I\backslash J\right|  }$ and
$\sum_{I\supseteq J}(-1)^{\left|  I\backslash J\right|  }=(1-1)^{k-\left|
J\right|  }=0$.
\end{proof}

\bigskip

Set $Q_{k}^{\alpha,\sigma}(u)=P_{k}^{\alpha,\sigma}(u+\alpha\sigma)$ so that
$P_{k}^{\alpha,\sigma}(\widehat{\mathcal{L}_{\alpha}}^{x})=Q_{k}%
^{\alpha,\sigma}(\widetilde{\mathcal{L}_{\alpha}}^{x})$. This quantity will be
called the $n$-th renormalized self intersection local time or the $n$-th
renormalized power of the occupation field and denoted $\widetilde
{\mathcal{L}}_{\alpha}^{x,n}$.

From the recurrence relation of Laguerre polynomials
\[
nL_{n}^{(\alpha-1)}(u)=(-u+2n+\alpha-2)L_{n-1}^{(\alpha-1)}-(n+\alpha
-2)L_{n-2}^{(\alpha-1)},
\]
we get that
\[
nQ_{n}^{\alpha,\sigma}(u)=(u-2\sigma(n-1))Q_{n-1}^{\alpha,\sigma}%
(u)-\sigma^{2}(\alpha+n-2)Q_{n-2}^{\alpha,\sigma}(u).
\]
In particular $Q_{2}^{\alpha,\sigma}(u)=\frac{1}{2}(u^{2}-2\sigma
u-\alpha\sigma^{2})$, $Q_{3}^{\alpha,\sigma}(u)=\frac{1}{6}(u^{3}-6\sigma
u^{2}+3u\sigma^{2}(2-\alpha)+4\sigma^{3}\alpha)$.\newline We have also, from
(\ref{corr})%

\begin{equation}
\mathbb{E}(Q_{k}^{\alpha,\sigma_{x}}(\widetilde{\mathcal{L}_{\alpha}}%
^{x}),Q_{l}^{\alpha,\sigma_{y}}(\widetilde{\mathcal{L}_{\alpha}}^{y}%
))=\delta_{k,l}(G^{x,y})^{2k}\frac{\alpha(\alpha+1)...(\alpha+k-1)}{k!}\!
\label{orth}%
\end{equation}

The comparison of the identity (\ref{orth}) and corollary \ref{mom} \ yields a
combinatorial result which will be extended in the renormalizing procedure
presented in the last section.

The identity (\ref{orth}) can be considered as a polynomial identity in the
variables $\sigma_{x}$, $\sigma_{y}$ and $G^{x,y}$.

Set $Q_{k}^{\alpha,\sigma_{x}}(u)=\sum_{m=0}^{k}q_{m}^{\alpha,k}u^{m}%
\sigma_{x}^{k-m}$, and denote $N_{n,m,r,p}$ the number of ordered
configurations of $n$ black points and $m$ red points on $r$\ non trivial
oriented cycles, such that only $2p$ links are between red and black points.
We have first by corollary \ref{mom}:%
\[
\mathbb{E}((\widetilde{\mathcal{L}_{\alpha}}^{x})^{n}(\widetilde
{\mathcal{L}_{\alpha}}^{y})^{m})=\sum_{r}\sum_{p\leq\inf(m,n)}\alpha
^{r}N_{n,m,r,p}(G^{x,y})^{2p}(\sigma_{x})^{n-p}(\sigma_{y})^{m-p}%
\]
and therefore%
\begin{align}
\sum_{r}\sum_{p\leq m\leq k}\sum_{p\leq n\leq l}\alpha^{r}q_{m}^{\alpha
,k}q_{n}^{\alpha,l}N_{n,m,r,p}  &  =0\text{ \negthinspace unless
}p=l=k.\label{null}\\
\sum_{r}\alpha^{r}q_{k}^{\alpha,k}q_{k}^{\alpha,k}N_{k,k,r,k}  &
=\frac{\alpha(\alpha+1)...(\alpha+k-1)}{k!}. \label{factup}%
\end{align}
Note that one can check directly that $q_{k}^{\alpha,k}=\frac{1}{k!}$, and
$N_{k,k,1,k}=k!(k-1)!$, $N_{k,k,k,k}=k!$ which confirms the identity
(\ref{factup}) above.

\section{\label{hit}Hitting probabilities}

Denote by
\[
\lbrack H^{F}]_{y}^{x}=\mathbb{P}_{x}(x_{T_{F}}=y)
\]
\index{hitting distribution}the hitting distribution of $F$ by the Markov chain starting at $x$ ($H^{F}$
is called the \emph{balayage or Poisson kernel }in Potential theory). Set
$D=F^{c}$ and denote by $e^{D}$, $P^{D}=P|_{D\times D}$, $V^{D}=[(I-P^{D}%
)]^{-1}$ and $G^{D}=[(M_{\lambda}-C)|_{D\times D}]^{-1}$ the energy, the
transition matrix, the potential and the Green function of the Markov chain
killed at the hitting time of $F$.

Denote by $\mathbb{P}_{x}^{D}$\ the law of the killed Markov chain starting at x.

Hitting probabilities can be expressed in terms of Green functions. For $y\in
F$, we have
\[
\lbrack H^{F}]_{y}^{x}=1_{\{x=y\}}+\sum_{0}^{\infty}\sum_{z\in D}[(P^{D}%
)^{k}]_{z}^{x}P_{y}^{z}%
\]

As $G$\ and $G^{D}$ are symmetric, we have $[H^{F}G]_{y}^{x}=[H^{F}G]_{x}^{y}$
so that for any measure $\nu$,
\[
H^{F}(G\nu)=G(\nu H^{F}).
\]
In particular, the \emph{capacitary potential} $H^{F}1$ is the potential of
the \emph{capacitary measure} $\kappa H^{F}$.\newline Therefore we see that
for any function $f$ and measure $\nu$,
\[
e(H^{F}f,G^{D}\nu)=e(H^{F}f,G\nu)-e(H^{F}f,H^{F}G\nu)=\left\langle H^{F}%
f,\nu\right\rangle -e(H^{F}f,G(H^{F}\nu))=0
\]
as $(H^{F})^{2}=H^{F}$.

Equivalently, we have the following:

\begin{proposition}
\label{proj}For any $g$ vanishing on $F$, $e(H^{F}f,g)=0$ so that $I-H^{F}$ is
the {$e$-orthogonal} projection on the space of functions supported in $D$.
\end{proposition}

The energy of the capacitary potential of $F$, $e(H^{F}1,H^{F}1)$ equals the
mass of the capacitary measure $\left\langle \kappa H^{F},1\right\rangle $. It
is called the \emph{capacity} of $F$ and denoted $Cap_{e}(F)$.

Note that some of these results extend without difficulty to the recurrent
case. In particular, for any measure $\nu$ supported in $D$, $G^{D}\nu
=G(\nu-\nu H^{F})$ and $e(H^{F}f,G^{D}\nu)=0$ for all $f$. For further
developments see for example ( \cite{LJ0}) and its references.

The restriction property holds for $\mathcal{L}_{\alpha}$ as it holds for
$\mu$. The set $\mathcal{L}_{\alpha}^{D}$ of loops inside $D$ is associated
with $\mu^{D}$ and is independent of $\mathcal{L}_{\alpha}-\mathcal{L}%
_{\alpha}^{D}$. Therefore, we see from corollary \ref{lapl} that
\[
\mathbb{E}(e^{-\left\langle \widehat{\mathcal{L}_{\alpha}}-\widehat
{\mathcal{L}_{\alpha}^{D}},\chi\right\rangle })=\Big(          \frac{\det
(G_{\chi})}{\det(G)}\frac{\det(G^{D})}{\det(G_{\chi}^{D})}\Big)    ^{\alpha}.
\]

Note that for all $x$, $\mu(\widehat{l}^{x}>0)=\infty$. This is due to trivial
loops and it can be seen directly from the definition of $\mu$ that in this
simple framework the loops of $\mathcal{L}_{\alpha}$ cover the whole space $X$.

Note however that
\begin{align*}
\mu(\widehat{l}(F)>0,p>1)  &  =\mu(p>1)-\mu(\widehat{l(}F)=0,p>1)=\mu
(p>1)-\mu^{D}(p>1)\\
&  =-\log(\frac{\det(I-P)}{\det_{D\times D}(I-P)})=-\log(\frac{\det(G^{D}%
)}{\prod_{x\in F}\lambda_{x}\det(G)}).
\end{align*}
It follows that the probability that no non trivial loop (i.e. a loop which is
not reduced to a point) in $\mathcal{L}_{\alpha}$ intersects $F$ equals
\[
\exp(-\alpha\mu(\{l,\!p(l)>1,\widehat{l}(F)>0\}))=(\frac{\det(G^{D})}%
{\prod_{x\in F}\lambda_{x}\det(G)})^{\alpha}.
\]

Recall Jacobi's identity: for any $(n+p,n+p)$ invertible matrix $A$, denoting
$e_{i}$ the canonical basis,
\begin{align*}
\det(A^{-1})\det(A_{ij},1\leq i,j\leq n)  &  =\det(A^{-1})\det(Ae_{1}%
,...,Ae_{n},e_{n+1},...,e_{n+p})\\
&  =\det(e_{1},...,e_{n},A^{-1}e_{n+1},...,A^{-1}e_{n+p})\\
&  =\det((A^{-1})_{k,l},n\leq k,l\leq n+p).
\end{align*}
In particular, ${\det(G^{D})=\frac{\det(G)}{\det(G|_{F\times F})}}$, we can
also denote ${\frac{\det(G)}{\det_{F\times F}(G)}}$. So we have the

\begin{proposition}
The probability that no non-trivial loop in $\mathcal{L}_{\alpha}$ intersects
$F$ equals
\[
\lbrack\prod_{x\in F}\lambda_{x}\det_{F\times F}(G)]^{-\alpha}.
\]
Moreover ${\mathbb{E}(e^{-\left\langle \widehat{\mathcal{L}_{\alpha}}%
-\widehat{\mathcal{L}_{\alpha}^{D}},\chi\right\rangle })=(\frac{\det_{F\times
F}(G_{\chi})}{\det_{F\times F}(G)})^{\alpha}}.$
\end{proposition}

In particular, it follows that the probability that no non-trivial loop in
$\mathcal{L}_{\alpha}$ visits $x$ equals $(\frac{1}{\lambda_{x}G^{x,x}%
})^{\alpha}$ which is also a consequence of the fact that $N_{x}$ follows a
negative binomial distribution of parameters $-\alpha$ and $\frac{1}%
{\lambda_{x}G^{x,x}}$.

Also, if $F_{1}$ and $F_{2}$ are disjoint,
\begin{align}
\mu(\widehat{l(}F_{1})\widehat{l(}F_{2})>0)  &  =\mu(\widehat{l(}%
F_{1})>0,p>1)+\mu(\widehat{l(}F_{2})>0,p>1)-\mu(\widehat{l(}F_{1}\cup
F_{2})>0,p>1)\label{F1F2}\\
&  =\log(\frac{\det(G)\det(G^{D_{1}\cap D_{2}})}{\det(G^{D_{1}})\det(G^{D_{2}%
})})=\log(\frac{\det_{F_{1}\times F_{1}}(G)\det_{F_{2}\times F_{2}}(G)}%
{\det_{F_{1}\cup F_{2}\times F_{1}\cup F_{2}}(G)}).\nonumber
\end{align}
Therefore the probability that no loop in $\mathcal{L}_{\alpha}$ intersects
$F_{1}$ and $F_{2}$ equals
\[
\exp(-\alpha\mu(\{l,\prod\widehat{l(}F_{i})>0\}))=(\frac{\det(G^{D_{1}}%
)\det(G^{D_{2}})}{\det(G)\det(G^{D_{1}\cap D_{2}})})^{\alpha}=(\frac{\det
_{F_{1}\times F_{1}}(G)\det_{F_{2}\times F_{2}}(G)}{\det_{F_{1}\cup
F_{2}\times F_{1}\cup F_{2}}(G)})^{-\alpha}%
\]
It follows that the probability no loop in $\mathcal{L}_{\alpha}$ visits two
distinct points $x$ and $y$ equals $(\frac{G^{x,x}G^{y,y}-(G^{x,y})^{2}%
}{G^{x,x}G^{y,y}})^{\alpha}$ and in particular $1-\frac{(G^{x,y})^{2}}%
{G^{x,x}G^{y,y}}$ if $\alpha=1$.

\begin{exercise}
Generalize this formula to $n$ disjoint sets:
\[
\mathbb{P}(\nexists l\in\mathcal{L}_{\alpha},\prod\widehat{l(}F_{i})>0)=\Big(
\frac{\det(G)\prod_{i<j}\det(G^{D_{i}\cap D_{j}})}{\prod\det(G^{D_{i}}%
)\prod_{i<j<k}\det(G^{D_{i}\cap D_{j}\cap D_{k}})}\Big)   ^{-\alpha}%
\]
\end{exercise}

Note this yields an interesting determinant product inequality.

\chapter{The Gaussian free field}

\section{Dynkin's Isomorphism}

By a well known calculation on Gaussian measure, if $X$ is finite, for any
$\chi\in\mathbb{R}_{+}^{X}$,
\[
\frac{\sqrt{\det(M_{\lambda}-C)}}{(2\pi)^{\left|  X\right|  /2}}%
\int_{\mathbb{R}^{X}}e^{-\frac{1}{2}\sum\chi_{u}(v^{u})^{2}}e^{-\frac{1}%
{2}e(v)}\Pi_{u\in X}dv^{u}=\sqrt{\frac{\det(G_{\chi})}{\det(G)}}%
\]
and%
\[
\frac{\sqrt{\det(M_{\lambda}-C)}}{(2\pi)^{\left|  X\right|  /2}}%
\int_{\mathbb{R}^{X}}e^{-\frac{1}{2}\sum\chi_{u}(v^{u})^{2}}e^{-\frac{1}%
{2}e(v)}\Pi_{u\in X}dv^{u}=(G_{\chi})^{x,y}\sqrt{\frac{\det(G_{\chi})}%
{\det(G)}}%
\]

\index{free field}This can be easily reformulated by introducing on an independent probability
space the Gaussian free field $\phi$ defined by the covariance $\mathbb{E}_{\phi
}\mathbb{(}\phi^{x}\phi^{y})=G^{x,y}$ (this reformulation cannot be dispensed
when $X$ becomes infinite)

So we have
\[
\mathbb{E}_{\phi}\mathbb{(}e^{-\frac{1}{2}<\phi^{2},\chi>})=\det(I+GM_{_{\chi
}})^{-\frac{1}{2}}=\sqrt{\det(G_{\chi}G^{-1})}%
\]
and%
\[
\mathbb{E}_{\phi}\mathbb{(}\phi^{x}\phi^{y}e^{-\frac{1}{2}<\phi^{2},\chi
>})=(G_{\chi})^{x,y}\sqrt{\det(G_{\chi}G^{-1})}.
\]
Then since sums of exponentials of the form $e^{-\frac{1}{2}<\cdot,\chi>}%
$\ \ are dense in continuous functions on $\mathbb{R}_{+}^{X}$\ the following holds:

\begin{theorem}
\label{iso}

\begin{enumerate}
\item[a)] The fields $\widehat{\mathcal{L}}_{\frac{1}{2}}$ and $\frac{1}%
{2}\phi^{2}$ have the same distribution.

\item[b)] $\mathbb{E}_{\phi}\mathbb{(}(\phi^{x}\phi^{y}F(\frac{1}{2}\phi
^{2}))=\int\mathbb{E}(F(\widehat{\mathcal{L}_{\frac{1}{2}}}+\widehat{\gamma
}))\mu^{x,y}(d\gamma)$ for any bounded functional $F$ of a non negative field.
\end{enumerate}
\end{theorem}

\noindent\textbf{Remarks:}

a) This can be viewed as a version of Dynkin's isomorphism (Cf \cite{Dy}). It
can be extended to non-symmetric generators (Cf \cite{LJ2}).

b) By corollary \ref{camcam}, if $C_{x,y}\neq0$, b) implies that%
\[
\mathbb{E}_{\phi}\mathbb{(}\phi^{x}\phi^{y}F(\frac{1}{2}\phi^{2}%
))=\frac{2}{C_{x,y}}\mathbb{E}(F(\widehat{\mathcal{L}_{\frac{1}{2}}}%
)N_{x,y}^{(\frac{1}{2})})
\]

c) An analogous result can be given when $\alpha$ is any positive half
integer, by using real vector valued Gaussian field, or equivalently complex
fields for integral values of $\alpha$ (in particular $\alpha=1)$: If
$\overrightarrow{\phi}=(\phi_{1},\phi_{2},...,\phi_{k})$ \ are $k$ independent
copies of the real free field, the fields $\widehat{\mathcal{L}}_{\frac{k}{2}%
}$ and $\frac{1}{2}\left\|  \overrightarrow{\phi}\right\|  ^{2}=\frac{1}%
{2}\sum_{1}^{k}\phi_{j}^{2}$ have the same law and $\mathbb{E}%
_{\overrightarrow{\phi}}\mathbb{(}\left\langle \overrightarrow{\phi}%
^{x},\overrightarrow{\phi}^{y}\right\rangle F(\frac{1}{2}\left\|
\phi\right\|  ^{2}))=k\int\mathbb{E}(F(\widehat{\mathcal{L}_{\frac{k}{2}}%
}+\widehat{\gamma}))\mu^{x,y}(d\gamma)$.

The complex free field $\phi_{1}+i\phi_{2}$ will be denoted $\varphi$. $\ $If
we consider $k$ independent copies $\varphi_{j}$ of this field, $\widehat
{\mathcal{L}}_{k}$ and $\frac{1}{2}\left\|  \overrightarrow{\varphi}\right\|
^{2}=\frac{1}{2}\sum_{1}^{k}\varphi_{j}\overline{\varphi_{j}}$ have the same law.

d) Note it implies immediately that the process $\phi^{2}$ is infinitely
divisible. See \cite{EK} and its references for a converse and earlier proofs
of this last fact.

Theorem \ref{iso} suggests the following:

\begin{exercise}
Show that for any bounded functional $F$ of a non negative field, if $x_{i}$
are\ $2k$ points:
\[
\mathbb{E}_\phi(F(\phi^{2})\prod\phi^{x_{i}})=\int\mathbb{E}(F(\widehat
{\mathcal{L}_{\frac{1}{2}}}+\sum_{1}^{k}\widehat{\gamma_{j}}))\sum
_{\text{pairings}}\prod_{pairs}\mu^{y_{j},z_{j}}(d\gamma_{j})
\]

where $\sum_{\text{pairings}}$ means that the $k$ pairs $y_{j},z_{j}$ are
formed with all the $2k$\ points $x_{i}$, in all $\frac{(2k)!}{2^{k}k!}%
$\ possible ways.
\end{exercise}

Hint: As in the proof of theorem \ref{iso}, we take $F$ of the form
$e^{-\frac{1}{2}<\cdot,\chi>}$. Then we use the classical expression for the
expectation of a product of Gaussian variables known as Wick theorem (see for
example \cite{Nev}, \cite{Sim2}).

\bigskip

\begin{exercise}
For any $f$ in the Dirichlet space $\mathbb{H}$ of functions of finite energy
(i.e. for all functions if $X$ is finite), the law of $f+\phi$ is absolutely
continuous with respect to the law of $\phi$, with density $\exp
(<-Lf,\phi>_{m}-\frac{1}{2}e(f))$.
\end{exercise}

\begin{exercise}
\label{Nels} a) Using proposition \ref{proj},\ show (it was observed by Nelson
in the context of the classical (or Brownian) free field) that the Gaussian
field $\phi$ is Markovian: Given any subset $F$ of $X$, denote $\mathcal{H}%
_{F}$ the Gaussian space spanned by $\{\phi^{y},y\in F\}$. Then, for $x\in
D=F^{c}$, the projection of $\phi^{x}$ on $\mathcal{H}_{F}$ (i.e. the
conditional expectation of $\phi^{x}$ given $\sigma(\phi^{y},y\in F)$ ) is
$\sum_{y\in F}[H^{F}]_{y}^{x}\phi^{y}$.

b) Moreover, show that $\phi^{D}=\phi-H^{F}\phi$ is the Gaussian free field
associated with the process killed at the exit of $D$.
\end{exercise}

\section{Wick products}

We have seen in theorem \ref{iso}\ that $L^{2}$ functionals of $\widehat
{\mathcal{L}_{1}}$ can be represented in this space of Gaussian functionals.
In order to prepare the extension of this representation to the more difficult
framework of continuous spaces (which can often be viewed as scaling limits of
discrete spaces), including especially the planar Brownian motion considered
in \cite{LW}, we shall introduce the renormalized (or Wick) powers of $\phi$.\index{Wick powers}
We set $:(\phi^{x})^{n}:=(G^{x,x})^{\frac{n}{2}}H_{n}(\phi^{x}/\sqrt{G^{x,x}%
})$ where $H_{n}$ in the $n$-th Hermite polynomial (characterized by
$\sum\frac{t^{n}}{n!}H_{n}(u)=e^{tu-\frac{t^{2}}{2}}$).These variables are
orthogonal in $L^{2}$ and $E((:(\phi^{x})^{n}:)^{2})=n!G^{x,x}$.

Setting as before $\sigma_{x}=G^{x,x}$, from the relation between Hermite
polynomials $H_{2n}$\ and Laguerre polynomials $L_{n}^{-\frac{1}{2}}$,%

\[
H_{2n}(x)=(-2)^{n}n!L_{n}^{-\frac{1}{2}}(\frac{x^{2}}{2})
\]
it follows that:%
\[
:(\phi^{x})^{2n}:=2^{n}n!P_{n}^{\frac{1}{2},\sigma}((\frac{(\phi^{x})^{2}}%
{2})).
\]
and
\[
\mathbb{E}_{\phi}((:(\phi^{x})^{n}:)^{2})=\sigma_{x}^{n}n!
\]

More generally, if $\phi_{1},\phi_{2},...,\phi_{k}$ \ are $k$ independent
copies of the free field, we can define

$:\prod_{j=1}^{k}(\phi_{j}^{x})^{n_{j}}:\;=\prod_{j=1}^{k}:(\phi_{j}%
^{x})^{n_{j}}:$. Then it follows that:%
\[
:(\sum_{1}^{k}((\phi_{j}^{x})^{2})^{n}:=\sum_{n_{1}+..+n_{k}=n}\frac{n!}%
{n_{1}!...n_{k}!}\prod_{j=1}^{k}:(\phi_{j}^{x})^{2n_{j}}:.
\]

On the other hand, from the generating function of the polynomials
$P_{n}^{\frac{k}{2},\sigma}$, we get easily that%
\[
P_{n}^{\frac{k}{2},\sigma}(\sum_{1}^{k}u_{j})=\sum_{n_{1}+..+n_{k}=n}%
\prod_{j=1}^{k}P_{n_{j}}^{\frac{1}{2},\sigma}%
(u_{j}).
\]
Therefore,%
\begin{equation}
P_{n}^{\frac{k}{2},\sigma}(\frac{\sum(\phi_{j}^{x})^{2}}{2})=\frac{1}{2^{n}%
n!}:(\sum_{1}^{k}(\phi_{j}^{x})^{2})^{n}:. \label{polywick}%
\end{equation}
Note that in particular, $:\sum_{1}^{k}(\phi_{j}^{x})^{2}:$ equals $(\phi
_{j}^{x})^{2}-\sigma^{x}$ These variables are orthogonal in $L^{2}$. Let
$\widetilde{l}^{x}=\widehat{l}^{x}-\sigma^{x}$ be the centered occupation
field. Note that an equivalent formulation of theorem \ref{iso} is that the
fields $\frac{1}{2}:\sum_{1}^{k}\phi_{j}^{2}:$ and $\widetilde{\mathcal{L}%
}_{\frac{k}{2}}$ have the same law.

If we use complex fields $P_{n}^{\frac{k}{2},\sigma}(\frac{\sum\varphi_{j}%
^{x}\overline{\varphi}_{j}^{x}}{2})=\frac{1}{2^{n}n!}:(\sum_{1}^{k}%
(\varphi_{j}^{x}\overline{\varphi}_{j}^{x})^{n}:.$

Let us now consider the relation of higher Wick powers with self intersection
local times.

Recall that the renormalized $n$-th self intersections field $\widetilde
{\mathcal{L}}_{\alpha}^{x,n}=P_{n}^{\alpha,\sigma}(\widehat{\mathcal{L}%
_{\alpha}}^{x})=Q_{n}^{\alpha,\sigma}(\widetilde{\mathcal{L}_{\alpha}}^{x}%
)$\ have been defined by orthonormalization in $L^{2}$\ of the powers of the
occupation time.

Then comes the

\begin{proposition}
a)The fields $\widetilde{\mathcal{L}}_{\frac{k}{2}}^{\cdot,n}$ and
$\frac{1}{n!2^{n}}:(\sum_{1}^{k}\phi_{j}^{2})^{n}:$ have the same law.

In particular $\widetilde{\mathcal{L}}_{k}^{\cdot,n}$ and $\frac{1}{n!2^{n}%
}:(\sum_{1}^{k}\varphi_{j}\overline{\varphi_{j}})^{n}:$ have the same law.
\end{proposition}

This follows directly from (\ref{polywick}).

\begin{remark}
As a consequence, we obtain from \ref{orth}\ and \ref{polywick}\ that:
\begin{equation}
\frac{k(k+2)...(k+2(n-1))}{2^{n}n!}=\sum_{n_{1}+...+n_{k}=n}\prod
\frac{2n_{i}!}{(2^{n_{i}}n_{i}!)^2}\label{freud}%
\end{equation}
Moreover, it can be shown that:
\[
\mathbb{E}(\prod_{j=1}^{r}Q_{k_{j}}^{\alpha,\sigma_{x_{j}}}(\widetilde
{\mathcal{L}_{\alpha}}^{x_{j}}))=\sum_{\sigma\in\mathcal{S}_{k_{1}%
,k_{2},...,k_{j}}}\alpha^{m(\sigma)}G^{y_{1},y_{\sigma(1)}}...G^{y_{k}%
,y_{\sigma(k)}}%
\]
$y_{i}=x_{j}$ for ${\sum_{1}^{j-1}k_{l}+1\leq i\leq\sum_{1}^{j-1}k_{l}+k_{j}}$
and where $\mathcal{S}_{k_{1},k_{2},...,k_{j}}$ denotes the set of
permutations $\sigma$ of $k=\sum k_{j}$ such that

${\sigma(\{\sum_{1}^{j-1}k_{l}+1,...\sum_{1}^{j-1}k_{l}+k_{j}\})\cap\{\sum
_{1}^{j-1}k_{l}+1,...\sum_{1}^{j-1}k_{l}+k_{j}\}}$ is empty for all $j$.
\end{remark}

The identity follows from Wick's theorem when $\alpha$ is an integer, then
extends to all $\alpha$ since both members are polynomials in $\alpha$. The
condition on $\sigma$ indicates that no pairing is allowed inside the same
Wick power. For the proof, one can view each term of the form $:(\varphi
_{l}^{x}\overline{\varphi}_{l}^{x})^{k}:$ as the product of $k$ distinct
pairs, in a given order, then the pairings between $\varphi$'s and
$\overline{\varphi}$'s are defined by an element of $\mathcal{S}_{k_{1}%
,k_{2},...k_{j}}$ and a system of permutations of $\mathcal{S}_{k_{1}%
},...\mathcal{S}_{k_{j}}$. This system of permutations produces multiplicities
that cancel with the $\frac{1}{k_{i}!}$ factors in the expression. Note
finally that $\mathbb{E}(\varphi_{l}^{x}\overline{\varphi}_{l}^{y})=2G^{x,y}$
to cancel the $2^{-k_{i}}$ factors.

\section{The Gaussian Fock space structure}

The Gaussian space $\mathcal{H}$ spanned by $\{\phi^{x},x\in X\}$ is
isomorphic to the dual of the Dirichlet space $\mathbb{H}^{\ast}$ by the
linear map mapping $\phi^{x}$ on $\delta_{x}$. This isomorphism extends into
an isomorphism between \ the space of square integrable functionals of the
Gaussian field and the real symmetric Fock space $\Gamma^{\odot}%
(\mathbb{H}^{\ast})=\overline{\oplus\mathbb{H}^{\ast\odot n}}$\ obtained as
the closure of the sum of all symmetric tensor powers of $\mathbb{H}^{\ast}$
(the zero-th tensor power is $\mathbb{R}$). In the case of discrete spaces,
these symmetric tensor powers can be represented simply as symmetric signed measures
on $X^{n}$ (with no additional constraint in the finite space case). In terms
of the ordinary tensor product $\otimes$, the symmetric tensor product
$\mu_{1}\odot...\odot\mu_{n}$ is defined as $\frac{1}{\sqrt{n!}}\sum
_{\sigma\in\mathcal{S}_{n}}\mu_{\sigma(1)}\otimes...\otimes\mu_{\sigma(n)}$ so that in particular, $\left\|  \mu^{\odot n}\right\|  ^{2}=n!(\sum
G^{x,y}\mu_{x}\mu_{y})^{n}$.
The construction of $\Gamma(\mathbb{H}^{\ast})$\ is known as Bose second
quantization. The isomorphism mentionned above is defined by the following
identification, done for any $\mu$ in $\mathbb{H}^{\ast}$:%
\[
\exp^{\odot}(\mu)\rightleftarrows\exp(\sum_{x}\phi^{x}\mu_{x}-\frac{1}{2}%
\sum_{x,y}G^{x,y}\mu_{x}\mu_{y})
\]
which is an isometry as
\[
\mathbb{E}_{\phi}(\exp(\sum_{x}\phi^{x}\mu_{x}-\frac{1}{2}\sum_{x,y}G^{x,y}%
\mu_{x}\mu_{y})\exp(\sum_{x}\phi^{x}\mu_{x}^{\prime}-\frac{1}{2}\sum
_{x,y}G^{x,y}\mu_{x}^{\prime}\mu_{y}^{\prime}))=\exp(-\sum_{x,y}G^{x,y}\mu
_{x}\mu_{y}^{\prime})).
\]
The proof is completed by observing that linear combination of
\[
\exp(\sum_{x}\phi^{x}\mu_{x}-\frac{1}{2}\sum_{x,y}G^{x,y}\mu_{x}\mu_{y})
\]
form a dense algebra in the space of square integrable functionals of the
Gaussian field.

The $n$-th Wick power $:(\phi^{x})^{n}:$ is the image of the $n$-th symmetric
tensor\ power $\delta_{x}^{\odot n}$. More generally, for any $\mu$ in
$\mathbb{H}^{\ast}$, $:(\sum_{x}\phi^{x}\mu_{x})^{n}:$ is the image of the
$n$-th symmetric tensor\ power $\mu^{\odot n}$. Therefore, $:\phi^{x_{1}}%
\phi^{x_{2}}...\phi^{x_{n}}:$ is the image of $\delta_{x_{1}}\odot
...\odot\delta_{x_{n}}$, and polynomials of the field are associated with the
non completed Fock space $\oplus\mathbb{H}^{\ast\odot n}$.

For any $x\in X$, the anihilation operator $a_{x}$\ and the creation operator
$a_{x}^{\ast}$ are defined as follows, on the uncompleted Fock space
$\oplus\mathbb{H}^{\ast\odot n}$:%
\begin{align*}
a_{x}(\mu_{1}\odot...\odot\mu_{n})  &  =\sum_{k}G\mu_{k}(x)\mu_{1}%
\odot...\odot\mu_{k-1}\odot\mu_{k+1}\odot...\odot\mu_{n}\\
a_{x}^{\ast}(\mu_{1}\odot...\odot\mu_{n})  &  =\delta_{x}\odot\mu_{1}%
\odot...\odot\mu_{n}.
\end{align*}
Moreover, we set $a_{x}1=0$ for all $x$. These operator $a_{x}$ and
$a_{x}^{\ast}$ are clearly dual of each other and verify the commutation
relations:%
\[
\lbrack a_{x},a_{y}^{\ast}]=G^{x,y}\;\;[a_{x}^{\ast},a_{y}^{\ast}%
]=[a_{x},a_{y}]=0.
\]
The isomorphism allows to represent them on polynomials of the field as
follows:
\begin{align*}
a_{x}  &  \rightleftarrows\sum_{y}G^{x,y}\frac{\partial}{\partial\phi^{y}}\\
a_{x}^{\ast}  &  \rightleftarrows\phi_{x}-\sum_{y}G^{x,y}\frac{\partial
}{\partial\phi^{y}}%
\end{align*}

Therefore, the Fock space structure is entirely transported on the space of
square integrable functionals of the free field.

In the case of a complex field $\varphi$, the space of square integrable
functionals of $\varphi$ and $\overline{\varphi}$ is isomorphic to the tensor
product of two copies of the symmetric Fock space $\Gamma^{\odot}%
(\mathbb{H}^{\ast})$, denoted by $\mathcal{F}_{B}$. The complex Fock space
stucture is defined by two commuting sets of creation and anihilation
operators:%
\[
a_{x}=\sqrt{2}\sum_{y}G^{x,y}\frac{\partial}{\partial\varphi^{y}}\;\;\;a_{x}^{\ast
}=\frac{\varphi^{x}}{\sqrt{2}}-\sqrt{2}\sum_{y}G^{x,y}\frac{\partial}{\partial\overline{\varphi}^{y}}%
\]%
\[
b_{x}=\sqrt{2}\sum_{y}G^{x,y}\frac{\partial}{\partial\varphi^{y}}\;\;\;b_{x}^{\ast
}=\frac{\overline{\varphi}^{x}}{\sqrt{2}}-\sqrt{2}\sum_{y}G^{x,y}\frac{\partial}{\partial\varphi^{y}}%
\]
(Recall that if $z=x+iy$, $\frac{\partial}{\partial z}=\frac{\partial
}{\partial x}-i\frac{\partial}{\partial y}$ and $\frac{\partial}%
{\partial\overline{z}}=\frac{\partial}{\partial x}+i\frac{\partial}{\partial
y}$).\newline See \cite{Sim2}, \cite{Nev} for a more general description of
this isomorphism.

\begin{exercise}
Let $A$ and $B$ be two polynomials in $\varphi$ and $\overline{\varphi}$,
identified with finite degree element in $\mathcal{F}_{B}$. Show by recurrence
on the degrees that $AB=\sum\frac{1}{p!}\gamma_{p}(A,B)$ with $\gamma
_{0}(A,B)=A\odot B$ and $\gamma_{p+1}(A,B)=\sum_{x}(\gamma_{p}(a_{x}%
A,b_{x}B)+\gamma_{p}(b_{x}A,a_{x}B))$.
\end{exercise}

\section{The Poissonian Fock space structure}

Another symmetric Fock space structure is defined on the spaces of $L^{2}%
$-functionals of the loop ensemble $\mathcal{LP}$. It is based on the space
$\mathfrak{h}=L^{2}(\mu^{L})$ where $\mu^{L}\ $denotes $\mu\otimes Leb$. For
any $G\in L^{2}(\mu^{L})$ define $G_{(\varepsilon)}%
(t,l)=G(t,l)1_{\{T(l)>\varepsilon\}}1_{\{\left|  t\right|  <\frac{1}%
{\varepsilon}\}}$. Note that $G_{(\varepsilon)}$ is always integrable. Define%
\[
\mathfrak{h}_{0}=\bigcup_{\varepsilon>0}\{G\in L^{\infty}(\mu^{L}%
),\;\exists\varepsilon>0,\;G=G_{(\varepsilon)}\text{ and }\int Gd(\mu
^{L})=0\}.
\]
The algebra $\mathfrak{h}_{0}$ is dense in $L^{2}(\mu^{L})$ (as, for example,
compactly supported square integrable functions with zero integral are dense
in $L^{2}(Leb)$).

Given any $F$ in $\mathfrak{h}_{0}$, $\mathcal{LP}(F)=\sum_{(t_{i},l_{i}%
)\in\mathcal{LP}}F(t_{i},l_{i})$  is well defined and $\mathbb{E}%
(\mathcal{LP}(F)^{2})=\left\langle F,F\right\rangle _{L^{2}(\mu^{L})}$. By
Stone Weierstrass theorem, the algebra generated by $\mathcal{LP}%
(\mathfrak{h}_{0})$ is dense in $L^{2}(\mu^{L})$.

By Campbell formula, the $n$-th chaos, isomorphic to the symmetric tensor
product $\mathfrak{h}^{\odot n}$, can be defined as the closure of the linear
span of functions of $n$ distinct points of $\mathcal{LP}$ of the form
\[
\sum_{\sigma\in\mathcal{S}_{n}}\prod_{1}^{n}G_{\sigma(j)}(l_{j},\alpha_{j})
\]
with $G_{j}$ in $\mathfrak{h}_{0}$.

Denote by $t\mathcal{LP}$ the jump times of the Poisson process $\mathcal{LP}%
$. It follows directly from formula \eqref{campb}  that for%

\[
\Phi=\frac{1}{\sqrt{n!}}\sum_{\sigma\in\mathcal{S}_{n}}\sum_{t_{1}%
<t_{2}<...<t_{n}\in t\mathcal{LP}}\prod_{1}^{n}G_{\sigma(j)}(l_{j}%
,t_{j}))\text{ and }\Phi^{\prime}=\frac{1}{\sqrt{n!}}\sum_{\sigma
\in\mathcal{S}_{n^{\prime}}}\sum_{t_{1}<t_{2}...<t_{n^{\prime}}\in
t\mathcal{LP}}\prod_{1}^{n^{\prime}}G_{\sigma(j^{\prime})}^{\prime
}(l_{j^{\prime}},t_{j^{\prime}})),
\]
with $G_{j}$, $G_{j^{\prime}}^{\prime}$ in $\mathfrak{h}_{0}$,
\[
\mathbb{E}(\Phi\Phi^{\prime})=1_{\{n=n^{\prime}\}}Per(\left\langle
G_{j},G_{j^{\prime}}^{\prime}\right\rangle _{L^{2}(\mu^{L})},1\leq
j,j^{\prime}\leq n)
\]
which equals $1_{\{n=n^{\prime}\}}\left\langle G_{1}\odot G_{2}...\odot
G_{n},G_{1}^{\prime}\odot G_{2}^{\prime}...\odot G_{n^{\prime}}^{\prime
}\right\rangle $, $\odot$ denoting the symmetric tensor product (Cf
\cite{Bourb}, \cite{Nev}.

This proves the existence of an isomorphism $Iso$\ between the algebra
generated by $\mathcal{LP}(\mathfrak{h}_{0})$\ and the tensor algebra
$\oplus\mathfrak{h}_{0}^{\odot n}$\ which extends into an isomorphism between
the space $L^{2}(\mathbb{P}_{\mathcal{LP}})$ of\ square integrable functionals
of the Poisson process of loops $\mathcal{LP}$ and the symmetric (or bosonic)
Fock space $\overline{\oplus\mathfrak{h}^{\odot n}}$. We have
\[
Iso(\frac{1}{\sqrt{n!}}\sum_{\sigma\in\mathcal{S}_{n}}\sum_{t_{1}%
<t_{2}<...<t_{n}\in t\mathcal{LP}}\prod_{1}^{n}G_{\sigma(j)}(l_{j}%
,t_{j})))=G_{1}\odot G_{2}\odot...\odot G_{n}.
\]

This formula extends to $G_{i}\in\mathfrak{h}$. The closure of the space that
functionals of this form generate linearly is by definition the n-th chaos
which is isomorphic to the symmetric tensor product $\mathfrak{h}^{\odot n}$. 

Note that for any $G$ in $\mathfrak{h}_{0}$, the image by this isomorphism
$Iso$\ of the tensor exponential $\exp^{\odot}(G)$ is $\prod_{t_{i}\in
t\mathcal{LP}}(1+G(l_{i},t_{i}))$.

Note also that for all $F$ in $\mathfrak{h}$, $\left\|  \exp^{\odot
}(F)\right\|  ^{2}=e^{\int F^{2}d\mu^{L}}$

\begin{proposition}
For any $F$ in $\mathfrak{h}$, the image by $Iso$ of the tensor exponential
$\exp^{\odot}(F)$ is obtained as the limit in $L^{2}$, as $\varepsilon
\rightarrow0$ of

 $\prod_{t_{i}\in t\mathcal{LP}}(1+F_{(\varepsilon)}%
(l_{i},t_{i}))e^{-\int F_{(\varepsilon)}d\mu^{L}}$.
\end{proposition}

\begin{proof}
Note first that
\[
\mathbb{E}(\prod_{t_{i}\in t\mathcal{LP}}(1+F_{(\varepsilon)}(l_{i}%
,t_{i}))e^{-\int F_{(\varepsilon)}d\mu^{L}})=1
\]
and%

\begin{multline*}
\mathbb{E}([\prod_{t_{i}\in t\mathcal{LP}}(1+F_{(\varepsilon)}(l_{i}%
,t_{i}))e^{-\int F_{(\varepsilon)}d\mu^{L}}]^{2})\\
=\exp(\int[(1+F_{(\varepsilon)})^{2}-1]d\mu^{L}e^{-2\int F_{(\varepsilon)}%
d\mu^{L}}\\
=\exp(\int F_{(\varepsilon)}^{2}d\mu^{L})
\end{multline*}
converges towards $\exp(\int F^{2}d\mu^{L}).$\newline Then note that for any
$G$ in $\mathfrak{h}_{0}$,
\begin{multline*}
\lim_{\varepsilon\rightarrow0}\mathbb{E}(\prod_{t_{i}\in t\mathcal{LP}%
}(1+G(l_{i},t_{i}))(1+F_{(\varepsilon)}(l_{i},t_{i}))e^{-\int F_{(\varepsilon
)}d\mu^{L}})\\
=\lim_{\varepsilon\rightarrow0}\exp(\int[[1+G][1+F_{(\varepsilon)}%
]-1]e^{-\int(F_{(\varepsilon)})d\mu^{L}}d(\mu^{L})\\
=\lim_{\varepsilon\rightarrow0}\exp(\int F_{(\varepsilon)}Gd\mu^{L})=\exp(\int
FGd\mu^{L})=\left\langle \exp^{\odot}(F),\exp^{\odot}(G)\right\rangle .
\end{multline*}
\end{proof}

\begin{exercise}
For any $F$ in $\mathfrak{h}$, set $F_{\leq\alpha}(l,t)=F(l,t)1_{\{t\leq
\alpha\}}$. Show that $Iso(\exp^{\odot}(F_{\leq\alpha}))$ is a $\sigma
(\mathcal{L}_{\alpha})$-martingale.

Prove that the $\sigma(\mathcal{L}_{\alpha})$-martingale $(1+\sigma
_{x}t)^{-\alpha}\exp(\frac{\widehat{\mathcal{L}_{\alpha}}^{x}t}{1+\sigma_{x}%
t})$ is in this way associated with $F(l,t)=e^{\frac{\widehat{l}^{x}%
t}{1+\sigma_{x}t}}-1$.

Deduce from this an expression of $P_{k}^{\alpha,\sigma_{x}}(\widehat
{\mathcal{L}_{\alpha}}^{x})$ in terms of $\sigma_{x}^{k_{i}}D_{k_{i}%
}(\frac{\widehat{l_{i}}^{x}}{\sigma_{x}})$ (the polynomials defined in section
\ref{Occup}) ,$l_{i}$ denoting distinct loops in $\mathcal{L}_{\alpha}$ and
$k_{i}$ positive integers less than $k$.
\end{exercise}

For any $G$ in $\mathfrak{h}_{0}$, unbounded annihilation and creation
operators $\mathfrak{A}_{G}$ and $\mathfrak{A}_{G}^{\ast}$ are defined on
$\oplus\mathfrak{h}^{\odot n}$\
\[
\mathfrak{A}_{G}(G_{1}\odot...\odot G_{n})=\sum_{k=1}^{n}\left\langle
G,G_{k}\right\rangle _{\mathfrak{h}}G_{1}\odot...G_{k-1}\odot G_{k+1}...\odot
G_{n}%
\]
and%
\[
\mathfrak{A}_{G}^{\ast}(G_{1}\odot...\odot G_{n})=G\odot G_{1}\odot...\odot
G_{n}%
\]
Note that
\begin{align*}
\lbrack\mathfrak{A}_{G}^{\ast},\mathfrak{A}_{F}^{\ast}]  &  =[\mathfrak{A}%
_{G},\mathfrak{A}_{F}]=0\\
\lbrack\mathfrak{A}_{G},\mathfrak{A}_{F}^{\ast}]  &  =\left\langle
F,G\right\rangle _{L^{2}(\mu^{L})}%
\end{align*}
Moreover, $\mathfrak{A}_{G}^{\ast}$\ is adjoint to $\mathfrak{A}_{G}$\ in
$L^{2}(\mathbb{P}_{\mathcal{LP}})$, and the operators $\mathfrak{F}%
_{G}=\mathfrak{A}_{G}^{\ast}+\mathfrak{A}_{G}$ commute.

Note also that the creation operator can be defined directly on the space of
loop configurations: by proposition \ref{camca}\ given any bounded functional
$\Psi$ on loops configurations,%
\[
Iso\mathfrak{A}_{G}^{\ast}Iso^{-1}\Psi(\mathcal{LP})=\int\Psi(\mathcal{LP}%
\cup\{l,t\})G(l,t)d\mu^{L}%
\]
It is enough to verify it for $Iso^{-1}\Psi$ in $\mathfrak{h}_{0}^{\odot n}$.

For any $G$ in $\mathfrak{h}_{0}\cap L^{\infty}$, note that $\mathfrak{F}_{G}$
does not represent the multiplication by $\mathcal{LP}(G)$, though we have,
for all $\Phi$ in $\oplus\mathfrak{h}^{\odot n}$\ and $G$ in $\mathfrak{h}%
_{0}$, $\mathbb{E}(Iso(\Phi)\mathcal{LP}(G))=\left\langle \mathfrak{F}%
_{G}1,\Phi\right\rangle =\left\langle 1,\mathfrak{F}_{G}\Phi\right\rangle $.

The representation of this operator of multiplication in the Fock space
structure can be done as follows:

Setting $M_{G}F=GF,$ for all $\Phi$ in $\oplus\mathfrak{h}^{\odot n}$
\[
(\sum_{t_{i}\in t\mathcal{LP}}G(l_{i},t_{i}))Iso(\Phi)=Iso(\mathfrak{F}%
_{G}\Phi+d\Gamma(M_{G})).
\]
The notation $d\Gamma$ refers to the second quantisation functor $\Gamma$: if
$B$ is any bounded operator on a Hilbert space $\mathfrak{h}$, $\Gamma(B)$ is
defined on $\oplus\mathfrak{h}^{\odot n}$\ by \ the sum of the operators
$B^{\otimes n}$ acting on each symmetric tensor product $\mathfrak{h}^{\odot
n}$ and $d\Gamma(B)=\frac{\partial}{\partial t}\Gamma(e^{tB})|_{t=0}$. In
fact, given any orthonormal basis $E_{k}$ of $\mathfrak{h}$, we have, for any
$\Phi$ in $\oplus\mathfrak{h}^{\odot n}$\ and $G$ in $\mathfrak{h}$
\[
d\Gamma(B)=\sum\mathfrak{A}_{BE_{k}}^{\ast}\mathfrak{A}_{E_{k}}=\sum
\mathfrak{A}_{E_{k}}^{\ast}\mathfrak{A}_{B^{\ast}E_{k}}%
\]
as
\[
\sum\left\langle F,E_{k}\right\rangle _{L^{2}(\mu^{L})}(BE_{k})=\sum
\left\langle F,BE_{k}\right\rangle _{L^{2}(\mu^{L})}E_{k}=BF.
\]

\chapter{Energy variation and representations}
\section{Variation of the energy form}
The loop measure $\mu$ depends on the energy $e$ which is defined by the free
parameters $C,\kappa$. It will sometimes be denoted $\mu_{e}$. We shall denote
$\mathcal{Z}_{e}$ the determinant $\det(G)=\det(M_{\lambda}-C)^{-1}$. Then
$\mu(p>1)=\log(\mathcal{Z}_{e})+\sum_{x\in X}\log(\lambda_{x})$.

$\mathcal{Z}_{e}^{\alpha}$ is called the partition function of $\mathcal{L}%
_{\alpha}$.

We wish to study the dependance of $\mu$ on $C$ and $\kappa$. The following
result is suggested by an analogy with quantum field theory (Cf \cite{Gaw}).

\begin{proposition}
\label{boub}

\begin{enumerate}
\item[i)] $\frac{\partial\mu}{\partial\kappa_{x}}=-\widehat{l}^{x}\mu$.

\item[ii)] If $C_{x,y}>0$, $\frac{\partial\mu}{\partial C_{x,y}}=-T^{x,y}\mu$
with $T^{x,y}(l)=(\widehat{l}^{x}+\widehat{l}^{y})-\frac{N_{x,y}}{C_{x,y}%
}(l)-\frac{N_{y,x}}{C_{x,y}}(l)$.
\end{enumerate}
\end{proposition}

\begin{proof}
Recall that by formula (\ref{d}): $\mu^{\ast}(p=1,\xi=x,\widehat{\tau}\in
dt)=e^{-\lambda_{x}t}\frac{dt}{t}$ and
\[
\mu^{\ast}(p=k,\xi_{i}=x_{i},\widehat{\tau}_{i}\in dt_{i})=\frac{1}{k}%
\prod_{x,y}C_{x,y}^{N_{x,y}}\prod_{x}\lambda_{x}^{-N_{x}}\prod_{i\in
\mathbb{Z}/p\mathbb{Z}}\lambda_{\xi_{i}}e^{-\lambda_{\xi_{i}}t_{i}}dt_{i}.
\]
Moreover we have $C_{x,y}=C_{y,x}=\lambda_{x}P_{y}^{x}$ and $\lambda
_{x}=\kappa_{x}+\sum_{y}C_{x,y}.$\newline The two formulas follow by
elementary calculation.
\end{proof}

\medskip Recall that $\mu(\widehat{l}^{x})=G^{x,x}$ and $\mu(N_{x,y}%
)=G^{x,y}C_{x,y}$. So we have $\mu(T^{x,y})=G^{x,x}+G^{y,y}-2G^{x,y}$. Then,
the above proposition allows us to compute all moments of $T$ and $\widehat
{l}$ relative to $\mu_{e}$ (they could be called Schwinger functions).

\begin{exercise}
Use the proposition above to show that:%
\[
\int\widehat{l}^{x}\widehat{l}^{y}\mu(dl)=(G^{x,y})^{2}%
\]%
\[
\int\widehat{l}^{x}T^{y,z}(l)\mu(dl)=(G^{x,y}-G^{x,z})^{2}%
\]
and%
\[
\int T^{x,y}(l)T^{u,v}(l)\mu(dl)=(G^{x,u}+G^{y,v}-G^{x,v}-G^{y,u})^{2}=(K^{(x,y),(u,v)})^{2}%
\]
\end{exercise}

Hint: The calculations are done noticing that for any invertible matrix
function $M(s)$, $\frac{d}{ds}M(s)^{-1}=-M(s)^{-1}M^{\prime}(s)M(s)^{-1}$. The
formula is applied to $M=M_{\lambda}-C$ and $s=\kappa_{x}$ or $C_{x,y}$.

\begin{exercise}
Show that $\int(\frac{1}{2}\sum_{x,y}C_{x,y}T^{x,y}(l)+\sum_{x}\kappa
_{x}\widehat{l}^{x})\mu(dl)=\left|  X\right|  $ and that more generally, for
any $D\subset X$,
\[
\int(\frac{1}{2}\sum_{x,y}C_{x,y\in D}T^{x,y}(l)+\sum_{x\in D}\kappa
_{x}\widehat{l}^{x})\mu(dl)=\left|  D\right|  +\sum_{x\in D,y\in X-D}%
C_{x,y}G^{x,y}
\]
\end{exercise}

Set
\[
T_{x,y}^{(\alpha)}=\sum_{l\in\mathcal{L}_{\alpha}}T_{x,y}(l)=(\widehat
{\mathcal{L}}_{\alpha}^{x}+\widehat{\mathcal{L}}_{\alpha}^{y})-\frac{N_{x,y}%
^{(\alpha)}}{C_{x,y}}-\frac{N_{y,x}^{(\alpha)}}{C_{x,y}}%
\]
and
\[
\widetilde{T}_{x,y}^{(\alpha)}=T_{x,y}^{(\alpha)}-\mathbb{E}(T_{x,y}%
^{(\alpha)})=T_{x,y}^{(\alpha)}-\alpha(G^{x,x}+G^{y,y}-2G^{x,y}).
\]
We can apply proposition \ref{boub} to the Poissonnian loop ensembles, to get
the following

\begin{corollary}
\label{boub2}For any bounded functional $\Phi$ on loop configurations

\begin{enumerate}
\item[i)] $\frac{\partial}{\partial\kappa_{x}}\mathbb{E}(\Phi(\mathcal{L}%
_{\alpha}))=-\mathbb{E}(\Phi(\mathcal{L}_{\alpha})\widetilde{\mathcal{L}%
}_{\alpha}^{x})=\alpha\int\mathbb{E}((\Phi(\mathcal{L}_{\alpha})-\Phi
(\mathcal{L}_{\alpha}\cup\{\gamma\}))\mu^{x,x}(d\gamma).$

\item[ii)] If $C_{x,y}>0$,
\begin{multline*}
\frac{\partial}{\partial C_{x,y}}\mathbb{E}(\Phi(\mathcal{L}_{\alpha
}))=-\mathbb{E}(\widetilde{T}_{x,y}^{(\alpha)}\Phi(\mathcal{L}_{\alpha}))\\
=\alpha\int\mathbb{E}((\Phi(\mathcal{L}_{\alpha})-\Phi(\mathcal{L}_{\alpha
}\cup\{\gamma\}))[\mu^{x,x}(d\gamma)+\mu^{y,y}(d\gamma)-\mu^{x,y}(d\gamma
)-\mu^{y,x}(d\gamma)].
\end{multline*}
\end{enumerate}
\end{corollary}

The proof is easily performed, taking first $\Phi$ of the form $\sum
_{l_{1}\neq l_{2}...\neq l_{q}\in\mathcal{L}_{\alpha}}\prod_{1}^{q}G_{j}%
(l_{j}))$. We apply Campbell formula to deduce the first half of both
identities, then corollary \ref{camcam} to get the second half.

This result should be put in relation with propositions \ref{bridg}\ and
\ref{bridg2}\ and with the Poissonian Fock space structure defined above.

\begin{exercise}
Show that using theorem \ref{iso}, corollary \ref{boub2} implies that for any
function $F$ of an non-negative field and any edge $(x_{i},y_{i})$:
\begin{align*}
&  \mathbb{E}_{\phi}\mathbb{(}\frac{1}{2}:(\phi^{x}-\phi^{y})^{2}%
:F(\frac{1}{2}\phi^{2}))=\int\mathbb{E}(F(\widehat{\mathcal{L}_{\frac{1}{2}}%
})\widetilde{T}_{x,y}^{(\frac{1}{2})})\\
&  \mathbb{E}_{\phi}\mathbb{(}:\phi^{x}\phi^{y}:F(\frac{1}{2}\phi^{2}%
))=-\int\mathbb{E}(F(\widehat{\mathcal{L}_{\frac{1}{2}}})N_{x,y})
\end{align*}
Hint: Express the Gaussian measure and use the fact that $-\frac{\partial
}{\partial C_{x,y}}Log(\det(G))=G^{x,x}+G^{y,y}-2G^{x,y}$
\end{exercise}

\begin{exercise}
Setting $\widetilde{l}^{x}=\widehat{l}^{x}-G^{x,x}$\ and $\widetilde{T}%
_{x,y}(l)=T_{x,y}(l)-(G^{x,x}+G^{y,y}-2G^{x,y})$, show that we have:%
\begin{align*}
\int\widetilde{l}_{x}\widetilde{l}_{y}\mu(dl) =  &  (G^{x,y})^{2}%
-G^{x,x}G^{y,y}=\det\left(
\begin{array}
[c]{cc}%
\left\langle \phi^{x},\phi^{x}\right\rangle  & \left\langle \phi^{x},\phi
^{y}\right\rangle \\
\left\langle \phi^{y},\phi^{x}\right\rangle  & \left\langle \phi^{y},\phi
^{y}\right\rangle
\end{array}
\right) \\
\int\widetilde{l}_{x}\widetilde{T}_{y,z}(l)\mu(dl) =  &  (G^{x,y}-G^{x,z}%
)^{2}-G^{x,x}(G^{z,z}+G^{y,y}-2G^{y,z})\\
=  &  \det\left(
\begin{array}
[c]{cc}%
\left\langle \phi^{x},\phi^{x}\right\rangle  & \left\langle \phi^{x},\phi
^{y}-\phi^{z}\right\rangle \\
\left\langle \phi^{x},\phi^{y}-\phi^{z}\right\rangle  & \left\langle \phi
^{y},\phi^{y}\right\rangle
\end{array}
\right) \\
\int\widetilde{T}_{x,y}(l)\widetilde{T}_{u,v}(l)\mu(dl) =  &  (G_{x,u}%
+G_{y,v}-G_{x,v}-G_{y,u})^{2}\\
&  -(G^{x,x}+G^{y,y}-2G^{x,y})(G^{u,u}+G^{v,v}-2G^{u,v})\\
=  &  -\det\left(
\begin{array}
[c]{cc}%
\left\langle \phi^{x}-\phi^{y},\phi^{x}-\phi^{y}\right\rangle  & \left\langle
\phi^{x}-\phi^{y},\phi^{u}-\phi^{v}\right\rangle \\
\left\langle \phi^{u}-\phi^{v},\phi^{x}-\phi^{y}\right\rangle  & \left\langle
\phi^{u}-\phi^{v},\phi^{u}-\phi^{v}\right\rangle
\end{array}
\right)  .
\end{align*}
\end{exercise}

\begin{exercise}
For any bounded functional $\Phi$ on loop configurations, give two different
expressions for $\frac{\partial^{2}}{\partial\kappa_{x}\partial\kappa_{y}%
}\mathbb{E}(\Phi(\mathcal{L}_{\alpha}))$, $\frac{\partial^{2}}{\partial
C_{x,y}\partial\kappa_{z}}\mathbb{E}(\Phi(\mathcal{L}_{\alpha}))$ and
$\frac{\partial^{2}}{\partial C_{x,y}\partial C_{u,v}}\mathbb{E}%
(\Phi(\mathcal{L}_{\alpha}))$.
\end{exercise}

The proposition \ref{boub}\ is in fact the infinitesimal form of the following formula.

\begin{proposition}
Consider another energy form $e^{\prime}$ defined on the same graph. Then we
have the following identity:%
\[
\frac{\partial\mu_{e^{\prime}}}{\partial\mu_{e}}=e^{\sum N_{x,y}%
\log(\frac{C_{x,y}^{\prime}}{C_{x,y}})-\sum(\lambda_{x}^{\prime}-\lambda
_{x})\widehat{l}^{x}}.
\]
Consequently
\begin{equation}
\mu_{e}((e^{\sum N_{x,y}\log(\frac{C_{x,y}^{\prime}}{C_{x,y}})-\sum
(\lambda_{x}^{\prime}-\lambda_{x})\widehat{l}^{x}}-1))=\log(\frac{\mathcal{Z}%
_{e^{\prime}}}{\mathcal{Z}_{e}}). \label{mupr}%
\end{equation}
\end{proposition}

\begin{proof}
The first formula is a straightforward consequence of (\ref{dd}). The proof of
(\ref{mupr})\ goes by evaluating separately the contribution of trivial loops,
which equals $\sum_{x}\log(\frac{\lambda_{x}}{\lambda_{x}^{\prime}})$.
Indeed,
\[
\mu_{e}((e^{\sum N_{x,y}\log(\frac{C_{x,y}^{\prime}}{C_{x,y}})-\sum
(\lambda_{x}^{\prime}-\lambda_{x})\widehat{l}^{x}}-1)=\ \mu_{e^{\prime}%
}(p>1)-\mu_{e}(p>1)+\ \mu_{e}(1_{\{p=1\}}(e^{\sum(\lambda_{x}^{\prime}%
-\lambda_{x})\widehat{l}^{x}}-1)).
\]

The difference of the first two terms equals $\log(\mathcal{Z}_{e^{\prime}%
})+\sum\log(\lambda_{x}^{\prime})-(\log(\mathcal{Z}_{e})-\sum\log(\lambda
_{x}))$. The last term equals $\sum_{x}\int_{0}^{\infty}(e^{-\frac{\lambda
_{x}^{\prime}-\lambda_{x}}{\lambda_{x}}t}-1)\frac{e^{-t}}{t}dt$\ which can be
computed as before:
\begin{equation}
\mu_{e}(1_{\{p=1\}}(e^{\sum(\lambda_{x}^{\prime}-\lambda_{x})\widehat{l}^{x}%
}-1))=-\sum\log(\frac{\lambda_{x}^{\prime}}{\lambda_{x}}) \label{unpt}%
\end{equation}
\end{proof}

Integrating out the holding times, formula (\ref{mupr})\ can be written
equivalently:%
\begin{equation}
\mu_{e}(\prod_{(x,y)}[\frac{C_{x,y}^{\prime}}{C_{x,y}}]^{N_{x,y}}\prod
_{x}[\frac{\lambda_{x}}{\lambda_{x}^{\prime}}]^{N_{x}+1}-1)=\log
(\frac{\mathcal{Z}_{e^{\prime}}}{\mathcal{Z}_{e}}) \label{F}%
\end{equation}
and therefore%
\begin{equation}
\mathbb{E}(\prod_{(x,y)}[\frac{C_{x,y}^{\prime}}{C_{x,y}}]^{N_{x,y}%
(\mathcal{L}_{\alpha})}\prod_{x}[\frac{\lambda_{x}}{\lambda_{x}^{\prime}%
}]^{N_{x}(\mathcal{L}_{\alpha})+1})=\mathbb{E}(\prod_{(x,y)}[\frac{C_{x,y}%
^{\prime}}{C_{x,y}}]^{N_{x,y}(\mathcal{L}_{\alpha})}e^{-\left\langle
\lambda^{\prime}-\lambda,\widehat{\mathcal{L}_{\alpha}}\right\rangle
})=(\frac{\mathcal{Z}_{e^{\prime}}}{\mathcal{Z}_{e}})^{\alpha} \label{muprim}%
\end{equation}
Note also that $\prod_{(x,y)}[\frac{C_{x,y}^{\prime}}{C_{x,y}}]^{N_{x,y}%
}=\prod_{\{x,y\}}[\frac{C_{x,y}^{\prime}}{C_{x,y}}]^{N_{x,y}+N_{y,x}}$.

\bigskip

\begin{remark}
These $\frac{\mathcal{Z}_{e^{\prime}}}{\mathcal{Z}_{e}}$ determine, when
$e^{\prime}$ varies with $\frac{C^{^{\prime}}}{C}\leq1$ and $\frac{\lambda
^{\prime}}{\lambda}=1$, the Laplace transform of the distribution of the
traversal numbers of non oriented links $N_{x,y}+N_{y,x}$.
\end{remark}

\bigskip

\begin{remark}
( h-transforms) Note that if $C_{x,y}^{^{\prime}}=h^{x}h^{y}C_{x,y}$ and
$\kappa_{x}^{\prime}=-h^{x}(Lh)^{x}\lambda_{x}$ for some positive function $h$
on $E$ such that $Lh\leq0$, as $\lambda^{\prime}=h^{2}\lambda$ and
$[P^{\prime}]_{y}^{x}=\frac{1}{h^{x}}P_{y}^{x}h^{y}$, we have $[G^{\prime
}]^{x,y}=\frac{G^{x,y}}{h^{x}h^{y}}$ and $\frac{\mathcal{Z}_{e^{\prime}}%
}{\mathcal{Z}_{e}}=\frac{1}{\prod(h^{x})^{2}}$.\newline 
\end{remark}

\begin{remark}
\label{baba}Note also that $[\frac{\mathcal{Z}_{e^{\prime}}}{\mathcal{Z}_{e}%
}]^{\frac{1}{2}}=\mathbb{E_\phi(}e^{-\frac{1}{2}[e^{\prime}-e](\phi)})$, if $\phi$
is the Gaussian free field associated with $e$.
\end{remark}
\section{One-forms and representations}
Other variables of interest on the loop space are associated with elements of
the space $\mathbb{A}^{-}$ of\ odd real valued functions $\omega$ on oriented
links\ : $\omega^{x,y}=-\omega^{y,x}$. Let us mention a few elementary results.

The operator $[P^{(\omega)}]_{y}^{x}=P_{y}^{x}\exp(i\omega^{x,y})$ is also
self adjoint in $L^{2}(\lambda)$. The associated loop variable can be written
$\sum_{x,y}\omega^{x,y}N_{x,y}(l)$. We will denote it $\int_{l}\omega$. This
notation will be used even when $\omega$ is not odd. Note that $\int_{l}%
\omega$ is invariant if $\omega$ is replaced by $\omega+dg$ for some $g$. Set
$[G^{(\omega)}]^{x,y}=\frac{[(I-P^{(\omega)})^{-1}]_{y}^{x}}{\lambda_{y}}$. By
an argument similar to the one given above for the occupation field, we have:
\[
\mathbb{P}_{x,x}^{t}(e^{i\int_{l}\omega}-1)=\exp(t(P^{(\omega)}-I))_{x}%
^{x}-\exp(t(P-I))_{x}^{x}.
\]
Integrating in $t$ after expanding, we get from the definition of $\mu$:%
\[
\int(e^{i\int_{l}\omega}-1)d\mu(l)=\sum_{k=1}^{\infty}\frac{1}{k}%
[Tr((P^{(\omega)})^{k})-Tr((P)^{k})].
\]
Hence $\int(e^{i\int_{l}\omega}-1)d\mu(l)=\log[\det(-L(I-P^{(\omega)}%
)^{-1})]=\log(\det(G^{(\omega)}G^{-1}))$

We can now extend the previous formulas (\ref{F})\ and (\ref{muprim}) to
obtain, setting $\det(G^{(\omega)})=\mathcal{Z}_{e,\omega}$%
\begin{equation}
\int(e^{\sum N_{x,y}\log(\frac{C_{x,y}^{^{\prime}}}{C_{x,y}})-\sum(\lambda
_{x}^{^{\prime}}-\lambda_{x})\widehat{l}_{x}+i\int_{l}\omega}-1)\mu
_{e}(dl)=\log(\frac{\mathcal{Z}_{e^{\prime},\omega}}{\mathcal{Z}_{e}})
\label{F4}%
\end{equation}
and%
\begin{equation}
\mathbb{E}(\prod_{x,y}[\frac{C_{x,y}^{\prime}}{C_{x,y}}e^{i\omega_{x,y}%
}]^{N_{x,y}^{(\alpha)}}e^{-\sum(\lambda_{x}^{^{\prime}}-\lambda_{x}%
)\widehat{\mathcal{L}_{\alpha}}^{x}})=(\frac{\mathcal{Z}_{e^{\prime},\omega}%
}{\mathcal{Z}_{e}})^{\alpha} \label{F5}%
\end{equation}

\begin{remark}
\label{complex}The $\alpha$-th power of a complex number is a priori not
univoquely defined as a complex number. But $\log[\det(I-P^{(\omega)})]$ and
therefore $\log(\mathcal{Z}_{e,\omega})$ are well defined as $P^{(\omega)}$ is
a contraction. Then $\mathcal{Z}_{e,\omega}^{\alpha}$ is taken to be
$\exp(\alpha\log(\mathcal{Z}_{e,\omega}))$.
\end{remark}

\begin{remark}
\label{baba2}Note also that if $\varphi=\phi_{1}+i\phi_{2}$ is the complex
Gaussian free field associated with $e$,
\[
\frac{\mathcal{Z}_{e^{\prime},\omega}}{\mathcal{Z}_{e}}=\mathbb{E}_{\varphi
}\mathbb{(}e^{-\frac{1}{2}[\sum(\lambda_{x}^{\prime}-\lambda_{x})\varphi
^{x}\overline{\varphi}^{x}-\sum(C_{x,y}^{\prime}e^{i\omega_{x,y}}%
-C_{x,y})\varphi^{x}\overline{\varphi}^{y]}}).
\]
\end{remark}

To simplify the notations slightly, one could consider more general energy
forms with complex valued conductances so that the discrete one form is
included in $e^{\prime}$. But it is more interesting to generalize the notion
of perturbation of $P$ into $P^{(\omega)}$ as follows:

\index{unitary representation}\begin{definition}
A unitary representation of the graph $(X,E)$ is a family of unitary matrices
$\left[  U^{x,y}\right]  $, with common rank $d_{U}$, indexed by $E^{O}$, such
that $\left[  U^{y,x}\right]  =\left[  U^{x,y}\right]  ^{-1}$.

We set $P^{(U)}=P\otimes U$ (more explicitly $\left[  P^{(U)}\right]
_{x,i}^{y,j}=P_{x}^{y}\left[  U^{x,y}\right]  _{i}^{j}$).
\end{definition}

Similarly, we can define $C^{(U)}=\frac{\lambda}{d_{U}}P^{(U)}$,
$V^{(U)}=(I-P^{(U)})^{-1}$, $G^{(U)}=\frac{d_{U}V^{(U)}}{\lambda}$. One should
think of these matrices as square matrices indexed by $X$, whose entries are multiples of
elements of $SU(d_{U})$.

One forms define one-dimensional representations. The sum and tensor product
of two unitary representations $U$ and $V$ are unitary representations are
defined as usual, and their ranks are respectively $d_{U}+d_{V}$ and
$d_{U}d_{V}$.

\begin{definition}
Given any based loop $l$, if $p(l)\geq2$ and the associated discrete based
loop is $\xi=(\xi_{1},\xi_{2},...,\xi_{p})$, set $\tau_{U}(l)=\frac{1}{d_{U}%
}Tr(\prod U^{\xi_{i},\xi_{i+1}})$, and $\tau_{U}(l)=1$ if $p(l)=1$.

For any set of loops $\mathcal{L}$, we set $\tau_{U}(\mathcal{L})=\prod
_{l\in\mathcal{L}}\tau_{U}(l)$.
\end{definition}

\begin{remark}
\begin{enumerate}

\item[a)] $\left|  \tau_{U}(l)\right|  \leq1.$

\item[b)] $\tau_{U}$ is obviously a functional of the discrete \emph{loop}
$\xi^{%
{{}^\circ}%
}$ contained in $l^{%
{{}^\circ}%
}$.

\item[c)] $\tau_{U}(l)=1$ if $\xi^{%
{{}^\circ}%
}$ is tree-like. In particular it is always the case when the graph is a tree.

\item[d)] If $U$ and $V$ are two unitary representations of the graph,
$\tau_{U+V}=\tau_{U}+\tau_{V}$ and $\tau_{U\otimes V}=\tau_{U}\tau_{V}$.
\end{enumerate}
\end{remark}

From b) and c) above it is easy to get the first part of

\begin{theorem}
\begin{enumerate}

\item[i)] The trace $\tau_{U}(l)$ depends only on the canonical \emph{geodesic
loop} associated with the loop $\xi^{%
{{}^\circ}%
}$, i.e. of the conjugacy class of the element of the fundamental group  defined by the based loop $\xi$.

\item[ii)] The variables $\tau_{U}(l)$ determine, as $U$ varies, the geodesic
loop associated with $l$.
\end{enumerate}
\end{theorem}

\begin{proof}
The second assertion follows from the fact that traces of unitary
representations separate the conjugacy classes of finite groups (Cf
\cite{serr2}) and from the so-called CS-property satisfied by free groups (Cf
\cite{Stebe}): given two elements belonging to different conjugacy classes,
there exists a finite quotient of the group in which they are not conjugate.

Let us fix a base point $x_{0}$ in $X$ and a spanning tree $T$. An oriented
edge $(x,y)$ which is not in $T$ defines an element $\gamma_{x,y}$ of the
fundamental group $\Gamma_{x_{0}}$, with $\gamma_{y,x}=\gamma_{x,y}^{-1}$. For
eny edge $(u,v)\in T$, we set $\gamma_{u,v}=I$. For any discrete based loop
$l=(x_{1},x_{2},...,x_{p})$, set $\gamma_{l}=\gamma_{x_{1},x_{2}}%
...\gamma_{x_{p-1},x_{p}}\gamma_{x_{p},x_{1}}$. Then, if two based loops
$l_{1}$ and $l_{2}$ define distinct geodesic loops, there exists a finite
quotient $G=\Gamma_{x_{0}}/H$ of $\Gamma_{x_{0}}$ in which the classes of
their representatives $\gamma_{l_{i}}$ are not conjugate. Denote by
$\overline{\gamma}$ the class of $\gamma$ in $G$. Then there exists a unitary
representation $\rho$ of $G$ such that $Tr(\rho(\gamma_{l_{1}})\neq
Tr(\rho(\gamma_{l_{2}})$. Then take $U_{x,y}=\rho(\overline{\gamma}_{x,y})$.
We see that $\tau_{U}(l_{1})\neq\tau_{U}(l_{2})$.
\end{proof}

\bigskip

Again, by an argument similar to the one given for the occupation field, we
have:
\[
\mathbb{P}_{x,x}^{t}(\tau_{U}-1)=\frac{1}{d_{U}}\sum_{i=1}^{d_{U}}%
\exp(t(P^{(U)}-I))_{x,i}^{x,i}-\exp(t(P-I))_{x}^{x}.
\]
Integrating in $t$ after expanding, we get from the definition of $\mu$:%
\[
\int(\tau_{U}(l)-1)d\mu(l)=\sum_{k=1}^{\infty}\frac{1}{k}[\frac{1}{d_{U}%
}Tr((P^{(U)})^{k})-Tr((P)^{k})].
\]
We can extend $P$ into a matrix $P^{(I_{d_{U}})}=P\otimes I_{d_{U}}$\ indexed
by $X\times\{1,...,d_{U}\}$ by taking its tensor product with the identity on
$\mathbb{R}^{d_{U}}$.

Then:%
\[
\int(\tau_{U}(l)-1)d\mu(l)=\frac{1}{d_{U}}\sum_{k=1}^{\infty}\frac{1}%
{k}[Tr((P^{(U)})^{k})-Tr((P^{(I_{d_{U}})})^{k})].
\]
Hence, as in the case of the occupation field
\[
\int(\tau_{U}(l)-1)d\mu(l)=\frac{1}{d_{U}}\log(\det(V^{(U)})[V\otimes
I_{d_{U}}]^{-1})) =\frac{1}{d_{U}}\log(\det(G^{(U)}))-\log(\det(G))
\]
as $\det(G\otimes I_{d_{U}})=\det(G)^{d_{U}}$. \ 

Then, denoting $\mathcal{Z}_{e,U}$ the $\Big(          \frac{1}{d_{U}}\Big)
$-th power of the determinant of the $(\left|  X\right|  d_{U},\left|
X\right|  d_{U})$ matrix $G^{(U)}$ (well defined by remark \ref{complex}), the
formulas\ (\ref{F4})\ and (\ref{F5})\ extend easily to give the following

\begin{proposition}
\begin{enumerate}

\item[a)] $\int(e^{\sum N_{x,y}\log(\frac{C_{x,y}^{^{\prime}}}{C_{x,y}}%
)-\sum(\lambda_{x}^{^{\prime}}-\lambda_{x})\widehat{l}_{x}}\tau_{U}%
(l)-1)\mu_{e}(dl)=\log(\frac{\mathcal{Z}_{e^{\prime},U}}{\mathcal{Z}_{e}}).$

\item[b)] $\mathbb{E}(\prod_{x,y}[\frac{C_{x,y}^{\prime}}{C_{x,y}}%
]^{N_{x,y}(\mathcal{L}_{\alpha})}e^{-\sum(\lambda_{x}^{^{\prime}}-\lambda
_{x})\widehat{\mathcal{L}_{\alpha}}^{x}}\tau_{U}(\mathcal{L}_{\alpha
}))=(\frac{\mathcal{Z}_{e^{\prime},U}}{\mathcal{Z}_{e}})^{\alpha}.$
\end{enumerate}
\end{proposition}

Let us now introduce a new

\begin{definition}
\index{loop network}We say that sets $\Lambda_{i}$ of non-trivial loops are equivalent when the
associated occupation fields are equal and when the total traversal numbers
$\sum_{l\in\Lambda_{i}}N_{x,y}(l)$ are equal for all oriented edges $(x,y)$.
Equivalence classes will be called loop networks on the graph. We denote
$\overline{\Lambda}$ the loop network defined by $\Lambda$.

Similarly, a set $L$\ of non-trivial discrete loops defines a discrete network
characterized by the total traversal numbers.\bigskip
\end{definition}

The expectations computed in \ref{F5} determine the distribution of the
network $\overline{\mathcal{L}_{\alpha}}$\ defined by the loop ensemble
$\mathcal{L}_{\alpha}$. We will denote $B^{e,e^{\prime},\omega}(l)$ the
variables
\[
e^{\sum N_{x,y}(l)\log(\frac{C_{x,y}^{^{\prime}}}{C_{x,y}})-\sum(\lambda
_{x}^{^{\prime}}-\lambda_{x})\widehat{l}_{x}+i\int_{l}\omega}%
\]
and $B^{e,e^{\prime},\omega}(\mathcal{L}_{\alpha})$ the variables
\[
\prod_{l\in\mathcal{L}_{\alpha}}B^{e,e^{\prime},\omega}(l)=\prod
_{x,y}[\frac{C_{x,y}^{\prime}}{C_{x,y}}e^{i\omega_{x,y}}]^{N_{x,y}%
(\mathcal{L}_{\alpha})}e^{-\sum(\lambda_{x}^{^{\prime}}-\lambda_{x}%
)\widehat{\mathcal{L}_{\alpha}}^{x}}.
\]

More generally, we can define $B^{e,e^{\prime},U}(l)$ and $B^{e,e^{\prime}%
,U}(\mathcal{L}_{\alpha})$\ in a similar way as $B^{e,e^{\prime},\omega}(l)$
and $B^{e,e^{\prime},\omega}(\mathcal{L}_{\alpha})$, using $\tau_{U}%
(l)$\ instead of $e^{i\int_{l}\omega}$. Note that for each fixed $e$, when $U$
and $e^{\prime}$ vary with $\frac{C^{^{\prime}}}{C}\leq1$ and $\lambda
^{\prime}=\lambda$, linear combinations of the variables $B^{e,e^{\prime}%
,U}(\mathcal{L}_{\alpha})$\ form an algebra as $B^{e,e_{1}^{\prime},U_{1}%
}B^{e,e_{2}^{\prime},U_{2}}=B^{e,e_{1,2}^{\prime},U_{1}\otimes U_{2}}$, with
$C^{e_{1,2}^{\prime}}=\frac{C^{e_{1}^{\prime}}C^{e_{2}^{\prime}}}{C}$ . In
particular, $B^{e,e_{1}^{\prime},\omega_{1}}B^{e,e_{2}^{\prime},\omega_{2}%
}=B^{e,e_{1,2}^{\prime},\omega_{1}+\omega_{2}}$.

\begin{remark}
Note that the expectations of the variables $B^{e,e^{\prime},\omega
}(\mathcal{L}_{\alpha})$\ determine the law of the network $\overline{\mathcal{L}%
_{\alpha}}$\ defined by the loop ensemble $\mathcal{L}_{\alpha}$.
\end{remark}

To work with $\mu$, we should rather consider linear combinations of the form
${\sum\lambda_{i}(B^{e,e_{i}^{\prime},U_{i}}-1)}$, with $\sum\lambda_{i}=0$,
which form also an algebra.

\begin{remark}
Formulas (\ref{F4})\ and (\ref{F5})\ apply to the calculation of loop indices:
If we have for example a simple random walk on an oriented planar graph, and
if $z^{\prime}$ is a point\ of the dual graph $X^{\prime}$, $\omega
^{(z^{\prime})}$ can be chosen such that for any loop $l$, $\int_{l}%
\omega^{(z^{\prime})}$\ is the winding number of the loop\ around a given
point $z^{\prime}$\ of the dual graph $X^{\prime}$. Then $e^{i\pi\sum
_{l\in\mathcal{L}_{\alpha}}\int_{l}\omega^{(z^{\prime})}}$ is a spin system of
interest. We then get for example that%
\[
\mu\Big(          \int_{l}\omega_{z^{\prime}}\neq0\Big)
=-\frac{1}{2\pi}\int_{0}^{2\pi}\log(\det(G^{(2\pi u\omega^{(z^{\prime})}%
)}G^{-1}))du
\]
and hence%
\[
\mathbb{P(}\sum_{l\in\mathcal{L}_{\alpha}}|\int_{l}\omega^{(z^{\prime}%
)}|=0)=e^{\frac{\alpha}{2\pi}\int_{0}^{2\pi}\log(\det(G^{(2\pi u\omega
^{(z^{\prime})})}G^{-1}))du}.
\]
Conditional distributions of the occupation field with respect to values of
the winding number can also be obtained.
\end{remark}

\bigskip

\chapter{Decompositions}

Note first that with the energy $e$, we can associate a time-rescaled Markov
chain $\widehat{x}_{t}$\ in which holding times at any point $x$ are
exponential times of parameters $\lambda_{x}$: $\widehat{x}_{t}=x_{\tau_{t}}$
with $\tau_{t}=\inf(s,\;\int_{0}^{s}\frac{1}{\lambda_{x_{u}}}du=t)$. For the
time-rescaled Markov chain, local times coincide with the time spent in a
point and the duality measure is simply the counting measure. The potential
operator then essentially coincides with the Green function. The Markov loops
can be time-rescaled as well and we did it in fact already when we introduced
pointed loops. More generally we may introduce different holding time
parameters but it would be rather useless as the random variables we are
interested in are intrinsic, i.e. depend only on $e$.

\section{Traces of Markov chains and energy decomposition}

If $D\subset X$ and we set $F=D^{c}$, the orthogonal decomposition of the
energy $e(f,f)=e(f)$\ into $e^{D}(f-H^{F}f)+e(H^{F}f)$ \ (see proposition
\ref{proj}) leads to the decomposition of the Gaussian free field mentioned
above and also to a decomposition of the time-rescaled Markov chain into the
time-rescaled Markov chain killed at the exit of $D$ and its trace on $F$,
i.e. $\widehat{x}_{t}^{\{F\}}=\widehat{x}_{S_{t}^{F}}$, with $S_{t}^{F}%
=\inf(s,\int_{0}^{s}1_{F}(\widehat{x}_{u})du=t)$.

\begin{proposition}
\index{trace (of a Markov chain)}The trace of the time-rescaled Markov chain on $F$ is the time-rescaled Markov
chain defined by the energy functional $e^{\{F\}}(f)=e(H^{F}f)$ , for which
\[
C_{x,y}^{\{F\}}=C_{x,y}+\sum_{a,b\in D}C_{x,a}C_{b,y}[G^{D}]^{a,b},
\]%
\[
\lambda_{x}^{\{F\}}=\lambda_{x}-\sum_{a,b\in D}C_{x,a}C_{b,x}[G^{D}]^{a,b},
\]
and%
\[
\mathcal{Z}_{e}=\mathcal{Z}_{e^{D}}\mathcal{Z}_{e^{\{F\}}}.
\]
\end{proposition}

\begin{proof}
For the second assertion, note first that for any $y\in F$,
\[
\lbrack H^{F}]_{y}^{x}=1_{x=y}+1_{D}(x)\sum_{b\in D}[G^{D}]^{x,b}C_{b,y}.
\]
Moreover, $e(H^{F}f)=e(f,H^{F}f)$, by proposition \ref{proj} and therefore
\[
\lambda_{x}^{\{F\}}=e^{\{F\}}(1_{\{x\}})=e(1_{\{x\}},H^{F}1_{\{x\}}%
)=\lambda_{x}-\sum_{a\in D}C_{x,a}[H^{F}]_{x}^{a}=\lambda_{x}(1-p_{x}%
^{\{F\}})
\]
where $p_{x}^{\{F\}}=\sum_{a,b\in D}P_{a}^{x}[G^{D}]^{a,b}C_{b,x}=\sum_{a\in
D}P_{a}^{x}[H^{F}]_{x}^{a}$ is the probability that the Markov chain starting
at $x$ will first perform an excursion in $D$ and then return to $x$.

Then for distinct $x$ and $y$ in $F$,
\begin{align*}
C_{x,y}^{\{F\}}  &  =-e^{\{F\}}(1_{\{x\}},1_{\{y\}})=-e(1_{\{x\}}%
,H^{F}1_{\{y\}})\\
&  =C_{x,y}+\sum_{a}C_{x,a}[H^{F}]_{y}^{a}=C_{x,y}+\sum_{a,b\in D}%
C_{x,a}C_{b,y}[G^{D}]^{a,b}.
\end{align*}

Note that the graph defined on $F$ by the non-vanishing conductances
$C_{x,y}^{\{F\}}$ has in general more edges than the restiction to $F$ of the
original graph.

For the third assertion, note also that $G^{\{F\}}$ is the restriction of $G$
to $F$ as for all ${x,y\in F}$, $e^{\{F\}}(G\delta_{y|F},1_{\{x\}}%
)=e(G\delta_{y},[H^{F}1_{\{x\}}])=1_{\{x=y\}}$. Hence the determinant
decomposition already given in section \ref{hit} yields the formula. The cases
where $F$ has one point was considered as a special case in \ref{hit}.

For the first assertion note the transition matrix $[P^{\{F\}}]_{y}^{x}$ can
be computed directly and equals%

\[
P_{y}^{x}+\sum_{a,b\in D}P_{a}^{x}V^{D\cup\{x\}}]_{b}^{a}P_{y}^{b}=P_{y}%
^{x}+\sum_{a,b\in D}P_{a}^{x}[G^{D\cup\{x\}}]^{a,b}C_{b,y}.
\]
It can be decomposed according to whether the jump to $y$ occurs from $x$ or
from $D$ and the number of excursions from $x$ to $x$:
\begin{align*}
\lbrack P^{\{F\}}]_{y}^{x}  &  =\sum_{k=0}^{\infty}(\sum_{a,b\in D}P_{a}%
^{x}[V^{D}]_{b}^{a}P_{x}^{b})^{k}(P_{y}^{x}+\sum_{a,b\in D}P_{a}^{x}%
[V^{D}]_{b}^{a}P_{y}^{b})\\
&  =\sum_{k=0}^{\infty}(\sum_{a,b\in D}P_{a}^{x}[G^{D}]^{a,b}C_{b,x}%
)^{k}(P_{y}^{x}+\sum_{a,b\in D}P_{a}^{x}[G^{D}]^{a,b}C_{b,y}).
\end{align*}
The expansion of $\frac{C_{x,y}^{\{F\}}}{\lambda_{x}^{\{F\}}}$ in geometric
series yields exactly the same result.

Finally, remark that the holding times of $\widehat{x}_{t}^{\{F\}}$ at any
point $x\in F$\ are sums of a random number of independent holding times of
$\widehat{x}_{t}$. This random integer counts the excursions from $x$ to $x$
performed by the chain $\widehat{x}_{t}$ during the holding time of
$\widehat{x}_{t}^{\{F\}}$. It follows a geometric distribution of parameter
$1-p_{x}^{\{F\}}$. Therefore, $\frac{1}{\lambda_{x}^{\{F\}}}=\frac{1}%
{\lambda_{x}(1-p_{x}^{\{F\}})}$ is the expectation of the holding times of
$\widehat{x}_{t}^{\{F\}}$ at $x$.
\end{proof}

\section{Excursion theory}

\index{excursions}A loop in $X$ which hits $F$\ can be decomposed into a loop \ in $F$ and its
excursions in $D$ which may come back to their starting point.

More precisely, a loop $l$ hitting $F$\ can be decomposed into its restriction
$l^{\{F\}}=(\xi_{i},\widehat{\tau}_{i})$\ in $F$ (possibly a one point loop),
a family of excursions $\gamma_{\xi_{i},\xi_{i+1}}$\ attached to the jumps of
$l^{\{F\}}$ and systems of i.i.d. excursions $(\gamma_{\xi_{i}}^{h},h\leq
n_{\xi_{i}})$ attached to the points of $l^{\{F\}}$. These sets of excursions
can be empty.

Let $\mu_{D}^{a,b}$\ denote the bridge measure (with mass $[G^{D}]^{a,b}%
$)\ associated with $e^{D}$.\newline Set
\[
\nu_{x,y}^{D}=\frac{1}{C_{x,y}^{\{F\}}}[C_{x,y}\delta_{\emptyset}+\sum_{a,b\in
D}C_{x,a}C_{b,y}\mu_{D}^{a,b}],\quad\nu_{x}^{D}=\frac{1}{\lambda_{x}%
p_{x}^{\{F\}}}(\sum_{a,b\in D}C_{x,a}C_{b,x}\mu_{D}^{a,b})
\]
and note that $\nu_{x,y}^{D}(1)=\nu_{x}^{D}(1)=1$.

Let $\mu^{D}$\ be the restriction of $\mu$ to loops in contained in $D$. It is
the loop measure associated to the process killed at the exit of $D$. We get a
decomposition of $\mu-\mu^{D}$ in terms of the loop measure $\mu^{\{F\}}%
$\ defined on loops of $F$ by the trace of the Markov chain on $F$,
probability measures $\nu_{x,y}^{D}$\ on excursions in $D$ indexed by pairs of
points in $F$ and $\nu_{x}^{D}\ $on excursions in $D$\ indexed by points of
$F$. Moreover, conditionally on $l^{\{F\}}$, the integers $n_{\xi_{i}}$ follow a Poisson distribution of
parameter $\lambda_{\xi_{i}}^{\{F\}}\widehat{\tau}_{i}$ (the total holding
time in $\xi_{i}$ before another point of $F$ is visited) and the conditional
distribution of the rescaled holding times in $\xi_{i}$ before each excursion
$\gamma_{\xi_{i}}^{l}$ is the distribution
$\beta_{n_{\xi_{i}},\widehat{\tau_{i}}}$\ of the increments of a uniform
sample of $n_{\xi_{i}}$ points in $[0\;\widehat{\tau}_{i}]$ put in increasing
order. We denote these holding times by $\widehat{\tau}_{i,h}$ and set
$l=\Lambda(l^{\{F\}},(\gamma_{\xi_{i},\xi_{i+1}}),(n_{\xi_{i}},\gamma_{\xi
_{i}}^{h},\widehat{\tau}_{i,h}))$.

Then $\mu-\mu^{D}$ is the image measure by $\Lambda$ of
\[
\mu^{\{F\}}(dl^{\{F\}})\prod(\nu_{\xi_{i},\xi_{i+1}}^{D})(d\gamma_{\xi_{i}%
,\xi_{i+1}})\prod e^{-\lambda_{\xi_{i}}^{\{F\}}\widehat{\tau}_{i}}\sum
_{k}\frac{[\lambda_{\xi_{i}}^{\{F\}}\widehat{\tau}_{i}]^{k}}{k!}\delta
_{n_{\xi_{i}}}^{k}[\nu_{x}^{D}]^{\otimes k}(d\gamma_{\xi_{i}}^{h}%
)\beta_{k,\widehat{\tau}_{i}}(d\widehat{\tau}_{i,h}).
\]

Note that for$\ x,y$ belonging to $F$, the bridge measure $\mu^{x,y}$ can be
decomposed in the same way, with the same excursion measures.

\paragraph{\textbf{The one point case and the excursion measure}}

If $F$ is reduced to a point $x_{0}$, and $\kappa$ vanishes on $D=\{x_{0}%
\}^{c}$, the decomposition is of course simpler.

First, $\lambda_{x_{0}}=\sum_{a}C_{x_{0},a}+\kappa_{x}$ and $\lambda_{x_{0}%
}^{\{x_{0}\}}=\kappa_{x_{0}}$. Then,
\[
p_{x_{0}}^{\{x_{0}\}}=\sum_{a,b\in D}P_{a}^{x_{0}}[G^{D}]^{a,b}C_{b,x_{0}%
}=\frac{\sum C_{x_{0},a}}{\lambda_{x_{0}}}=1-\frac{\kappa_{x_{0}}}%
{\lambda_{x_{0}}},
\]
as $C_{.,x_{0}}$ is the killing measure of $e^{D}$ and therefore its $G^{D}$
potential equals $1$.

$l^{\{x_{0}\}}$ is a trivial one point loop with rescaled lifetime
$\widehat{\tau}=\widehat{l}^{x_{0}}\frac{\lambda_{x_{0}}}{\kappa_{x_{0}}%
}=\frac{\widehat{l}^{x_{0}}}{1-p_{x_{0}}^{\{x_{0}\}}}$\ and the number of
excursions (all independent with the same distribution $\rho_{x_{0}}%
^{\{x_{0}\}^{c}}$) follows a Poisson distribution of parameter $\kappa_{x_{0}%
}\widehat{\tau}=\lambda_{x_{0}}\widehat{l}^{x_{0}}$.

The non-normalized excursion measure $\rho^{D}=(\lambda_{x_{0}}-\kappa_{x_{0}%
})\nu_{x_{0}}^{D}=\sum_{a,b\in D}C_{x_{0},a}C_{b,x_{0}}\mu_{D}^{a,b}$ verifies
the following property: for any subset $K$ of $D$,%
\[
\rho^{D}(\{\gamma,\widehat{\gamma}(K)>0\})=Cap_{e^{D}}(K).
\]
Indeed, the lefthand side can be expressed as%
\[
\sum_{a,b\in D}C_{x_{0},a}C_{b,x_{0}}[H^{K}G^{D}]^{a,b}=\sum_{a\in D}%
C_{x_{0},a}[H^{K}1]^{a}=e^{D}(1,H^{K}1).
\]
It should be noted that $\rho^{D}$ depends only of $e^{D}$ (i.e. does not
depend on $\kappa_{x_{0}}$).

\begin{proposition}
\label{rod}

\begin{enumerate}
\item[a)] Under $\rho^{D}$, the non-normalized hitting distribution of any
$K\subseteq D$ is the $e^{D}$-capacitary measure of $K$. The same property
holds for the last hitting distribution.

\item[b)] Under $\rho^{D}(d\gamma)$, the conditional distribution of the path
$\gamma$ between $T_{K}(\gamma)$\ (the first time in $K$) and $L_{K}(\gamma
)$\ (the last time in $K)$, given $\gamma_{T_{K}}$\ and $\gamma_{L_{K}}$\ is
$\frac{1}{[G^{D}]^{\gamma_{T_{K}},\gamma_{L_{K}}}}\mu_{D}^{\gamma_{T_{K}%
},\gamma_{L_{K}}}$
\end{enumerate}
\end{proposition}

\begin{proof}
\begin{enumerate}

\item[a)] By definition of $\rho^{D}$, the non-normalized hitting distribution
of $K$ is expressed for any $z\in K$ by $\sum_{a,b,c\in D}C_{x_{0}%
,a}C_{b,x_{0}}[G^{D-K}]^{a,c}C_{c,z}[G^{D}]^{z,b}=\sum_{a,c\in D}C_{x_{0}%
,a}[G^{D-K}]^{a,c}C_{c,z}$. $C_{x_{0},a},a\in D$ is the killing measure of
$e^{D}$ and $[G^{D-K}]^{a,c}C_{c,z}$ the $e^{D}$-balayage kernel on $K$. The
case of last hitting distribution follows from the invariance of $\rho^{D}$
under time reversal.

\item[b)] Indeed, on functions of a path after $T_{K}$,
\[
\mu_{D}^{a,b}=\sum_{a,c\in D,z\in K}C_{x_{0},a}[G^{D-K}]^{a,c}C_{c,z}\mu
_{D}^{z,b}%
\]
and therefore on functions of a path restricted to $[T_{K},L_{K]}$, $\rho^{D}$
equals:
\begin{multline*}
\sum_{z\in K}\sum_{a,b,c\in D}C_{x_{0},a}[G^{D-K}]^{a,c}C_{c,z}\mu_{D}%
^{z,b}C_{b,x_{0}}\\
=\sum_{z,t\in K}\sum_{a,b,c,d\in D}C_{x_{0},a}[G^{D-K}]^{a,c}C_{c,z}\mu
_{D}^{z,t}C_{t,d}[G^{D-K}]^{d,b}C_{b,x_{0}}.
\end{multline*}
\end{enumerate}
\end{proof}

\begin{remark}
This construction of $\rho^{D}$\ can be extended to transient chains on
infinite spaces with zero killing measure. There exists a unique measure on
equivalence classes under the shift of doubly infinite paths converging to
infinity on both sides, such that the hitting distribution of any compact set
is given by its capacitary measure (Cf \cite{Hunt}, \cite{Weil}, \cite{Silv},
and the first section of \cite{Sznit} for a recent presentation in the case of
$\mathbb{Z}^{d}$ random walks). Proposition \ref{rod} holds also in this context.

Following \cite{Sznit}, the set of points hit by a Poissonian set of
excursions of intensity $\alpha\rho^{D}$ can be called the interlacement at
level $\alpha$.\index{interlacement}
\end{remark}

The law $\frac{\mu^{x_{0},x_{0}}}{G^{x_{0},x_{0}}}$ can of course be
decomposed in the same way, with the same conditional distribution given
$\widehat{l}^{x_{0}}$. Recall that by proposition \ref{mul},\ $\widehat
{l}^{x_{0}}$ follows an exponential distribution with mean $G^{x_{0},x_{0}}$.

\ $\widehat{\mathcal{L}}_{\alpha}^{x_{0}}$ follows a $\Gamma(\alpha
,G^{x_{0},x_{0}})$ distribution, in particular an exponential distribution
with mean $G^{x_{0},x_{0}}$ for $\alpha=1$. Moreover, the union of the
excursions of all loops of $\mathcal{L}_{a}$ outside $x_{0}$ has obviously the
same Poissonian conditional distribution, given $\widehat{\mathcal{L}}%
_{\alpha}^{x_{0}}=s$ than $\mu$ and $\frac{\mu^{x_{0},x_{0}}}{G^{x,x}}$, given
$\widehat{l}^{x_{0}}=s$. The set of excursions outside $x_{0}$ defined by the
$\frac{\mu^{x_{0},x_{0}}}{G^{x_{0},x_{0}}}$-distributed bridge and by
$\mathcal{L}_{1}$\ are therefore identically distributed, as the total holding
time in $x_{0}$.

\begin{remark}
\label{split}Note finally that by exercise \ref{poidir}, the distribution of
$\mathcal{L}_{1}/\mathcal{L}_{1}^{D}$ can be recovered from a unique sample of
$\frac{\mu^{x_{0},x_{0}}}{G^{x_{0},x_{0}}}$ by splitting the bridge according
to an independent sample $U_{i}$\ of $Poisson-Dirichlet(0,\alpha)$, more
precisely, by splitting the bridge (in fact a based loop) $l$ into based
subloops $l|_{[\sigma_{i},\sigma_{i+1}]}$, with $\sigma_{i}=\inf
(s,\frac{1}{\lambda_{x}}\int_{0}^{s}1_{\{x_{0}\}}(l_{s})ds=\sum_{1}^{i}%
U_{j}\widehat{l}^{x_{0}})$.
\end{remark}

Conversely, a sample of the bridge could be recovered from a sample of the
loop set $\mathcal{L}_{1}/\mathcal{L}_{1}^{D}$\ by concatenation in random
order. This random ordering can be defined by taking a projective limit of the
randomly ordered finite subset of loops $\{l_{i,n}\}$ defined by assuming for
example that $\widehat{l}_{i,n}^{x}>\frac{1}{n}$.

\section{Conditional expectations}

Coming back to the general case, the Poisson process $\mathcal{L}_{\alpha
}^{\{F\}}=\{l^{\{F\}},l\in\mathcal{L}_{\alpha}\}$ has intensity $\mu^{\{F\}}$
and is independent of $\mathcal{L}_{\alpha}^{D}$.

Note that $\widehat{\mathcal{L}_{\alpha}^{\{F\}}}$ is the restriction of
$\widehat{\mathcal{L}_{\alpha}}$ to $F$.

If $\chi$ is carried by $D$ and if we set $e_{\chi}=e+\left\|  \quad\right\|
_{L^{2}(\chi)}$ and denote $[e_{\chi}]^{\{F\}}$ by $e^{\{F,\chi\}}$ we have
\[
C_{x,y}^{\{F,\chi\}}=C_{x,y}+\sum_{a,b}C_{x,a}C_{b,y}[G_{\chi}^{D}%
]^{a,b},\quad p_{x}^{\{F,\chi\}}=\sum_{a,b\in D}P_{a}^{x}[G_{\chi}^{D}%
]^{a,b}C_{b,x}%
\]
and $\lambda_{x}^{\{F,\chi\}}=\lambda_{x}(1-p_{x}^{\{F,\chi\}})$.

More generally, if $e^{\#}$ is such that $C^{\#}=C$\ on $F\times F$, and
$\lambda=\lambda^{\#}$ on $F$ we have:
\[
C_{x,y}^{\#\{F\}}=C_{x,y}+\sum_{a,b}C_{x,a}^{\#}C_{b,y}^{\#}[G^{\#D}%
]^{a,b},\quad p_{x}^{\#\{F\}}=\sum_{a,b\in D}P_{a}^{\#x}[G^{\#D}]^{a,b}%
C_{b,x}^{\#}%
\]
and $\lambda_{x}^{\#\{F\}}=\lambda_{x}(1-p_{x}^{\#\{F\}}).$

If $\chi$ is a measure carried by $D$, we have:%
\begin{align*}
\mathbb{E}(e^{-\left\langle \widehat{\mathcal{L}_{\alpha}},\chi\right\rangle
}|\mathcal{L}_{\alpha}^{\{F\}})  &  =\mathbb{E}(e^{-\left\langle
\widehat{\mathcal{L}_{\alpha}^{D}},\chi\right\rangle })(\prod_{x,y\in F}[\int
e^{-\left\langle \widehat{\mathcal{\gamma}},\chi\right\rangle }\nu_{x,y}%
^{D}(d\gamma)]^{N_{x,y}(\mathcal{L}_{\alpha}^{\{F\}})}\\
&  \times\prod_{x\in F}e^{\lambda_{x}^{\{F\}}[\widehat{\mathcal{L}_{\alpha
}^{\{F\}}}]^{x}\int(e^{-\left\langle \widehat{\mathcal{\gamma}},\chi
\right\rangle }-1)\nu_{x}^{D}(d\gamma)}\\
&  =[\frac{\mathcal{Z}_{e_{\chi}^{D}}}{\mathcal{Z}_{e^{D}}}]^{\alpha}%
(\prod_{x,y\in F}[\frac{C_{x,y}^{\{F,\chi\}}}{C_{x,y}^{\{F\}}}]^{N_{x,y}%
(\mathcal{L}_{\alpha}^{\{F\}})}\prod_{x\in F}e^{[\lambda_{x}^{\{F,\chi
\}}-\lambda_{x}^{\{F\}}]\widehat{\mathcal{L}_{\alpha}^{x}}}.
\end{align*}
(recall that $\widehat{\mathcal{L}_{\alpha}^{\{F\}}}$ is the restriction of
$\widehat{\mathcal{L}_{\alpha}}$ to $F$). Also, if we condition on the set of
discrete loops $\mathcal{DL}_{\alpha}^{\{F\}}$%
\[
\mathbb{E}(e^{-\left\langle \widehat{\mathcal{L}_{\alpha}},\chi\right\rangle
}|\mathcal{DL}_{\alpha}^{\{F\}})=[\frac{\mathcal{Z}_{e_{\chi}^{D}}%
}{\mathcal{Z}_{e^{D}}}]^{\alpha}(\prod_{x,y\in F}[\frac{C_{x,y}^{\{F,\chi\}}%
}{C_{x,y}^{\{F\}}}]^{N_{x,y}(\mathcal{L}_{\alpha}^{\{F\}})}\prod_{x\in
F}[\frac{\lambda_{x}^{\{F\}}}{\lambda_{x}^{\{F,\chi\}}}]^{N_{x}(\mathcal{L}%
_{\alpha}^{\{F\}})+1})
\]
where the last exponent $N_{x}+1$ is obtained by taking into account the loops
which have a trivial trace on $F$ (see formula (\ref{unpt})).

More generally we can show in the same way the following

\begin{proposition}
\label{dec}If $C^{\#}=C$\ on $F\times F$, and $\lambda=\lambda^{\#}$ on $F$,
we denote $B^{e,e^{\#}}$ the multiplicative functional
\[
{\prod_{x,y}[\frac{C_{x,y}^{\#}}{C_{x,y}}]^{N_{x,y}}e^{-\sum_{x\in D}%
\widehat{l_{x}}(\lambda_{x}^{\#}-\lambda_{x})}}.
\]
Then,
\[
\mathbb{E}(B^{e,e^{\#}}|\mathcal{L}_{\alpha}^{\{F\}})=[\frac{\mathcal{Z}%
_{e^{\#D}}}{\mathcal{Z}_{e^{D}}}]^{\alpha}(\prod_{x,y\in F}[\frac{C_{x,y}%
^{\#\{F\}}}{C_{x,y}^{\{F\}}}]^{N_{x,y}(\mathcal{L}_{\alpha}^{\{F\}})}%
\prod_{x\in F}e^{[\lambda_{x}^{\#\{F\}}-\lambda_{x}^{\{F\}}]\widehat
{\mathcal{L}_{\alpha}^{x}}}%
\]
and
\[
\mathbb{E}(B^{e,e^{\#}}|\mathcal{DL}_{\alpha}^{\{F\}})=[\frac{\mathcal{Z}%
_{e^{\#D}}}{\mathcal{Z}_{e^{D}}}]^{\alpha}(\prod_{x,y\in F}[\frac{C_{x,y}%
^{\#\{F\}}}{C_{x,y}^{\{F\}}}]^{N_{x,y}(\mathcal{L}_{\alpha}^{\{F\}})}%
\prod_{x\in F}[\frac{\lambda_{x}^{\{F\}}}{\lambda_{x}^{\#\{F\}}}%
]^{N_{x}(\mathcal{L}_{\alpha}^{\{F\}})+1}.
\]
\end{proposition}

\medskip These decomposition and conditional expectation formulas extend to
include a current $\omega$ in $C^{\#}$. Note that if $\omega$ is closed (i.e.
vanish on every loop) in $D$, one can define $\omega^{F}$ such that
$[Ce^{i\omega}]^{\{F\}}=C^{\{F\}}e^{i\omega^{F}}$. Then%
\[
\mathcal{Z}_{e,\omega}=\mathcal{Z}_{e^{D}}\mathcal{Z}_{e^{\{F\}},\omega^{F}}.
\]
The previous proposition implies the following \textsl{Markov property}:

\begin{remark}
If $D=D_{1}\cup D_{2}$ with $D_{1}$ and $D_{2}$ strongly disconnected, (i.e.
such that for any $(x,y,z)\in D_{1}\times D_{2}\times F$, $C_{x,y}$ and
$C_{x,z}C_{y,z}$ vanish), the restrictions of the network $\overline
{\mathcal{L}_{\alpha}}$ to $D_{1}\cup F$ and $D_{2}\cup F$ are independent
conditionally on the restriction of $\mathcal{L}_{\alpha}$\ to $F$.
\end{remark}

\begin{proof}
This follows from the fact that as $D_{1}$ and $D_{2}$\ are strongly
disconnected, any excursion measure $\nu_{x,y}^{D}$ or $\rho_{x}^{D}$\ from
$F$ into $D=D_{1}\cup D_{2}$ is an excursion measure either in $D_{1}$ or in
$D_{2}$.
\end{proof}

\section{Branching processes with immigration}

An interesting example can be given after extending slightly the scope of the
theory to countable transient symmetric Markov chains: We can take
$X=\mathbb{N}-\{0\}$, $C_{n,n+1}=1$ for all $n\geq1$, \ $\kappa_{n}=0$ for
$n\geq2$ and $\kappa_{1}=1$. $P$ is the transfer matrix of the simple
symmetric random walk killed at $0$.

Then we can apply the previous considerations to check that $\widehat
{\mathcal{L}}_{\alpha}^{n}$ is a branching process with immigration.

The immigration at level $n$ comes from the loops whose infimum is $n$ and the
branching from the excursions to level $n+1$ of the loops existing at level
$n$. Set $F_{n}=\{1,2,...,n\}$ and $D_{n}=F_{n}^{c}$.

From the calculations of conditional expectations made above, we get that for
any positive parameter $\gamma$,%
\[
\mathbb{E}(e^{-\gamma\widehat{\mathcal{L}}_{\alpha}^{n}}||\mathcal{L}_{\alpha
}^{\{F_{n-1}\}})=\mathbb{E}(e^{-\gamma\lbrack\widehat{\mathcal{L}_{\alpha
}^{D_{n-1}}}]^{n}})e^{[\lambda_{n-1}^{\{F_{n-1},\gamma\delta_{n}\}}%
-\lambda_{n-1}^{\{F_{n-1}\}}]\widehat{\mathcal{L}}_{\alpha}^{n-1}}%
\]
($[\widehat{\mathcal{L}_{\alpha}^{D_{n-1}}}]^{n}$ denotes the occupation field
of the trace of $\mathcal{L}_{\alpha}$ on $D_{n-1}$ evaluated at $n$).

From this formula, it is clear that $\widehat{\mathcal{L}}_{\alpha}^{n}$ is a
branching Markov chain with immigration. To be more precise, note that for any
$n,m>0$, the potential operator $V_{m}^{n}$ equals $2(n\wedge m)$ that
$\lambda_{n}=2$ and that $G^{1,1}=1$. Moreover, by the generalized resolvent
equation, $G_{\gamma\delta_{1}}^{1,n}=G^{1,n}-G^{1,1}\gamma G_{\gamma
\delta_{1}}^{1,n}$ so that $G_{\gamma\delta_{1}}^{1,n}=\frac{1}{1+\gamma}$.
For any $n>0$, the restriction of the Markov chain to $D_{n}$ is isomorphic to
the original Markov chain. Then it comes that for all $n$, $p_{n}^{\{F_{n}%
\}}=\frac{1}{2}$, $\lambda_{n}^{\{F_{n}\}}=1$, and $\lambda_{n}^{\{F_{n}%
,\gamma\delta_{n+1}\}}=2-\frac{1}{1+\gamma}=\frac{2\gamma+1}{1+\gamma}$\ so
that the Laplace exponent of the convolution semigroup $\nu_{t}$\ defining the
branching mechanism $\lambda_{n-1}^{\{F_{n-1},\gamma\delta_{n}\}}%
-\lambda_{n-1}^{\{F_{n-1}\}}$equals $\frac{2\gamma+1}{1+\gamma}-1=\frac{\gamma
}{1+\gamma}=\int(1-e^{-\gamma s})e^{-s}ds$. It is the semigroup of a compound
Poisson process whose Levy measure is exponential.

The immigration law (on $\mathbb{R}^{+}$) is a Gamma distribution
$\Gamma(\alpha,G^{1,1})=\Gamma(\alpha,1)$. It is the law of $\widehat
{\mathcal{L}}_{\alpha}^{1}$ and also of $[\widehat{\mathcal{L}_{\alpha
}^{D_{n-1}}}]^{n}$\ for all $n>1$.

The conditional law of $\widehat{\mathcal{L}}_{\alpha}^{n+1}$ given
$\widehat{\mathcal{L}}_{\alpha}^{n}$ is the convolution of the immigration law
$\Gamma(\alpha,1)$ with $\nu_{\widehat{\mathcal{L}}_{\alpha}^{n}}$

\begin{exercise}
Alternatively, we can consider the integer valed process $N_{n}(\mathcal{L}%
_{\alpha}^{\{F_{n}\}})+1$ which is a Galton Watson process with immigration.
In our example, we find the reproduction law $\pi(n)=2^{-n-1}$for all $n\geq0$
(critical binary branching).
\end{exercise}

\begin{exercise}
Show that more generally, if $C_{n,n+1}=[\frac{p}{1-p}]^{n}$, for $n>0$ and
$\kappa_{1}=1,$with $0<p<1$, we get all asymetric simple random walks. Show
that $\lambda_{n}=\frac{p^{n-1}}{(1-p)^{n}}$ and $G^{1,1}=1$. Determine the
distributions of the associated branching and Galton Watson process with immigration.
\end{exercise}

If we consider the occupation field defined by the loops whose infimum equals $1$ (I.e. going through $1$), we
get a branching process without immigration: it is the classical relation
between random walks local times and branching processes.

\section{Another expression for loop hitting distributions}

Let us come back to formula \ref{F1F2}. Setting $F=F_{1}\cup F_{2}$, we see
that this result involves only $\mu^{\{F\}}$ and $e^{\{F\}}$ i.e. it can be
expressed interms of the restrictions of the loops to $F$.

\begin{lemma}
If $X=X_{1}\cup X_{2}$ with $X_{1}\cap X_{2}=\emptyset$,
\[
\log(\frac{\det(G)}{\det(G^{X_{1}})\det(G^{X_{2}})})=\sum_{1}^{\infty}%
\frac{1}{2k}Tr([H_{12}H_{21}]^{k}+[H_{12}H_{21}]^{k})
\]
with $H_{12}=H^{X_{2}}|_{X_{1}}$ and $H_{21}=H^{X_{1}}|_{X_{2}}$.

\begin{proof}%
\begin{align*}
\frac{\det(G)}{\det(G^{X_{1}})\det(G^{X_{2}})}  &  =\left(  \det\Big(
\begin{array}
[c]{cc}%
I_{X_{1}\times X_{1}} & -G^{X_{1}}C_{X_{1}\times X_{2}}\\
-G^{X_{2}}C_{X_{2}\times X_{1}} & I_{X_{2}\times X_{2}}%
\end{array}
\Big)  \right)  ^{-1}\\
&  =\left(  \det\Big(
\begin{array}
[c]{cc}%
I_{X_{1}\times X_{1}} & -H_{12}\\
-H_{21} & I_{X_{2}\times X_{2}}%
\end{array}
\Big)  \right)  ^{-1}.
\end{align*}

The transience implies that either $H_{12}1$, either $H_{21}1$ is strictly
less than $1$, and therefore, $H_{12}H_{21}$ and $H_{12}H_{21}$ are strict
contractions. From the expansion of $-\log(1-x)$, we get that:
\[
\log\Big(          \frac{\det(G)}{\det(G^{X_{1}})\det(G^{X_{2}})}\Big)
=\sum_{1}^{\infty}\frac{1}{k}Tr\left[  \Big(
\begin{array}
[c]{cc}%
0 & -H_{12}\\
-H_{21} & 0
\end{array}
\Big)          ^{k}\right]  .
\]
The result follows, as odd terms have obviously zero trace.
\end{proof}
\end{lemma}

Noting finally that the lemma can be applied to the restrictions of $G$ to
$F_{1}\cup F_{2}$, $F_{1}$ and $F_{2}$, and that hitting distributions of
$F_{1}$ from $F_{2}$ and $F_{2}$ from $F_{1}$\ are the same for the Markov
chain on $X$ and its restriction to $F_{1}\cup F_{2}$, we get finally:

\begin{proposition}
\label{F1FF}If $F_{1}$ and $F_{2}$ are disjoint,
\[
\mu(\widehat{l(}F_{1})\widehat{l(}F_{2})>0)=\sum_{1}^{\infty}\frac{1}%
{2k}Tr([H_{12}H_{21}]^{k}+[H_{12}H_{21}]^{k})
\]
with $H_{12}=H^{F_{2}}|_{F_{1}}$ and $H_{21}=H^{F_{1}}|_{F_{2}}$.
\end{proposition}

\begin{exercise}
\label{bfk}Show that the $k$-th term of the expansion can be interpreted as
the measure of loops with exactly $k$-crossings between $F_{1}$ and $F_{2}$.
\end{exercise}

\begin{exercise}
Prove analogous results for $n$ disjoint sets $F_{i}$.
\end{exercise}

\chapter{Loop erasure and spanning trees.}

\section{Loop erasure}

Recall that an oriented link $g$ is a pair of points $(g^{-},g^{+})$ such that
$C_{g}=C_{g^{-},g^{+}}\neq0$. Define $-g=(g^{+},g^{-})$.

Let $\mu_{\neq}^{x,y}$ be the measure induced by $C$ on discrete self-avoiding
paths between $x$ and $y$: $\mu_{\neq}^{x,y}(x,x_{2},...,x_{n-1}%
,y)=C_{x,x_{2}}C_{x_{1},x_{3}}...C_{x_{n-1},y}$.

\index{loop erasure}Another way to define a measure on discrete self avoiding paths from $x$ to
$y$ from a measure on paths from $x$ to $y$ is loop erasure defined in section
\ref{secred}\ (see also \cite{Law} ,\cite{Quianmp}, \cite{Law2} and
\cite{Marchal}). In this context, the loops, which can be reduced to points,
include holding times, and loop erasure produces a discrete path without
holding times.

We have the following:

\begin{theorem}
The image of $\mu^{x,y}$ by the loop erasure map $\gamma\rightarrow\gamma
^{BE}$ is $\mu_{BE}^{x,y}$ defined on self avoiding paths by $\mu_{BE}%
^{x,y}(\eta)=\mu_{\neq}^{x,y}(\eta)\frac{\det(G)}{\det(G^{\{\eta\}^{c}})}%
=\mu_{\neq}^{x,y}(\eta)\det(G_{|\{\eta\}\times\{\eta\}})$ (Here $\{\eta\}$
denotes the set of points in the path $\eta$) and by $\mu_{BE}^{x,y}%
(\emptyset)=\delta_{y}^{x}G_{x,x}$
\end{theorem}

\begin{proof}
Set $\eta=(x_{1}=x,x_{2},...,x_{n}=y)$ and $\eta_{m}=(x,...,x_{m})$, for any
$m>1$. Then,%
\[
\mu^{x,y}(\gamma^{BE}=\eta)=\sum_{k=0}^{\infty}[P^{k}]_{x}^{x}P_{x_{2}}^{x}%
\mu_{\{x\}^{c}}^{x_{2},y}(\gamma^{BE}=\theta\eta)
\]
where $\mu_{\{x\}^{c}}^{x_{2},y}$ denotes the bridge measure for the Markov
chain killed as it hits $x$ and $\theta$ the natural shift on discrete paths.
By recurrence, this clearly equals
\[
V_{x}^{x}P_{x_{2}}^{x}[V^{\{x\}^{c}}]_{x_{2}}^{x_{2}}...[V^{\{\eta_{n-1}%
\}^{c}}]_{x_{n-1}}^{x_{n-1}}P_{y}^{x_{n-1}}[V^{\{\eta\}^{c}}]_{y}^{y}%
\lambda_{y}^{-1}=\mu_{\neq}^{x,y}(\eta)\frac{\det(G)}{\det(G^{\{\eta\}^{c}})}%
\]
as
\[
\lbrack V^{\{\eta_{m-1}\}^{c}}]_{x_{m}}^{x_{m}}=\frac{\det([(I-P]|_{\{\eta
_{m}\}^{c}\times\{\eta_{m}\}^{c}})}{\det([(I-P]|_{\{\eta_{m-1}\}^{c}%
\times\{\eta_{m-1}\}^{c}})}=\frac{\det(V^{\{\eta_{m-1}\}^{c}})}{\det
(V^{\{\eta_{m}\}^{c}})}=\frac{\det(G^{\{\eta_{m-1}\}^{c}})}{\det(G^{\{\eta
_{m}\}^{c}})}\lambda^{x_{m}}.
\]
for all $m\leq n-1$.
\end{proof}

\begin{remark}
It is worth noticing that, also the operation of loop erasure clearly depends
on the orientation of the path (as shown in the picture below), the
distribution of the loop erased bridge is reversible.
\end{remark}

\begin{center}
\includegraphics[scale=0.4]{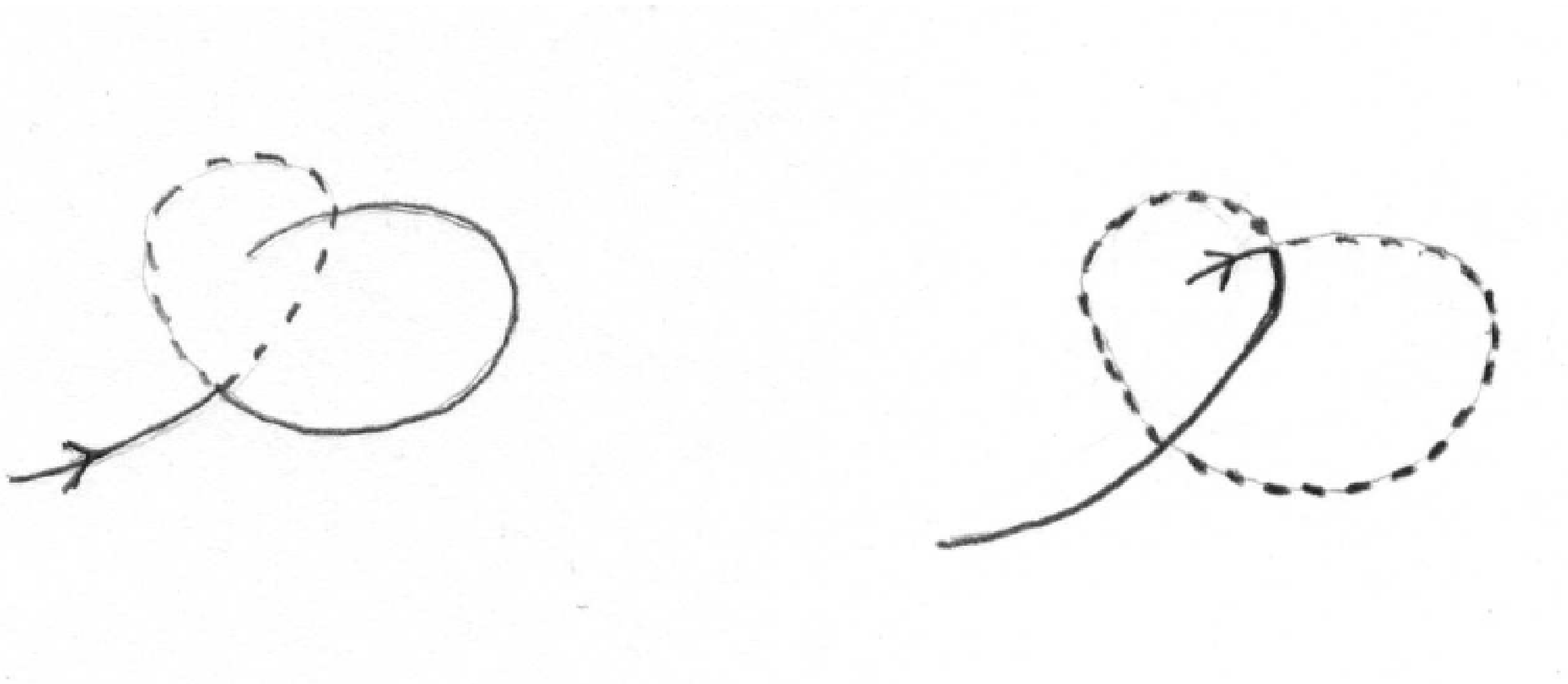}  \newline 
\end{center}

Also, by Feynman-Kac formula, for any self-avoiding path $\eta$:
\begin{align*}
\int e^{-<\widehat{\gamma},\chi>}1_{\{\gamma^{BE}=\eta\}}\mu^{x,y}(d\gamma)
&  =\frac{\det(G_{\chi})}{\det(G_{\chi}^{\{\eta\}^{c}})}\mu_{\neq}^{x,y}%
(\eta)=\det(G_{\chi})_{|\{\eta\}\times\{\eta\}}\mu_{\neq}^{x,y}(\eta)\\
&  =\frac{\det(G_{\chi})_{|\{\eta\}\times\{\eta\}}}{\det(G_{|\{\eta
\}\times\{\eta\}})}\mu_{BE}^{x,y}(\eta).
\end{align*}

Therefore, recalling that by the results of section \ref{hit}\ conditionally
on $\eta$, $\mathcal{L}_{1}/\mathcal{L}_{1}^{\{\eta\}^{c}}$ and $\mathcal{L}%
_{1}^{\{\eta\}^{c}}$ are independent, we see that under $\mu^{x,y}$, the
conditional distribution of $\widehat{\gamma}$ given $\gamma^{BE}=\eta$ is the
distribution of $\widehat{\mathcal{L}}_{1}-\widehat{\mathcal{L}}_{1}%
^{\{\eta\}^{c}}\mathcal{\ }$i.e. the occupation field of the loops of
$\mathcal{L}_{1}$\ which intersect $\eta$.

More generally, it can be shown that

\begin{proposition}
\label{be}The conditional distribution of the network $\overline
{\mathcal{L}_{\gamma}}$ defined by the loops of $\gamma$, given that
$\gamma^{BE}=\eta$, is identical to the distribution of the network defined by
$\mathcal{L}_{1}/\mathcal{L}_{1}^{\{\eta\}^{c}}$ i.e. the loops of
$\mathcal{L}_{1}$\ which intersect $\eta$.
\end{proposition}

\begin{proof}
Recall the notation $\mathcal{Z}_{e}=\det(G)$. First an elementary calculation
using (\ref{d})\ shows that $\mu_{e^{\prime}}^{x,y}(e^{i\int_{\gamma}\omega
}1_{\{\gamma^{BE}=\eta\}})$ equals
\begin{multline*}
\mu_{e}^{x,y}\Big(  1_{\{\gamma^{BE}=\eta\}}\prod[\frac{C_{\xi_{i},\xi_{i+1}%
}^{\prime}}{C_{\xi_{i},\xi_{i+1}}}e^{i\omega_{\xi_{i},\xi_{i+1}}}%
\frac{\lambda_{\xi_{i}}}{\lambda_{\xi_{i}}^{\prime}}]\Big) \\
\frac{C_{x,x_{2}}^{\prime}C_{x_{1},x_{3}}^{\prime}...C_{x_{n-1},y}^{\prime}%
}{C_{x,x_{2}}C_{x_{1},x_{3}}...C_{x_{n-1},y}}e^{i\int_{\eta}\omega}\mu
_{e}^{x,y}\Big(  \prod_{u\neq v}[\frac{C_{u,v}^{\prime}}{C_{u,v}}%
e^{i\omega_{u,v}}]^{N_{u,v}(\mathcal{L}_{\gamma})}e^{-\left\langle
\lambda^{^{\prime}}-\lambda,\widehat{\mathcal{\gamma}}\right\rangle
}1_{\{\gamma^{BE}=\eta\}}\Big)  .
\end{multline*}
(Note the term $e^{-\left\langle \lambda^{^{\prime}}-\lambda,\widehat
{\mathcal{\gamma}}\right\rangle }$ can be replaced by $\prod_{u}%
(\frac{\lambda_{u}}{\lambda_{u}^{\prime}})^{N_{u}(\gamma)+1}$).\newline
Moreover, by the proof of the previous proposition, applied to the Markov
chain defined by $e^{\prime}$ perturbed by $\omega$, we have also
\[
\mu_{e^{\prime}}^{x,y}(e^{i\int_{\gamma}\omega}1_{\{\gamma^{BE}=\eta
\}})=C_{x,x_{2}}^{\prime}C_{x_{1},x_{3}}^{\prime}...C_{x_{n-1},y}^{\prime
}e^{i\int_{\eta}\omega}\frac{\mathcal{Z}_{e^{\prime},\omega}}{\mathcal{Z}%
_{[e^{\prime}]^{\{\eta\}^{c}},\omega}}.
\]
Therefore,%
\[
\mu_{e}^{x,y}(\prod_{u\neq v}[\frac{C_{u,v}^{\prime}}{C_{u,v}}e^{i\omega
_{u,v}}]^{N_{u,v}(\mathcal{L}_{\gamma})}e^{-\left\langle \lambda^{^{\prime}%
}-\lambda,\widehat{\mathcal{\gamma}}\right\rangle }|\gamma^{BE}=\eta
)=\frac{\mathcal{Z}_{e^{\{\eta\}^{c}}}\mathcal{Z}_{e^{\prime},\omega}%
}{\mathcal{Z}_{e}\mathcal{Z}_{[e^{\prime}]^{\{\eta\}^{c}},\omega}}.
\]
Moreover, by (\ref{F5}) and the properties of the Poisson processes,%
\[
\mathbb{E}(\prod_{u\neq v}[\frac{C_{u,v}^{\prime}}{C_{u,v}}e^{i\omega_{u,v}%
}]^{N_{u,v}(\mathcal{L}_{1}/\mathcal{L}_{1}^{\{\eta\}^{c}})}e^{-\left\langle
\lambda^{^{\prime}}-\lambda,\widehat{\mathcal{L}}_{1}-\widehat{\mathcal{L}%
}_{1}^{\{\eta\}^{c}}\right\rangle })=\frac{\mathcal{Z}_{e^{\{\eta\}^{c}}%
}\mathcal{Z}_{e^{\prime},\omega}}{\mathcal{Z}_{e}\mathcal{Z}_{[e^{\prime
}]^{\{\eta\}^{c}},\omega}}.
\]
It follows that the joint distribution of the traversal numbers and the
occupation field are identical for the set of erased loops and $\mathcal{L}%
_{1}/\mathcal{L}_{1}^{\{\eta\}^{c}}$.
\end{proof}

\medskip

The general study of loop erasure which is done in this chapter yields the
following result when applied to a universal covering $\widehat{X}$. Let
$\widehat{G}$ be the Green function associated with the lift of the Markov chain.

\begin{corollary}
The image of $\mu^{x,y}$ under the reduction map is given as follows: If $c$
is a geodesic arc between $x$ and $y$: $\mu^{x,y}(\{\xi,\xi^{R}=c\})=\prod
C_{c_{i},c_{i+1}}\det(\widehat{G}_{|\{c\}\times\{c\}})$.

Besides, if $\ \widehat{x}$ and $\widehat{y}$ are the endpoints of the lift of
$c$ to a universal covering,
\[
\mu^{x,y}(\{\xi,\xi^{R}=c\})=\widehat{G}_{\widehat{x},\widehat{y}}.
\]
\end{corollary}

Note this yields an interesting identity on the Green function $\widehat{G}$.

\begin{exercise}
Check it in the special case treated in proposition \ref{greencov}.
\end{exercise}
Similarly one can define the image of $\mathbb{P}^{x}$ by $BE$ and check it is
given by
\begin{align*}
\mathbb{P}_{BE}^{x}(\eta)  &  =\delta_{x_{1}}^{x}C_{x_{1},x_{2}}%
...C_{x_{n-1},x_{n}}\kappa_{x_{n}}\det(G_{|\{\eta\}-\Delta\times
\{\eta\}-\Delta})\\
&  =\delta_{x_{1}}^{x}C_{x_{1},x_{2}}...C_{x_{n-1},x_{n}}\kappa_{x_{n}%
}\frac{\det(G)}{\det(G^{\{\eta\}^{c}})}%
\end{align*}
for $\eta=(x_{1},...,x_{n},\Delta)$.

Note that in particular, $\mathbb{P}_{BE}^{x}((x,\Delta))=V_{x}^{x}(1-\sum
_{y}P_{y}^{x})=\kappa_{x}G^{x,x}$.

Slightly more generally, que can determine the law of the image, by loop
erasure path killed at it hits a subset $F$, the hitting point being now the
end point of the loop erased path (instead of $\Delta$, unless $F$ is not hit
during the lifetime of the path). If $x\in D=F^{c}$ is the starting point and
$y\in F^{\Delta}$, the probability of $\eta=(x_{1},...,x_{n},y)$ is
\[
\delta_{x_{1}}^{x}C_{x_{1},x_{2}}...C_{x_{n-1},x_{n}}C_{x_{n},y}%
\det(G_{|\{\eta\}-y\times\{\eta\}-y}^{D})=\delta_{x_{1}}^{x}C_{x_{1},x_{2}%
}...C_{x_{n-1},x_{n}}C_{x_{n},y}\frac{\det(G^{D})}{\det(G^{D-\{\eta\}})}%
\]

\section{Wilson algorithm}

\index{spanning tree}Wilson's algorithm (see \cite{Lyo2}) iterates this last construction, starting
with the points $x$ arranged in an arbitrary order. The first step of the
algorithm is the construction of a loop erased path starting at the first
point and ending at $\Delta.$ This loop erased path is the first branch of the
spanning tree. Each step of the algorithm reproduces this first step except it
starts at the first point which is not visited by the already constructed tree
of self avoiding paths, and stops when it hits that tree, or $\Delta$,
producing a new branch of the tree. This algorithm provides a construction,
branch by branch, of a random spanning tree rooted in $\Delta$. It turns out,
as we will show below, that the distribution of this spanning tree is very
simple, and does not depend on the ordering chosen on $X$.%

\begin{figure}
[ptb]
\begin{center}
\includegraphics[
height=2.9776in,
width=3.691in
]%
{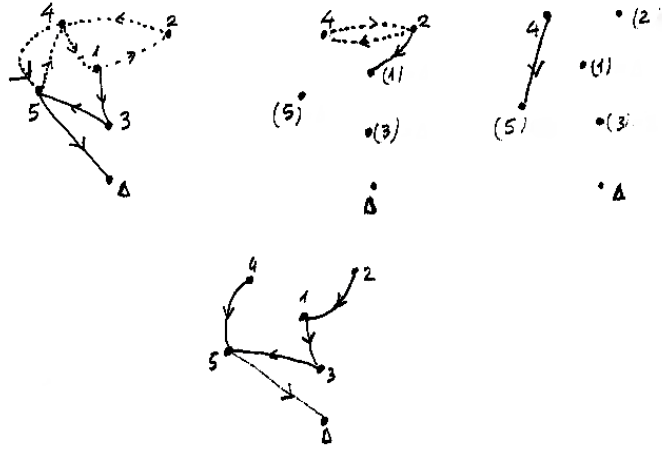}%
\caption{Wilson's algorithm}%
\end{center}
\end{figure}
\bigskip

This law is a probability measure $\mathbb{P}_{ST}^{e}$\ on the set
$ST_{X,\Delta}$\ of spanning trees of $X$ rooted at the cemetery point
$\Delta$ defined by the energy $e$. The weight attached to each oriented link
$g=(x,y)$ of $X\times X$ is the conductance and the weight attached to the
link $(x,\Delta)$ is $\kappa_{x}$ which we can also denote by $C_{x,\Delta}$.
As the determinants simplify in the iteration, the probability of a tree
$\Upsilon$ is given by a simple formula:%

\begin{equation}
\mathbb{P}_{ST}^{e}(\Upsilon)=\mathcal{Z}_{e}\prod_{\xi\in\Upsilon}C_{\xi}
\label{span}%
\end{equation}
It is clearly independent of the ordering chosen initially. Now note that,
since we get a probability
\begin{equation}
\mathcal{Z}_{e}\sum_{\Upsilon\in ST_{X,\Delta}}\prod_{(x,y)\in\Upsilon}%
C_{x,y}\prod_{x,(x,\Delta)\in\Upsilon}\kappa_{x}=1 \label{normtree}%
\end{equation}
or equivalently
\[
\sum_{\Upsilon\in ST_{X,\Delta}}\prod_{(x,y)\in\Upsilon}P_{y}^{x}%
\prod_{x,(x,\Delta)\in\Upsilon}P_{\Delta}^{x}=\frac{1}{\prod_{x\in X}%
\lambda_{x}\mathcal{Z}_{e}}%
\]
Then, it follows that, for any $e^{\prime}$ for which conductances (including
$\kappa^{\prime}$) are positive only on links of $e$,%
\[
\mathbb{E}_{ST}^{e}\left(  \prod_{(x,y)\in\Upsilon}\frac{P_{y}^{\prime x}%
}{P_{y}^{x}}\prod_{x,(x,\Delta)\in\Upsilon}\frac{P_{\Delta}^{\prime x}%
}{P_{\Delta}^{x}}\right)  =\frac{\prod_{x\in X}\lambda_{x}}{\prod_{x\in
X}\lambda_{x}^{\prime}}\frac{\mathcal{Z}_{e}}{\mathcal{Z}_{e^{\prime}}}%
\]
and%
\begin{equation}
\mathbb{E}_{ST}^{e}\left(  \prod_{(x,y)\in\Upsilon}\frac{C_{x,y}^{\prime}%
}{C_{x,y}}\prod_{x,(x,\Delta)\in\Upsilon}\frac{\kappa_{x}^{\prime}}{\kappa
_{x}}\right)  =\frac{\mathcal{Z}_{e}}{\mathcal{Z}_{e^{\prime}}}. \label{rel}%
\end{equation}

Note also that in the case of a graph (i.e. when all conductances are equal to
$1$), all spanning trees have the same probability. The expression of their
cardinal as the determinant $\mathcal{Z}_{e}$\ is known as Cayley's theorem
(see for example \cite{Lyo2}).

The formula (\ref{rel}) shows a kind of duality between random spanning trees
and $\mathcal{L}_{1}$. It can be extended to $\mathcal{L}_{k}$ for any integer
$k$ if we consider the sum (in terms of number of transitions) of $k$
independent spanning trees.
\begin{exercise}
\label{restree}Show that more generally, for any tree $T$\ rooted in $\Delta$,

$\mathbb{P}_{ST}^{e}(\{\Upsilon,\;T\subseteq\Upsilon\})=\det(G_{|\{T\}-\Delta
\times\{T\}-\Delta})\prod_{\xi\in Edges(T)}C_{\xi}$, $\{T\}$ denoting the
vertex set of $T$.

(As usual, $C_{x,\Delta}=\kappa_{x}$. Hint: Run Wilson's algorithm starting
from the leaves of $T$)
\end{exercise}

\begin{exercise}
Using exercise \ref{greencomp}, prove Cayley's Theorem: the complete graph
$K_{n}$ has $n^{n-2}$ spanning trees.
\end{exercise}

The following result follows easily from proposition \ref{be}.

\begin{corollary}
The network defined by the random set of loops $\mathcal{L}_{W}$\ constructed
in this algorithm is independent of the random spanning tree, and independent
of the ordering. It has the same distribution as the network defined by the
loops of $\mathcal{L}_{1}$.
\end{corollary}

\begin{remark}
Note that proposition \ref{be}\ and its corollary can be made more precise
with the help of remark \ref{split}. The splitting procedure used there with
the help of an auxiliary independent set of Poisson Dirichlet variables allows
to reconstruct the set of loops $\mathcal{L}_{1}/\mathcal{L}_{1}^{\{x\}^{c}}%
$\ by splitting the first erased loop in the proof of the proposition.
Iterating the procedure we can successively reconstruct all sets
$\mathcal{L}_{1}^{\{\eta_{m}\}^{c}}/\mathcal{L}_{1}^{\{\eta_{m+1}\}^{c}}$\ and
finally $\mathcal{L}_{1}/\mathcal{L}_{1}^{\{\eta\}^{c}}$.\ Then, by Wilson
algorithm, we can reconstruct $\mathcal{L}_{1}$.
\end{remark}

\bigskip

Let us now consider the \emph{recurrent case.}

A probability is defined on the non oriented spanning trees by the
conductances: $\mathbb{P}_{ST}^{e}((\mathcal{T})$ is defined by the product of
the conductances of the edges of $\mathcal{T}$ normalized by the sum of these
products on all spanning trees.

Note that any non oriented spanning tree of $X$ along edges of $E$ defines
uniquely an oriented spanning tree $I_{\Delta}(\mathcal{T})$ if we choose a
root $\Delta$. The orientation is taken towards the root which can be viewed
as a cemetery point. Then, if we consider the associated Markov chain killed
as it hits $\Delta$ defined by the energy form $e^{_{\{\Delta\}^{c}}}$, the
previous construction yields a probability $\mathbb{P}_{ST}^{e^{_{\{\Delta
\}^{c}}}}$\ on spanning trees rooted at $\Delta$ which by (\ref{span})
coincides with the image of $\mathbb{P}_{ST}^{e}$ by $I_{\Delta}$. This
implies in particular that the normalizing factor $\mathcal{Z}_{e_{\{\Delta
\}^{c}}}$ is independent of the choice of $\ \Delta$ as it has to be equal to
$(\sum_{T\in ST_{X}}\prod_{\{x,y\}\in T}C_{x,y})^{-1}$. We denote it by
$\mathcal{Z}_{e}^{0}$. This factor can also be expressed in terms of the
recurrent Green operator $G$. Recall it is defined as a scalar product on
measures of zero mass. The determinant of $G$ is defined as the determinant of
its matrix in any orthonormal basis of this hyperplane, with respect to the
natural Euclidean scalar product.

Recall that for any $x\neq\Delta$, $G(\varepsilon_{x}-\varepsilon_{\Delta
})=-\frac{\left\langle \lambda,G^{\{\Delta\}^{c}}\varepsilon_{x}\right\rangle
}{\lambda(X)}+G^{\{\Delta\}^{c}}\varepsilon_{x}$. Therefore, for any
$y\neq\Delta$, $\left\langle \varepsilon_{y}-\varepsilon_{\Delta
},G(\varepsilon_{x}-\varepsilon_{\Delta})\right\rangle =[G^{\{\Delta\}^{c}%
}]^{x,y}$.

The determinant of the matrix $[G^{\{\Delta\}^{c}}]$, equal to $\mathcal{Z}%
_{e}^{0}$, is therefore also the determinant of $G$ in the basis $\{\delta
_{x}-\delta_{\Delta},x\neq\Delta\}$ which is not orthonormal with repect to
the natural euclidean scalar product. An easy calculation shows it equals
\[
\det\Big(          \left\langle \delta_{y}-\delta_{\Delta},\delta_{x}%
-\delta_{\Delta}\right\rangle _{\mathbb{R}^{\left|  X\right|  }},x,y\neq
\Delta\Big)       \det(G)=\left|  X\right|  \det(G).
\]

\begin{exercise}
Prove that if we set $\alpha_{x_{0}}(\mathcal{T})=\prod_{(x,y)\in I_{x_{0}%
}(\mathcal{T})}P_{y}^{x}$ then $\sum_{T\in ST_{X}}\alpha_{x_{0}}(\mathcal{T}%
)$\ is proportional to $\lambda_{x_{0}}$ as $x_{0}$ varies in $X$. More
precisely, it equals $K\lambda_{x_{0}}$, with $K=\frac{\mathcal{Z}_{e}^{0}%
}{\prod_{x\in X}\lambda_{x}}$. This fact is known as the matrix-tree theorem
(\cite{Lyo2}).
\end{exercise}

\begin{exercise}
Check directly that $\mathcal{Z}_{e_{\{x_{0}\}^{c}}}$ is independent of the
choice of $\ x_{0}.$
\end{exercise}
\begin{exercise}
Given a spanning tree $\mathcal{T}$, we say a subset $A$ is wired iff the
restriction of $\mathcal{T}$\ to $A$ is a tree.

a) Let $\widetilde{e}_{A}$ be the recurrent energy form defined on $A$ by the
conductances $C$. Show that $\mathbb{P}_{ST}^{e}(A$ is
wired$)=\frac{\mathcal{Z}_{e}^{0}}{\mathcal{Z}_{e_{A^{c}}}\mathcal{Z}%
_{\widetilde{e}_{A}}^{0}}$ (Hint: Choose a root in $A$. Then use exercise
\ref{restree}\ and identity \ref{normtree}).

b) Show that under $\mathbb{P}_{ST}^{e}$, given that $A$ is wired, the
restriction of the spanning tree of $X$ to $A$ and the spanning tree of
$A^{c}\cup\{\Delta\}$ obtained by rooting at an external point $\Delta$ the
spanning forest induced on $A^{c}$ by restriction of the spanning tree are
independent, with distributions respectively given by $\mathbb{P}%
_{ST}^{\widetilde{e}_{A}}$\ and $\mathbb{P}_{ST}^{e_{A^{c}}}$.

c) Conversely, given such a pair, the spanning tree of $X$ can be recovered by
attaching to $A$\ the roots $y_{i}$\ of the spanning forest of $A^{c}$
independently, according to the distributions $\frac{C_{y_{i},u}}{\sum_{u\in
A}C_{y_{i},u}},\;u\in A$.
\end{exercise}

\section{The transfer current theorem}

Let us come back to the transient case by choosing some root $x_{0}=\Delta$.
As by the strong Markov property, $V_{x}^{y}=\mathbb{P}_{y}(T_{x}<\infty
)V_{x}^{x}$, we have $\frac{G^{y,x}}{G^{x,x}}=\frac{V_{x}^{y}}{V_{x}^{x}%
}=\mathbb{P}_{y}(T_{x}<\infty)$, and therefore
\[
\mathbb{P}_{ST}^{e}((x,y)\in\Upsilon)=\mathbb{P}_{x}(\gamma_{1}^{BE}%
=y)=V_{x}^{x}P_{y}^{x}\mathbb{P}^{y}(T_{x}=\infty)=C_{x,y}G^{x,x}%
(1-\frac{G^{x,y}}{G^{x,x}}).
\]
Directly from the above, we recover Kirchhoff's theorem:%
\begin{align*}
\mathbb{P}_{ST}^{e}(\pm(x,y)\in\Upsilon)  &  =C_{x,y}[G^{x,x}(1-\frac{G^{x,y}%
}{G^{x,x}})+G^{y,y}(1-\frac{G^{y,x}}{G^{y,y}})]\\
&  =C_{x,y}(G^{x,x}+G^{y,y}-2G^{x,y})=C_{x,y}K^{x,y),(x,y)}%
\end{align*}

with the notation introduced in section \ref{difforms}, and this is clearly
independent of the choice of the root.

\begin{exercise}
Give an alternative proof of Kirchhoff's theorem\ by using (\ref{rel}), taking
$C_{x,y}^{\prime}=sC_{x,y}$ and $C_{u,v}^{\prime}=C_{u,v}$ for $\{u,v\}\neq
\{x,y\}$.
\end{exercise}

In order to go further, it is helpful to introduce some elements of exterior
algebra. Recall that in any vector space $E$, in terms of the ordinary tensor
product $\otimes$, the skew symmetric tensor product $v_{1}\wedge v_{2}%
\wedge...\wedge v_{n}$\ of $n$ vectors $v_{1}...v_{n}$ is defined as
$\frac{1}{\sqrt{n!}}\sum_{\sigma\in\mathcal{S}_{n}}(-1)^{m(\sigma)}%
v_{\sigma(1)}\otimes...\otimes v_{\sigma(n)}$. They generate the $n$-th skew
symmetric tensor power of $E$, denoted $E^{\wedge n}$. Obviously,
$v_{\sigma(1)}\wedge...\wedge v_{\sigma(n)}=(-1)^{m(\sigma)}v_{1}\wedge
v_{2}\wedge...\wedge v_{n}$. If the vector space is equipped with a scalar
product $\left\langle .,.\right\rangle $, it extends to tensors and
$\left\langle v_{1}\wedge v_{2}\wedge...\wedge v_{n},v_{1}^{\prime}\wedge
v_{2}^{\prime}\wedge...\wedge v_{n}^{\prime}\right\rangle =\det(\left\langle
v_{i},v_{j}^{\prime}\right\rangle )$.

\bigskip

The following result, which generalizes Kirchoff's theorem, is known as the
transfer current theorem (see for example \cite{Lyo},\ \cite{Lyo2}):

\begin{theorem}
$\mathbb{P}_{ST}^{e}(\pm\xi_{1},...\pm\xi_{k}\in\Upsilon)=(\prod_{1}^{k}%
C_{\xi_{i}})\det(K^{\xi_{i},\xi_{j}}\;1\leq i,j\leq k)$.
\end{theorem}

\noindent Note this determinant does not depend on the orientation of the links.\newline 

\begin{proof}
Note first that if $\Upsilon$ is a spanning tree rooted in $x_{0}=\Delta$ and
$\xi_{i}=(x_{i-1},x_{i}),$ $1\leq i\leq\left|  X\right|  -1$ are its oriented
edges, the measures $\delta_{x_{i}}-\delta_{x_{i-1}}$ form another basis of
the euclidean hyperplane\ of signed measures with zero charge, which has the
same determinant as the basis $\delta_{x_{i}}-\delta_{x_{0}}$.

Therefore, $\mathcal{Z}_{e}^{0}$\ is also the determinant of the matrix of
$G$\ in this basis, i.e.$\ $%
\[
\mathcal{Z}_{e}^{0}=\det(K^{\xi_{i},\xi_{j}}\;1\leq i,j\leq\left|  X\right|
-1)
\]
and
\begin{multline*}
\mathbb{P}_{ST}^{e}(\Upsilon)=(\prod_{1}^{\left|  X\right|  -1}C_{\xi_{i}%
})\det(K^{\xi_{i},\xi_{j}}\;1\leq i,j\leq\left|  X\right|  -1)\\
=\det(\sqrt{C_{\xi_{i}}}K^{\xi_{i},\xi_{j}}\sqrt{C_{\xi_{j}}}\;1\leq
i,j\leq\left|  X\right|  -1).
\end{multline*}

Recall that $\sqrt{C_{\xi_{i}}}K^{\xi_{i},\xi_{j}}\sqrt{C_{\xi_{j}}%
}=\left\langle \alpha_{\xi_{i}}^{\ast}|\Pi|\alpha_{\xi_{j}}^{\ast
}\right\rangle _{\mathbb{A}_{-}}$, where $\Pi$ denotes the projection on the
space of differentials and that $\alpha_{(\eta)}^{\ast x,y}=\pm\frac{1}%
{\sqrt{C_{\eta}}}$ if $(x,y)=\pm(\eta)$ and $=0$ elsewhere.

To finish the proof of the theorem, it is helpful to use the exterior algebra.
Note first that for any ONB $e_{1},...,e_{\left|  X\right|  -1}$ of the space
of differentials, $\Pi\alpha_{\xi}^{\ast}=\sum\left\langle \alpha_{\xi}^{\ast
}|e_{j}\right\rangle e_{j}$ and $\mathbb{P}_{ST}^{e}(\Upsilon)=\det
(\left\langle \alpha_{\xi_{i}}^{\ast}|e_{j}\right\rangle )^{2}=\left\langle
\alpha_{\xi_{1}}^{\ast}\wedge...\wedge\alpha_{\xi_{\left|  X\right|  -1}%
}^{\ast}|e_{1}\wedge...\wedge e_{\left|  X\right|  -1}\right\rangle
_{\bigwedge^{\left|  X\right|  -1}\mathbb{A}_{-}}^{2}$. Therefore
\begin{multline*}
\mathbb{P}_{ST}^{e}(\xi_{1},,...,\xi_{k}\in\Upsilon)\\
=\sum_{\eta_{k+1},...\eta_{\left|  X\right|  -1}}\left\langle \alpha_{\xi_{1}%
}^{\ast}\wedge...\wedge\alpha_{\xi_{k}}^{\ast}\wedge\alpha_{\eta_{k+1}}^{\ast
}\wedge...\wedge\alpha_{\eta_{\left|  X\right|  -1}}^{\ast}|e_{1}%
\wedge...\wedge e_{\left|  X\right|  -1}\right\rangle _{\bigwedge^{\left|
X\right|  -1}\mathbb{A}_{-}}^{2}%
\end{multline*}
where the sum is on all edges $\eta_{k+1},...,\eta_{\left|  X\right|  -1}$
completing $\xi_{1},,...,\xi_{k}$ into a spanning tree. It can be extended to
all systems of distinct $q=\left|  X\right|  -1-k$ edges $\eta^{\prime}%
=\{\eta_{1}^{\prime},...,\eta_{q}^{\prime}\}$ as all the additional term
vanish. Indeed, an exterior product of $\alpha_{\xi_{1}}^{\ast}$ vanishes as
soon as they form a loop. Hence the expression above equals:
\[
\sum_{\eta^{\prime}}(\sum_{i_{1}<...<i_{k}}\varepsilon_{i_{1}...i_{k}%
}\left\langle \alpha_{\xi_{1}}^{\ast}\wedge...\wedge\alpha_{\xi_{k}}^{\ast
}|e_{i_{1}}\wedge...\wedge e_{i_{k}}\right\rangle \left\langle \alpha
_{\eta_{1}^{\prime}}^{\ast}\wedge...\wedge\alpha_{\eta_{q}^{\prime}}^{\ast
}|e_{i_{1}^{\prime}}\wedge...\wedge e_{i_{q}^{\prime}}\right\rangle )^{2}%
\]
where the $i_{l}^{\prime}$ are the indices complementing $i_{1},...,i_{k}$ put
in increasing order and $\varepsilon_{i_{1}...i_{k}}=(-1)^{(i_{1}%
-1)...(i_{k}-k)}$. Recalling that the $\alpha_{\xi}^{\ast}$ form an
orthonormal base of $\mathbb{A}_{-}$, we see that the sum in $\eta^{\prime}$
of each mixed term in the square vanishes and
\[
\sum_{\eta^{\prime}}\left\langle \alpha_{\eta_{1}^{\prime}}^{\ast}%
\wedge...\wedge\alpha_{\eta_{q}^{\prime}}^{\ast}|e_{i_{1}^{\prime}}%
\wedge...\wedge e_{i_{q}^{\prime}}\right\rangle ^{2}=1.
\]
Hence we obtain finally:
\[
\sum_{i_{1}<i_{2}<...<i_{k}}\left\langle \alpha_{\xi_{1}}^{\ast}%
\wedge...\wedge\alpha_{\xi_{k}}^{\ast}|e_{i_{1}}\wedge...\wedge e_{i_{k}%
}\right\rangle _{\bigwedge^{k}\mathbb{A}_{-}}^{2}=\det(\sqrt{C_{\xi_{i}}%
}K^{\xi_{i},\xi_{j}}\sqrt{C_{\xi_{j}}}\;1\leq i,j\leq k).
\]
\end{proof}

It follows that given any function $g$ on non oriented links,
\begin{align*}
\mathbb{E}_{ST}^{e}(e^{-\sum_{\xi\in\Upsilon}g(\xi)}) &  =\mathbb{E}_{ST}%
^{e}(\prod_{\xi}(1+(e^{-g(\xi)}-1)1_{\xi\in\Upsilon})\\
&  =1+\sum_{k=1}^{\left|  E\right|  }\sum_{\pm\xi_{1}\neq\pm\xi_{2}\neq
...\neq\pm\xi_{k}}\prod(e^{-g(\xi_{i})}-1)\mathbb{P}_{ST}^{e}(\pm\xi
_{1},...,\pm\xi_{k}\in\Upsilon)\\
&  =1+\sum_{k}\sum_{\pm\xi_{1}\neq\pm\xi_{2}\neq...\neq\pm\xi_{k}}%
\prod(e^{-g(\xi_{i})}-1)\det(K^{\xi_{i},\xi_{j}}\;1\leq i,j\leq k)\\
&  =1+\sum Tr((M_{C(e^{-g}-1)}K)^{\wedge k})=\det(I+KM_{C(e^{-g}-1)})
\end{align*}
and we have proved the following

\begin{proposition}
$\mathbb{E}_{ST}^{e}(e^{-\sum_{\xi\in\Upsilon}g(\xi)})=\det(I-M_{\sqrt
{C(1-e^{-g})}}KM_{\sqrt{C(1-e^{-g})}}).$
\end{proposition}

Here determinants are taken on matrices indexed by $E$.\newline 

\begin{remark}
This is an example of the Fermi point processes (also called determinantal
point processes) discussed in \cite{Sosh} and \cite{ShiTak}. It is determined
by the matrix $M_{\sqrt{C}}KM_{\sqrt{C}}$. Note that it follows also easily
from the previous proposition that the set of edges which do not belong to the
spanning tree \ also form a Fermi point process defined by the matrix
$I-M_{\sqrt{C}}KM_{\sqrt{C}}$.

In particular,under $\mathbb{P}_{ST}$, the set of points $x$\ such that
$(x,\Delta)\in\Upsilon$ (i.e. the set of points directly connected to the root
$\Delta$) is a Fermi point process the law of which is determined by the
matrix $Q^{x,y}=\sqrt{\kappa_{x}}G^{x,y}\sqrt{\kappa_{y}}$.

For example, if $X$ is an interval of $\mathbb{Z}$, with $C_{x,y}=0$ iff
$\left|  x-y\right|  >1$, it is easily verified that for $x<y<z$,%
\[
Q^{x,z}=\frac{Q^{x,y}Q^{y,z}}{Q^{y,y}}%
\]

Then using the remark following theorem 6 in \cite{Sosh}, we see that the
spacings of this point process are independent.

The edges of $X$ which do not belong to the spanning tree form a determinantal
process of edges, of the same type, intertwinned with the points connected to
$\Delta$.
\end{remark}

\bigskip

A consequence is that for any spanning tree $T$, if
$\pi_{T}$\ denotes $M_{1_{\{T\}}}$ (the multiplication by the indicator
function of $T$), it follows from the above, by letting $g$ be $m1_{\{T^{c}%
\}}$, $m\rightarrow\infty$ that
\[
\mathbb{P}_{ST}^{e}(T)=\det((I-KM_{C})(I-\pi_{T})+\pi_{T})=\det((I-KM_{C}%
)_{T^{c}\times T^{c}}).
\]

Another consequence is that if $e^{\prime}$ is another energy form on the same
graph,
\[
\mathbb{E}_{ST}^{e}(\prod_{(x,y)\in\Upsilon}\frac{C_{x,y}^{\prime}}{C_{x,y}%
})=\det(I-M_{\sqrt{C-C^{\prime}}}KM_{\sqrt{C-C^{\prime}}%
}).
\]
On the other hand, from (\ref{rel}),\ it also equals $(\sum_{T\in ST_{X}}%
\prod_{\{x,y\}\in T}C_{x,y})\mathcal{Z}_{e}^{0}=\frac{\mathcal{Z}_{e}^{0}%
}{\mathcal{Z}_{e^{\prime}}^{0}}$ so that finally%
\[
\frac{\mathcal{Z}_{e}^{0}}{\mathcal{Z}_{e^{\prime}}^{0}}=\det(I-M_{\sqrt{C-C^{\prime}}}KM_{\sqrt{C-C^{\prime}}}).
\]
Note that indicators of distinct individual edges are negatively correlated.
More generally:

\begin{theorem}
(Negative association) Given any sets disjoint of edges $E_{1}$ and $E_{2}$,
\[
\mathbb{P}_{ST}^{e}(E_{1}\cup E_{2}\subseteq\Upsilon)\leq\mathbb{P}_{ST}%
^{e}(E_{1}\subseteq\Upsilon)\mathbb{P}_{ST}^{e}(E_{2}\subseteq\Upsilon).
\]
\end{theorem}

\begin{proof}
Denote by $K^{\#}(i,j)$ the restriction of $K^{\#}=(\sqrt{C_{\xi}}K^{\xi,\eta
}\sqrt{C_{\eta}},\;\xi,\eta\in E)$ to $E_{i}\times E_{j}$.\ Then,
\[
\frac{\mathbb{P}_{ST}^{e}(E_{1}\cup E_{2}\subseteq\Upsilon)}{\mathbb{P}%
_{ST}^{e}(E_{1}\subseteq\Upsilon)\mathbb{P}_{ST}^{e}(E_{2}\subseteq\Upsilon
)}=\frac{\det(K^{\#})}{\det(K^{\#}(2,2))\det(K^{\#}(2,2))}=\det\left(
\left[
\begin{array}
[c]{cc}%
I & F\\
F^{\ast} & I
\end{array}
\right]  \right)
\]
with $F=K^{\#}(1,1)^{-\frac{1}{2}}K^{\#}(1,2)K^{\#}(2,2)^{-\frac{1}{2}}$%

Finally, note that $\log(\det\left(  \left[
\begin{array}
[c]{cc}%
I & F\\
F^{\ast} & I
\end{array}
\right]  \right)  )=Tr(\log\left(  \left[
\begin{array}
[c]{cc}%
I & F\\
F^{\ast} & I
\end{array}
\right]  \right)  )$

$=-\sum_{1}^{\infty}\frac{1}{2k}Tr((FF^{\ast})^{k})\leq0$.
\end{proof}
\begin{remark}
\label{derivZ}Note that it follows directly from the expression of $P_{ST}$
and from the transfer current theorem that for any set of disjoint edges
$\xi_{1},...,\xi_{k}$:%
\[
\lbrack\mathcal{Z}_{e}^{0}]^{-1}\frac{\partial^{k}}{\partial C_{\xi_{1}%
}...\partial C_{\xi_{k}}}[\mathcal{Z}_{e}^{0}]^{-1}=\det(K_{\xi_{i},\xi_{j}%
},1\leq i,j\leq k).
\]
\end{remark}

\begin{proof}
Note that
\[
\frac{\partial k}{\partial C_{\xi_{1}}...\partial C_{\xi_{k}}}[\mathcal{Z}%
_{e}^{0}]^{-1}=\frac{\partial^{k}}{\partial C_{\xi_{1}}...\partial C_{\xi_{k}%
}}\sum_{T\in ST_{X}}\prod_{\{x,y\}\in T}C_{x,y}=[\mathcal{Z}_{e}^{0}\prod
_{1}^{k}C_{\xi_{i}}]^{-1}\mathbb{P}_{ST}^{e}(\pm\xi_{1},...,\pm\xi_{k}%
\in\Upsilon).
\]
\end{proof}

This result can be proved directly using for example Grassmann variables (as
used in \cite{LJ1}). The transfer current theorem can then be derived
immediately from it as shown in the following section.

\section{The skew-symmetric Fock space}

Consider the real Fermionic Fock space $\Gamma^{\wedge}(\mathbb{H}^{\ast
})=\overline{\oplus\mathbb{H}^{\ast\wedge n}}$\ obtained as the closure of the
sum of all skew-symmetric tensor powers of $\mathbb{H}^{\ast}$ (the zero-th
tensor power is $\mathbb{R}$).

For any $x\in X$, the anihilation operator $c_{x}$\ and the creation operator
$c_{x}^{\ast}$ are defined as follows, on the uncompleted Fock space
$\oplus\mathbb{H}^{\ast\wedge n}$:%
\[
c_{x}(\mu_{1}\wedge...\wedge\mu_{n})=(-1)^{k-1}\sum_{k}G\mu_{k}(x)\mu
_{1}\wedge...\wedge\mu_{k-1}\wedge\mu_{k+1}\wedge...\wedge\mu_{n}%
\]%
\[
c_{x}^{\ast}(\mu_{1}\wedge...\wedge\mu_{n})=\delta_{x}\wedge\mu_{1}%
\wedge...\wedge\mu_{n}%
\]
Note that $c_{y}^{\ast}$ is the dual of $c_{y}$ and that $[c_{x},c_{y}^{\ast
}]^{+}=G(x,y)$ with all others anticommutators vanishing.

We will work on the complex Fermionic Fock space $\mathcal{F}_{F}$ defined the
tensor product of two copies of $\Gamma^{\wedge}(\mathbb{H}^{\ast})$. The
complex Fock space stucture is defined by two anticommuting sets of creation
and anihilation operators. $\mathcal{F}_{F}$\ is generated by the vector $1$
and creation/anihilation operators $c_{x},c_{x}^{\ast},d_{x},d_{x}^{\ast}$
with $[c_{x},c_{y}^{\ast}]^{+}=[d_{x},d_{y}^{\ast}]^{+}=G(x,y)$ and with all
others anticommutators vanishing.

Anticommuting variables $\psi^{x},\overline{\psi}^{x}$ are defined as
operators on the Fermionic Fock space $\mathcal{F}_{F}$ by:%
\[
\psi^{x}=\sqrt{2}(d_{x}+c_{x}^{\ast})\text{ and }\overline{\psi}^{x}=\sqrt
{2}(-c_{x}+d_{x}^{\ast}).
\]
Note that $\overline{\psi}_{x}$ is not the dual of $\psi_{x}$, but there is an
involution $\mathfrak{I}$ on $\mathcal{F}_{F}$ such that $\overline{\psi
}=\mathfrak{I}\psi^{\ast}\mathfrak{I}$.

$\mathfrak{I}$ is defined by its action on each tensor power: it multiplies
each element in $\mathbb{H}^{\ast\wedge m}\otimes\mathbb{H}^{\ast\wedge p}$ by
$(-1)^{m}$.

\begin{exercise}
Show that in contrast with the Bosonic case, all these operators are bounded.
\end{exercise}

Simple calculations yield that:%
\[
\left\langle 1,\psi^{x_{m}}...\psi^{x_{1}}\overline{\psi}^{y_{1}}%
...\overline{\psi}^{y_{n}}1\right\rangle =\delta_{nm}2^{n}\det(G(x_{i},y_{j}))
\]
and that%

\[
\left\langle 1,\exp(\frac{1}{2}e(\psi,\overline{\psi})-\frac{1}{2}e^{\prime
}(\psi,\overline{\psi}))1\right\rangle _{\mathcal{F}_{F}}=\frac{\det(G)}%
{\det(G^{\prime})}.
\]

Indeed, if $e_{i}$ is an orthonormal basis of $\mathbb{H}^{\ast}$, in which
$e^{\prime}$ is diagonal with eigenvalues $\mu_{i}$, the first side equals
$\left\langle 1,\prod_{i}(1+\frac{1}{2}(1-\lambda_{i})\left\langle \psi
,e_{i}\right\rangle \left\langle \overline{\psi},e_{i}\right\rangle
1\right\rangle _{\mathcal{F}_{F}}=\sum_{k}(1+\sum_{i_{1}<...i_{k}}%
(1-\lambda_{i_{1}})...(1-\lambda_{i})=\prod\lambda_{i}$. In particular, for any
positive measure $\chi$ on $X$,%
\[
\left\langle 1,\exp(-\sum_{x}\chi_{x}\psi^{x}\overline{\psi}^{x}%
)1\right\rangle _{\mathcal{F}_{F}}=\frac{\det(G)}{\det(G_{\chi})}=\left\langle
1,\exp(-\sum_{x}\chi_{x}\varphi^{x}\overline{\varphi}^{x})1\right\rangle
_{\mathcal{F}_{B}}^{-1}.
\]
We observe a ''Supersymmetry'' between $\phi$ and $\psi$: for any exponential
or polynomial $F$%
\[
\left\langle 1,F(\phi\overline{\phi}-\psi\overline{\psi})1\right\rangle
_{\mathcal{F}_{B}\otimes\mathcal{F}_{F}}=F(0).
\]
($1$ denotes $1_{(B)}\otimes1_{(F)}$)

\begin{remark}
On a finite graph, $\psi_{x},\overline{\psi}_{x}$ and the whole supersymmetric
complex Fock space structure can also be defined in terms of complex
differential forms defined on $\mathbb{C}^{\left|  X\right|  }$, using
exterior products, interior products and De Rham $\ast$ operator. This
extension of the Gaussian representation of the complex Bosonic Fock space is
explained in the introduction of \cite{LJ1}. It was used for example in
\cite{LJ11}.
\end{remark}

Note that%
\[
E_{ST}(\prod_{(x,y)\in\tau}\frac{C_{x,y}^{\prime}}{C_{x,y}}\prod
_{x,(x,\delta)\in\tau}\frac{\kappa_{x}^{\prime}}{\kappa_{x}})=\frac{\det
(G)}{\det(G^{\prime})}=\left\langle 1,\exp(\frac{1}{2}e(\psi,\overline{\psi
})-\frac{1}{2}e^{\prime}(\psi,\overline{\psi}))1\right\rangle _{\mathcal{F}%
_{F}}.
\]
The Transfer Current Theorem follows directly, by calculation of%
\begin{align*}
P_{ST}((x_{i},y_{i}) &  \in\tau)=\prod C_{x_{i},y_{i}}\frac{\partial^{k}%
}{\partial C_{x_{1},y_{1}}^{\prime}...\partial C_{x_{k},y_{k}}^{\prime}%
}|_{C^{\prime}=C}\left\langle 1,\exp(\frac{1}{2}e(\psi,\overline{\psi
})-\frac{1}{2}e^{\prime}(\psi,\overline{\psi}))1\right\rangle _{\mathcal{F}%
_{F}})\\
&  =\prod2^{-k}C_{x_{i},y_{i}}\left\langle 1,(\prod(\psi^{y_{i}}-\psi^{x_{i}%
})(\overline{\psi}^{y_{i}}-\overline{\psi}^{x_{i}})1\right\rangle
_{\mathcal{F}_{F}})\\
&  =\det(K^{(x_{i},y_{i}),(x_{j},y_{j})})\prod C_{x_{i},y_{i}}.
\end{align*}
\bigskip

The relations we have established can be summarized in the following diagram:

\bigskip%

\begin{tabular}
[c]{ccc}
& (Wilson Algorithm) & \\
Loop\ ensemble $\mathcal{L}_{1}$ & $\longleftrightarrow$ &
Random\ Spanning\ Tree\\
&  & \\
$\updownarrow$ &  & $\updownarrow$\\
&  & \\
Free field $\phi,\overline{\phi}$ & $\longleftrightarrow$ & Grassmann field
$\psi,\overline{\psi}$\\
& (''Supersymmetry'') &
\end{tabular}

\bigskip

NB: $\phi$ and $\psi$ can also be used jointly to represent bridge functionals
(Cf \cite{LJ1}): in particular%
\[
\int F(\widehat{l})\mu_{x,y}(dl)=\left\langle 1,\phi_{x}\overline{\phi}%
_{y}F(\phi\overline{\phi}-\psi\overline{\psi})1\right\rangle _{\mathcal{F}%
_{B}\otimes\mathcal{F}_{F}}=\left\langle 1,\psi_{x}\overline{\psi}_{y}%
F(\phi\overline{\phi}-\psi\overline{\psi})1\right\rangle _{\mathcal{F}%
_{B}\otimes\mathcal{F}_{F}}.
\]

\chapter{Reflection positivity}
\section{Main result}
In this section, we assume there exists a partition of $X$: $X=X^{+}\cup
X^{-}$, $X^{+}\cap X^{-}=\varnothing$ and an involution $\rho$ on $X$ such that:

\begin{enumerate}
\item[a)] $e$ is $\rho$-invariant.

\item[b)] $\rho$ exchanges $X^{+}$ and $X^{-}$.

\item[c)] The $X^{+}\times X^{+}$\ matrix $C_{x,y}^{\pm}=C_{x,\rho(y)}$, is
nonnegative definite.
\end{enumerate}

Then the following holds:

\begin{theorem}
\label{RF}

\begin{enumerate}
\item[i)] For any \emph{positive integer} $d$ and square integrable function
$\Phi$ in

$\sigma(\widehat{\mathcal{L}_{d}}^{x},x\in X^{+})\vee\sigma(N_{x,y}%
^{(d)},x,y\in X^{+})$,
\[
\mathbb{E}(\Phi(\mathcal{L}_{d})\overline{\Phi}(\rho(\mathcal{L}_{d})))\geq0.
\]

\item[ii)] For any square integrable function $\Sigma$ of the free field
$\phi$ restricted to $X^{+}$,
\[
\mathbb{E}_{\phi}(\Sigma(\phi)\overline{\Sigma}(\rho(\phi)))\geq0.
\]

\item[iii)] For any set of edges $\{\xi_{i}\}$ in $X^{+}\times X^{+}$ the
matrix,
\[
K_{i,j}=\mathbb{P}_{ST}(\xi_{i}\in T,\rho\xi_{j}\in T)-\mathbb{P}_{ST}(\xi
_{i}\in T)\mathbb{P}_{ST}(\xi_{j}\in T)
\]
is nonpositive definite.
\end{enumerate}
\end{theorem}

\begin{proof}
The property ii) is well known in a slightly different context and is named
reflexion positivity: Cf for example \cite{Sim2}, \cite{Gaw} and their
references. Reflection positivity is a keystone in the bridge between
statistical and quantum mechanics.

To prove i), we use the fact that the $\sigma$-algebra is generated by the
algebra of random variables of the form $\Phi=\sum\lambda_{j}B_{(d)}%
^{e,e_{j},\omega_{j}}$ with $C^{(e_{j})}=C$ and $\omega_{j}=0$\ except on
$X^{+}\times X^{+}$, $C^{(e_{j})}\leq C$ on $X^{+}\times X^{+}$,
$\lambda^{(e_{j})}=\lambda$ on $X^{-}$ and $\lambda^{(e_{j})}\geq\lambda$ on
$X^{+}$.\newline Then
\[
\mathbb{E}(\Phi(\mathcal{L}_{d})\overline{\Phi}(\rho(\mathcal{L}%
_{d})))=\mathbb{E}(\sum\lambda_{j}\overline{\lambda}_{q}B_{(d)}^{e,e_{j,q}%
,\omega_{j}-\rho(\omega_{q})})=\sum\lambda_{j}\overline{\lambda}%
_{q}(\frac{\mathcal{Z}_{e_{j,q},\omega_{j}-\rho(\omega_{q})}}{\mathcal{Z}_{e}%
})^{d}%
\]
with $e_{j,q}=e_{j}+\rho(e_{q})-e$.

We have to prove this is non negative. It is enough to prove it for $d=1$, as
the Hadamard product of two nonnegative definite Hermitian matrices is
nonnegative definite.

Let us first assume that the nonnegative definite matrix $C^{\pm}$ is positive
definite. We will see that the general case can be reduced to this one.

Now note that $\mathcal{Z}_{e_{j}+\rho(e_{q})-e,\omega_{j}-\rho(\omega_{q})}$
is the inverse of the determinant of a positive definite matrix of the form:
\[
D(j,q)=\left[
\begin{array}
[c]{cc}%
A(j) & -C^{\pm}\\
-C^{\pm} & A(q)^{\ast}%
\end{array}
\right]
\]
with $[A(j)]_{u,v}=\lambda_{u}^{(e_{j})}\delta_{u,v}-C_{u,v}^{(e_{j}%
)}e^{i\omega_{j}^{u,v}}$ and $C_{u,v}^{\pm}=C_{u,\rho(v)}$.\newline It is
enough to show that $\det(D(j,k))^{-1}$ can be expanded in series of products
$\sum q_{n}(j)\overline{q_{n}}(k)$ with $\sum\left|  q_{n}(j)\right|
^{2}<\infty$.\newline As
\begin{multline*}
D(j,q)=\\
\left[
\begin{array}
[c]{cc}%
\lbrack C^{\pm}]^{\frac{1}{2}} & 0\\
0 & [C^{\pm}]^{\frac{1}{2}}%
\end{array}
\right]  \left[
\begin{array}
[c]{cc}%
\lbrack C^{\pm}]^{-\frac{1}{2}}A(j)[C^{\pm}]^{-\frac{1}{2}} & -I\\
-I & [C^{\pm}]^{-\frac{1}{2}}A(q)^{\ast}[C^{\pm}]^{-\frac{1}{2}}%
\end{array}
\right]  \left[
\begin{array}
[c]{cc}%
\lbrack C^{\pm}]^{\frac{1}{2}} & 0\\
0 & [C^{\pm}]^{\frac{1}{2}}%
\end{array}
\right]
\end{multline*}
the inverse of this determinant can be written%
\[
\det(C^{\pm})^{-2}\det(F(j))\det(F(q)^{\ast})\det(I-\left[
\begin{array}
[c]{cc}%
0 & F(j)\\
F(q)^{\ast} & 0
\end{array}
\right]  )^{-1}%
\]
with $F(j)=[C^{\pm}]^{\frac{1}{2}}A(j)^{-1}[C^{\pm}]^{\frac{1}{2}}$, or more
simply:
\[
F(j)=\det(A(j))^{-1}\det(A(q)^{\ast})^{-1}\det(I-\left[
\begin{array}
[c]{cc}%
0 & F(j)\\
F(q)^{\ast} & 0
\end{array}
\right]  )^{-1}.
\]

Note that $A(j)^{-1}$ is also the Green function of the restriction to $X^{+}$
of the Markov chain associated with $e_{j}$, twisted by $\omega_{j}$.
Therefore $A(j)^{-1}C^{\pm}=[C^{\pm}]^{-\frac{1}{2}}F(j)[C^{\pm}]^{\frac{1}%
{2}}$ is the balayage kernel on $X^{-}$\ defined by this Markov chain with an
additional phase under the expectation produced by $\omega_{j}$. It is
therefore clear that the eigenvalues of the matrices $A(j)^{-1}C^{\pm}$ and
$F(j)$ are of modulus less than one and it follows that
\[
\left[
\begin{array}
[c]{cc}%
0 & F(j)\\
F(q)^{\ast} & 0
\end{array}
\right]  =\left[
\begin{array}
[c]{cc}%
0 & I\\
I & 0
\end{array}
\right]  \left[
\begin{array}
[c]{cc}%
F(q)^{\ast} & 0\\
0 & F(j)
\end{array}
\right]
\]
is a contraction. We can always assume it is a strict contraction, by adding a
killing term we can let converge to zero once the inequality is proved.

If $X^{+}$ has only one point, $(1-F(j)F(q)^{\ast})^{-1}=\sum F(j)^{-n}%
\overline{F(q)}^{-n}$ which allows to conclude. Let us now treat the general case.

For any $(n,m)$ matrix $N$, and $k=(k_{1},...,k_{m})\in\mathbb{N}^{m}$,
$l=(l_{1},...,l_{n})\in\mathbb{N}^{n}$, let $N^{\{k,l\}}$ denote the $(\left|
k\right|  ,\left|  l\right|  )$ matrix obtained from by repeating $k_{i}$
times\ each line $i$; then $l_{j}$ times each column $j$.

We use the expansion
\[
\det(I-M)^{-1}=1+\sum\frac{1}{\left|  k\right|  !}Per(M^{\{k,k\}})
\]
valid for any strict contraction $M$ (Cf \cite{VJ1} and \cite{VJ2}).

Note that if $X$ has $2d$ points, if we denote $(k_{1},...,k_{2d})$ by
$(k^{+},k^{-})$, with $k^{+}=(k_{1},...,k_{d})$ and $k^{-}=(k_{d+1}%
,...,k_{2d})$,
\[
\left[
\begin{array}
[c]{cc}%
0 & F(j)\\
F(q)^{\ast} & 0
\end{array}
\right]  ^{\{k,k\}}=\left[
\begin{array}
[c]{cc}%
0 & F(j)^{\{k^{+},k^{-}\}}\\
\lbrack F(q)^{\ast}]^{\{k^{-},k^{+}\}} & 0
\end{array}
\right]  .
\]

But the all terms in the permanent of a $(2n,2n)$\ matrix of the form
$\left[
\begin{array}
[c]{cc}%
0 & A\\
B^{\ast} & 0
\end{array}
\right]  $\ vanish unless the submatrices $A$ and $B$ are square matrices (not
necessarily of equal ranks). Hence in our case, we necessary have $\left|
k^{+}\right|  =\left|  k^{-}\right|  $, so that, $A$ and $B$ are $(n,n)$ matrices.

Then, the non zero terms in the permanent come from permutations exchanging
$\{1,2,...,n\}$ and $\{n+1,...,2n\}$, which can be decomposed into a pair of
permutations of $\{1,2,...,n\}$. Therefore:%

\[
Per(\left[
\begin{array}
[c]{cc}%
0 & A\\
B^{\ast} & 0
\end{array}
\right]  )=Per(A)Per(B^{\ast})
\]
which concludes the proof in the positive definite case as%

\[
Per(B^{\ast})=\sum_{\tau\in\mathcal{S}_{n}}\prod_{1}^{n}B_{i,\tau(i)}^{\ast
}=\sum_{\tau\in\mathcal{S}_{n}}\prod_{1}^{n}\overline{B}_{\tau(i),i}%
=\overline{Per(B)}.
\]

To treat the general case where $C^{\pm}$ is only nonnegative definite., we
can use use a passage to the limit or alternatively, the proposition \ref{dec}
(or more precisely its extension including a current) to reduce the sets
$X^{+}$ and $X^{-}$ to the support of $C^{\pm}$.

To prove ii) let us first show the assumptions imply that the $X^{+}\times
X^{+}$\ matrix $G_{x,y}^{\pm}=G^{x,\rho(y)}\ $is also nonnegative definite.
Let us write $G$ in the form $\left[
\begin{array}
[c]{cc}%
A & -C^{\pm}\\
-C^{\pm} & A
\end{array}
\right]  ^{-1}$ with $A=M_{\lambda}-C$. Then
\[
G=\left[
\begin{array}
[c]{cc}%
A^{-\frac{1}{2}} & 0\\
0 & A^{-\frac{1}{2}}%
\end{array}
\right]  \left[
\begin{array}
[c]{cc}%
I & -A^{-\frac{1}{2}}C^{\pm}A^{-\frac{1}{2}}\\
-A^{-\frac{1}{2}}C^{\pm}A^{-\frac{1}{2}} & I
\end{array}
\right]  ^{-1}\left[
\begin{array}
[c]{cc}%
A^{-\frac{1}{2}} & 0\\
0 & A^{-\frac{1}{2}}%
\end{array}
\right]  .
\]
$A^{-\frac{1}{2}}C^{\pm}A^{-\frac{1}{2}}$ is non negative definite and as
before, we can check it is a contraction since $A^{-1}C^{\pm}$ is a balayage kernel.

Note that if a symmetric nonnegative definite matrix $K$\ has eigenvalues
$\mu_{i}$, the eigenvalues of the symmetric matrix $E$ defined by
\[
\left[
\begin{array}
[c]{cc}%
I & -K\\
-K & I
\end{array}
\right]  ^{-1}=\left[
\begin{array}
[c]{cc}%
D & E\\
E & D
\end{array}
\right]
\]
are easily seen (exercise)\ to be $\frac{\mu_{i}}{1-\mu_{i}^{2}}$. Taking
$K=A^{-\frac{1}{2}}C^{\pm}A^{-\frac{1}{2}}$, it follows that the symmetric
matrix $E$ , (and in our particular case $G^{\pm}=A^{-\frac{1}{2}%
}EA^{-\frac{1}{2}}$)\ is nonnegative definite.

To finish the proof, let us take $\Sigma$ of the form $\sum\lambda
_{j}e^{\left\langle \phi,\chi_{j}\right\rangle }$. Then
\begin{multline*}
\mathbb{E}_{\phi}(\Sigma(\phi)\overline{\Sigma}(\rho(\phi))=\sum\lambda
_{j}\lambda_{q}\mathbb{E}_{\phi}(e^{\left\langle \phi,\chi_{j}\right\rangle
+\left\langle \phi,\rho(\chi_{q})\right\rangle })\\
=\sum\lambda_{j}e^{\frac{1}{2}\left\langle \chi_{j},G^{++}\chi_{j}%
\right\rangle }\lambda_{k}e^{\frac{1}{2}\left\langle \chi_{q},G^{++}\chi
_{q}\right\rangle }e^{\left\langle \chi_{j},G^{\pm}\chi_{q}\right\rangle }%
\end{multline*}
(using that $G^{\pm}$ is symmetric).

As $G^{\pm}$ is positive definite, we can conclude since $e^{\frac{1}%
{2}\left\langle \chi_{j},G^{\pm}\chi_{q}\right\rangle }=\mathbb{E}%
_{w}(e^{\left\langle w,\chi_{j}\right\rangle }e^{\left\langle w,\chi
_{q})\right\rangle })$, $w$ denoting the Gaussian field on $X^{+}$ with
covariance $G^{\pm}$.

To prove iii), note that the transfer impedance matrix can be decomposed as
$G$. In particular, $K_{i,j}=-(K_{\xi_{i},\xi_{j}}^{\pm})^{2}$, with
\[
K_{(x,y),(u,v)}^{\pm}=K^{(x,y),(\rho(u),\rho(v))}=G^{\pm}(x,u)+G^{\pm
}(y,v)-G^{\pm}(x,v)-G^{\pm}(y,u).
\]
Then, using again the Gaussian vector $w$, and the Wick squares of its
components:
\[
(K_{(x,y),(u,v)}^{\pm})^{2}=E(:(w_{u}-w_{v})^{2}::(w_{x}-w_{y})^{2}:).
\]
\end{proof}

\begin{remark}
\begin{enumerate}

\item[a)] If $U_{j}$ are unitary representations with $d_{U}=d$ and such that
$U_{j}^{x,y}$ is the identity outside $X^{+}\times X^{+}$, i) can be extended
to variables of the form $\sum\lambda_{j}B_{(d)}^{e,e_{j},U_{j}}$ and to the
$\sigma$-field they generate.

\item[b)] The property i) can be also derived from the reflection positivity
of the free field ii) and by remark \ref{baba2}. Then it can also be proved
that for any set of points $\{x_{i}\}$ in $X^{+}$, the matrix $\mathbb{E}%
(\Phi(\mathcal{L}_{d})\overline{\Phi}(\rho(\mathcal{L}_{d})N_{x_{i},\rho
x_{j}})$ is non-negative definite.

\item[c)] In the case where $\alpha$ is a half integer, by remark \ref{baba},
the reflection positivity of the free field ii), implies i) holds also for any
\emph{half integer} $\alpha$ provided that $\Phi\in\sigma(\widehat
{\mathcal{L}_{\alpha}}^{x},x\in X^{+})\vee\sigma(N_{x,y}^{(\alpha)}%
+N_{y,x}^{(\alpha)},x,y\in X^{+})$.
\end{enumerate}
\end{remark}

\begin{exercise}
Prove the above remarks.
\end{exercise}

\begin{remark}
\label{gem}If there exists a partition of $X$: $X=X^{+}\cup X^{-}\cup X^{0}$,
and an involution $\rho$ on $X$ such that:

\begin{enumerate}
\item[a)] $e$ and $X^{0}$ are $\rho$-invariant.

\item[b)] $\rho(X^{\pm})=X^{\mp}$

\item[c)] $X^{+}$ and $X^{-}$ are disconnected.
\end{enumerate}

Then the assumptions of the previous theorem are satisfied for the trace on
$X^{+}\cup X^{-}$.
\end{remark}

Moreover, if $X^{0}\times X^{0}$ does not contain any edge of the graph, the
assertion i) of theorem \ref{RF}\ holds for the non disjoint sets $X^{+}\cup
X^{0}$ and $X^{-}\cup X^{0}$. More precisely, i), holds for $\Phi$\ in
$\sigma(\widehat{\mathcal{L}_{d}}^{x},x\in X^{+}\cup X^{0})\vee\sigma
(N_{x,y}^{(d)},\;x,y\in X^{+}\cup X^{0})$. It is enough to apply the theorem to the graph
obtained by duplication of each point $\ x_{0}$ in $X^{0}$ into $(x_{0}^{+},
x_{0}^{-})$, with $x_{0}^{\pm}$ connected to points in $X^{\pm}$ and connected
together by conductances $C_{x_{0}^{+}, x_{0}^{-}}$\ we can let increase to infinity.

\section{A counter example}

Let show that the reflexion positivity does not hold under $\mu$ for loop
functionals. Therefore, it will be clear it does not hold for small $\alpha$.
We will consider functionals of the occupation field.

Consider the graph formed by a cube $\pm a$, $\pm b$, $\pm c$, $\pm d$\ and
the mid-points $\pm\alpha$, $\pm\beta$, $\pm\gamma$, $\pm\delta$ of the sides
$\pm ab$, $\pm cd$, $\pm ac$, $\pm bd$. The edges are given by the sides of
the cube, as in the picture.%

\begin{figure}
[ptb]
\begin{center}
\includegraphics[
height=2.9784in,
width=3.0658in
]%
{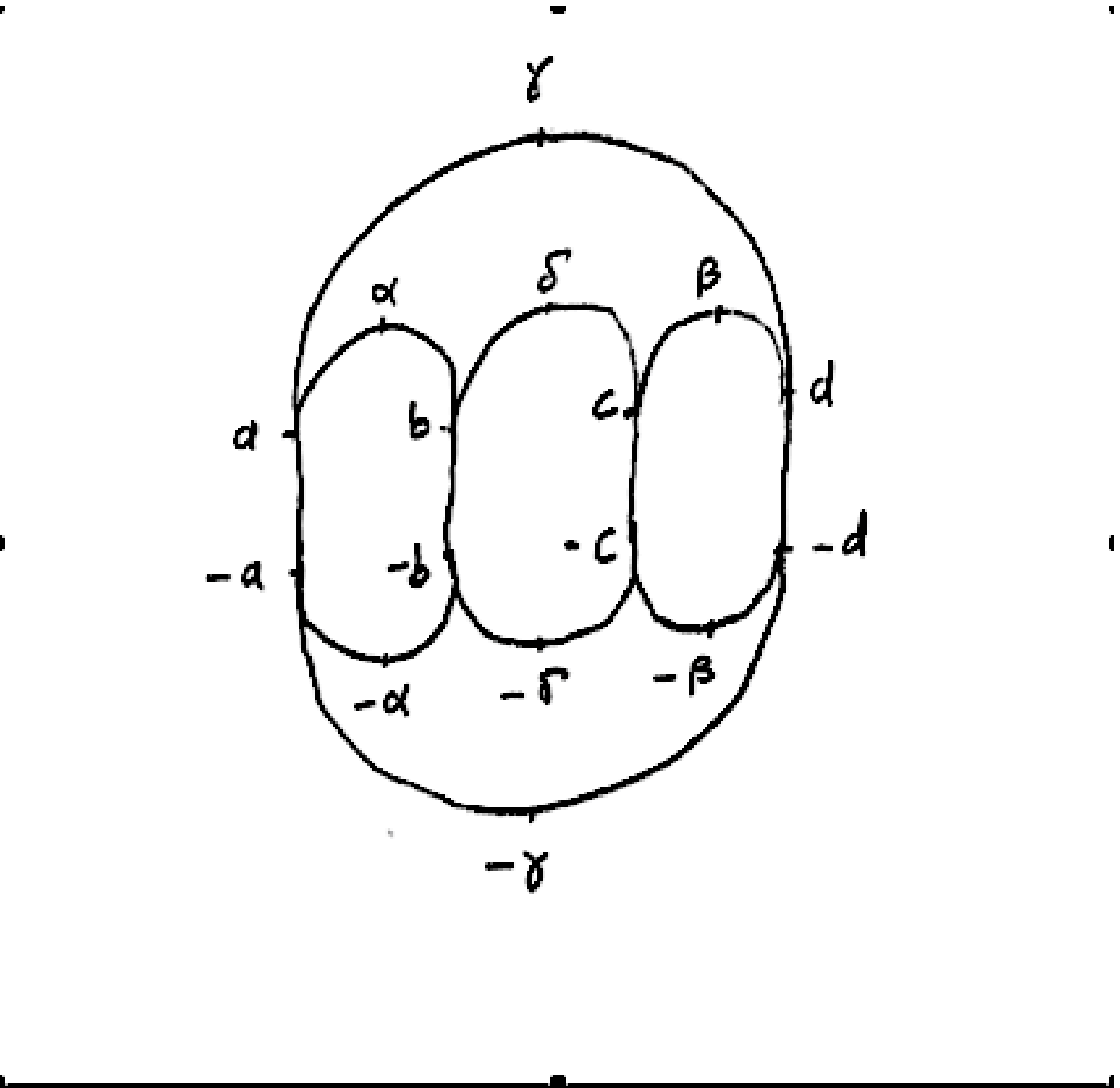}%
\caption{A counter example}%
\end{center}
\end{figure}

We can take for example all conductances and killing rates to be equal. Then
the symmetry $\rho:x\rightarrow-x$ defines an involution satisfying the
assumption of theorem \ref{RF}. Define the set of loops $A=\{l,\ \widehat
{l}^{\alpha}\widehat{l}^{\beta}>0\}$, $A^{\prime}=\{l,\ \widehat{l}^{\alpha
}=\widehat{l}^{\beta}=0\}$, $B=\{l,\ \widehat{l}^{\gamma}\widehat{l}^{\delta
}>0\}$ and $B^{\prime}=\{l,\ \widehat{l}^{\gamma}=\widehat{l}^{\delta}=0\}$.
Note that $A\cap B^{\prime}\cap\rho(A)\cap\rho(B^{\prime})$, $A^{\prime}\cap
B\cap\rho(A^{\prime})\cap\rho(B)$ are empty. But $A^{\prime}\cap B\cap
\rho(A)\cap\rho(B^{\prime})$ and $A\cap B^{\prime}\cap\rho(A^{\prime})\cap
\rho(B)$ are not (consider the loop $a\alpha b(-b)(-\delta)(-c)c\beta
d(-d)(-\gamma)(-a)a$).

Then, if we set $\Phi=1_{A\cap B^{\prime}}-1_{A^{\prime}\cap B}$, it is clear
that
\[
\mu(\Phi.\Phi\circ\rho)=-2\mu(A\cap B^{\prime}\cap\rho(A^{\prime})\cap
\rho(B))<0.
\]

\section{Physical Hilbert space and time shift:}

We will now work under the assumptions of remark \ref{gem}, namely, without
assuming that $X=X^{+}\cup X^{-}$.

The following results and terminology are inpired by methods of constructive
Quantum field theory (Cf \cite{Sim2} and \cite{Gaw}).

Let $\mathcal{H}^{+}$ be the space of square integrable functions in
$\sigma((\widehat{\mathcal{L}_{1}}^{x},x\in X^{+})\vee\sigma(N_{x,y}%
^{(1)},x,y\in X^{+})$, equipped with the scalar product $\left\langle
\Phi,\Psi\right\rangle _{\mathcal{H}}=\mathbb{E(}\Phi(\mathcal{L}_{1}%
)\Psi(\rho(\mathcal{L}_{1}))$ Note that $\left\langle \Phi,\Phi\right\rangle
_{\mathcal{H}}\leq$\ $\mathbb{E(}\Phi^{2}(\mathcal{L}_{1}))$ by
Cauchy-Schwartz inequality.

Let $\mathcal{N}$ be the subspace $\{\Psi\in\mathcal{H}^{+},\mathbb{E(}%
\Psi(\mathcal{L}_{1})\Psi(\rho(\mathcal{L}_{1}))=0\}$ and $\mathcal{H}$ the
closure (for the topology induced by this scalar product) of the quotient
space $\mathcal{H}^{+}/\mathcal{N}$ (which can be called the physical Hilbert
space). We denote $\Phi^{\sim}$ the equivalence class of $\Phi$. $\mathcal{H}$
is equipped with the scalar product defined unambiguously by $\left\langle
\Phi^{\sim},\Psi^{\sim}\right\rangle _{\mathcal{H}}=\left\langle \Phi
,\Psi\right\rangle _{\mathcal{H}}$.

Assume $X$ is of the form $X_{0}\times\mathbb{Z}$ (space $\times$ time)) and
let $\theta$ be the natural time shift. We assume $\theta$ preserves
$e$,\ i.e. that conductances and $\kappa$ are $\theta$-invariant. We define
$\rho$ by $\rho(x_{0},n)=(x_{0},-n)$ and assume $e$ is $\rho$-invariant. Note
that $\theta(X^{+})\subseteq X^{+}$ and $\rho\theta=\theta^{-1}\rho$. The
transformations $\rho$ and $\theta$ induce a transformations on loops that
preserves $\mu$, and $\theta$ induces a linear transformation of
$\mathcal{H}^{+}$. Moreover, given any $F$ in $\mathcal{N}$, $F\circ\theta
\in\mathcal{N}$, as $\left\langle F\circ\theta,F\circ\theta\right\rangle
_{\mathcal{H}}$ is nonegative and equals
\begin{align*}
\mathbb{E(}F\circ\theta(\mathcal{L}_{1})F\circ\theta(\rho(\mathcal{L}_{1}))
&  =\mathbb{E(}F(\theta(\mathcal{L}_{1}))F(\rho\circ\theta^{-1}(\mathcal{L}%
_{1}))=\mathbb{E(}F(\theta^{2}(\mathcal{L}_{1}))F(\rho(\mathcal{L}_{1}))\\
&  =\left\langle F\circ\theta^{2},F\right\rangle _{\mathcal{H}}\leq
\sqrt{\left\langle F\circ\theta^{2},F\circ\theta^{2}\right\rangle
_{\mathcal{H}}\left\langle F,F\right\rangle _{\mathcal{H}}}%
\end{align*}
which vanishes.

\begin{proposition}
There exist a self adjoint contraction of $\mathcal{H}$, we will denote
$\Pi^{(\theta)}$ such that $[\Phi\circ\theta]^{\sim}=\Pi^{(\theta)}(\Phi
^{\sim})$.
\end{proposition}

\begin{proof}
The existence of $\Pi^{(\theta)}$ follows from the last observation made
above. As $\theta$ preserves $\mathbb{\mu}$, it follows from the identity
$\rho\theta=\theta^{-1}\rho$ that
\begin{align*}
\left\langle F\circ\theta,G\right\rangle _{\mathcal{H}}  &  =\mathbb{E(}%
F(\theta(\mathcal{L}_{1}))G(\rho(\mathcal{L}_{1})=\mathbb{E(}F(\mathcal{L}%
_{1})G(\rho\circ\theta^{-1}(\mathcal{L}_{1}))=\mathbb{E(}F(\mathcal{L}%
_{1})G(\theta\circ\rho(\mathcal{L}_{1}))\\
&  =\mathbb{E(}F(\rho(\mathcal{L}_{1}))G(\theta(\mathcal{L}_{1}))=\left\langle
G\circ\theta,F\right\rangle _{\mathcal{H}}.
\end{align*}
Therefore, $\Pi^{(\theta)}$ is self adjoint on $\mathcal{H}^{+}/\mathcal{N}$.
To prove that it is a contraction, it is enough to show that $\left\langle
F\circ\theta,F\circ\theta\right\rangle _{\mathcal{H}}\leq\left\langle
F,F\right\rangle _{\mathcal{H}}$ for all $F\in\mathcal{H}^{+}$.\newline But as
shown above, $\left\langle F\circ\theta,F\circ\theta\right\rangle
_{\mathcal{H}}=\left\langle F\circ\theta^{2},F\right\rangle _{\mathcal{H}}%
\leq\sqrt{\left\langle F\circ\theta^{2},F\circ\theta^{2}\right\rangle
_{\mathcal{H}}\left\langle F,F\right\rangle _{\mathcal{H}}}.$\newline By
recursion, it follows that:
\[
\left\langle F\circ\theta,F\circ\theta\right\rangle _{\mathcal{H}}%
\leq\left\langle F\circ\theta^{2^{n}},F\circ\theta^{2^{n}}\right\rangle
_{\mathcal{H}}^{2^{-n}}\left\langle F,F\right\rangle _{\mathcal{H}}^{1-2^{-n}}%
\]
As $\left\langle F\circ\theta^{2^{n}},F\circ\theta^{2^{n}}\right\rangle
_{\mathcal{H}}^{2^{-n}}\leq(\mathbb{E(}F^{2}(\mathcal{L}_{1}))^{2^{-n}}$
converges to $1$ as $n\rightarrow\infty$, the inequality follows.
\end{proof}

\medskip

For all $n\in\mathbb{Z}$, the symmetry $\rho^{(n)}=\theta^{-n}\rho\theta^{n}$
allows to define spaces $\mathcal{H}^{(n)}$ isometric to $\mathcal{H}$. These
isometries can be denoted by the shift $\theta^{n}$. For $n>m$, $j_{n,m}%
=\theta^{m}[\Pi^{(\theta)}]^{n-m}\theta^{-n}$ is a contraction from
$\mathcal{H}^{(n)}$ into $\mathcal{H}^{(m)}$.

\chapter{The case of general symmetric Markov processes}
\section{Overview}
We now explain briefly how some of the above results can be extended to
symmetric Markov processes on continuous spaces. The construction of the loop
measure as well as a lot of computations can be performed quite generally,
using Markov processes or Dirichlet space theory (Cf for example
\cite{Fukutak}). It works as soon as the bridge or excursion measures
$\mathbb{P}_{t}^{x,y}$\ can be properly defined. The semigroup should have a
density with repect to the duality measure given by a locally integrable
kernel $p_{t}(x,y)$. This is very often the case in examples of interest,
especially in finite dimensional spaces.

The main issue is to determine wether the results which have been developped
in the previous chapters still hold, and precisely in what sense..

\paragraph{\textbf{Loop hitting distributions}}

An interesting result is formula \ref{F1F2}, and its reformulation in
proposition \ref{F1FF}.

The expression on the lefthand side is well defined but the determinants
appearing in \ref{F1F2}\ are not. In the example of Brownian motion killed at
the exit of a bounded domain, Weyl asymptotics show that the divergences
appearing on the righthand side of \ref{F1F2} may cancel. And in fact, the
righthand side in \ref{F1FF} can be well defined in terms of the densities of
the hitting distributions of $F_{1}$ and $F_{2}$ with repect to their
capacitary measures, which allow to take the trace. A direct proof, using
Brownian motion and classical potential theory, should be easy to provide,
along the lines of the solution of exercise \ref{bfk}.

\paragraph{\textbf{Determinantal processes}}

Another result of interest involves the point process defined by the points
connected to the root of a random spanning tree.  In the case of an interval of $\mathbb{Z}$, we get a process with independent spacings. For one dimensional
diffusions, this point process with independent spacings has clearly an
analogue which is the determinantal process with independent spacings (See
\cite{Sosh})\ defined by the kernel $\sqrt{k(x)}G(x,y)\sqrt{k(y)}$ ($k$ beeing
the killing rate and $G$ the Green function). For one dimensional Brownian
motion killed at a constant rate, we recover Macchi point process (Cf
\cite{Mach}).

It suggests that this process (together with the loop ensemble $\mathcal{L}%
_{1}$) can be constructed by various versions of Wilson algorithm adapted to
the real line. A similar result holds on $\mathbb{Z}$ or $\mathbb{N}$ , where
the natural ordering can be used to construct the spanning tree by Wilson
algorithm, starting at $0$.

For constant killing rate, $\sqrt{k(x)}G(x,y)\sqrt{k(y)}$ can be expressed as

 $\rho\exp(-\left|x-y\right|  /a)$, with $a,\rho>0$ and $2\rho a<1$, the law of the spacings has
therefore a density proportional to $e^{-\frac{x}{a}}\sinh(\sqrt{1-2\rho
a}\frac{x}{a})$ (Cf \cite{Mach}), which appears to be the convolution of two
exponential distributions of parameters $\frac{1}{a}(\sqrt{1-2\rho a}+1)$ and
$\frac{1}{a}(-\sqrt{1-2\rho a}+1)$. A similar result holds on $\mathbb{Z}$
with geometric distributions. The spanning forest obtained by removing the
cemetery point is composed of trees made of pair of intervals joining at
points directly connected to the cemetery, whose length are independent with
laws given by these (different!) exponential distributions. The separating
points between these trees form a determinantal process intertwinned with the
previous one (the roots directly connected to the cemetery point), with the
same distribution. There are two equally probable intertwinning configurations
on $\mathbb{R}$, and only one in $\mathbb{R}^{+}$ or $\mathbb{R}^{-}$.

\paragraph{Occupation field and continuous branching}

Let us consider more closely the occupation field $\widehat{l}$. The extension
is rather straightforward when points are not polar. We can start with a
Dirichlet space of continuous functions and a measure $m$ such that there is a
mass gap. Let $P_{t}$ denote the associated Feller semigroup. Then the Green
function  $G(x,y)$ is well defined as the mutual energy of the Dirac measures
$\delta_{x}$ and $\delta_{y}$ which have finite energy. It is the covariance
function of a Gaussian free field $\phi(x)$, and the field $\frac{1}{2}%
\phi(x)^{2}$ will have the same distribution as the field $\widehat{\mathcal{L}}_{\frac{1}{2}}^{x}$ of local times of the Poisson
process of random loops whose intensity is given by the loop measure defined
by the semigroup $P_{t}$. This will applies to examples related to
one-dimensional Brownian motion (or to Markov chains on countable spaces).

\begin{remark}
When we consider Brownian motion on the half line, the associated occupation
field $\widehat{\mathcal{L}_{\alpha}}$\ is a continuous branching process with
immigration, as in the simple random walk case considered above.
\end{remark}

\paragraph{Generalized fields and renormalization}

When points are polar, one needs to be more careful. We will consider only the
case of the two and three dimensional Brownian motion in a bounded domain
$D$\ killed at the boundary, i.e. associated with the classical energy with
Dirichlet boundary condition. The Green function does not induce a trace class
operator but it is still Hilbert-Schmidt which allows us to define
renormalized determinants $\det_{2}$ (Cf \cite{Sim}).

If $A$ is a symmetric Hilbert Schmidt operator, $\det_{2}(I+A)$ is defined as
$\prod(1+\lambda_{i})e^{-\lambda_{i}}$ where $\lambda_{i}$ are the eigenvalues
of $A$.

The Gaussian field (called free field) whose covariance function is the Green
function is now a generalized field: Generalized fields are not defined
pointwise but have to be smeared by a compactly supported continuous test
function $f$. Still $\phi(f)$ is often denoted $\int\phi(x)f(x)dx.$

The Wick powers $:\phi^{n}:$\ of the free field can be defined as generalized
fields by approximation as soon as the $2n$-th power of the Green function,
$G(x,y)^{2n}$ is locally integrable (Cf \cite{Sim2}). This is the case for all
$n$ for the two dimensional Brownian motion killed at the exit of an open set,
as the Green function has only a logarithmic singularity on the diagonal, and
for $n=2$ in dimension three as the singularity is of the order of
$\frac{1}{\left\|  x-y\right\|  }$. More precisely, taking for example
$\pi_{\varepsilon}^{x}(dy)$ to be the normalized area measure on the sphere of
radius $\varepsilon$ around $x$, $\phi(\pi_{\varepsilon}^{x})$ is a Gaussian
field with variance $\sigma_{\varepsilon}^{x}=\int G(z,z^{\prime}%
)\pi_{\varepsilon}^{x}(dz)\pi_{\varepsilon}^{x}(dz^{\prime})$. Its Wick powers
are defined with Hermite polynomials as we did previously:

$:\phi(\pi_{\varepsilon}^{x})^{n}:\;=(\sigma_{\varepsilon}^{x})^{\frac{n}{2}%
}H_{n}(\frac{\phi(\pi_{\varepsilon}^{x})}{\sqrt{\sigma_{\varepsilon}^{x}}})$.
Then one can see that, for any compactly supported continuous function $f$,
$\int f(x):\phi(\pi_{\varepsilon}^{x})^{n}:dx$ converges in $L^{2}$ towards a
limit called the $n$-th Wick power of the free field evaluated on $f$ and
denoted $:\phi^{n}:(f)$. Moreover, $\mathbb{E}(:\phi^{n}:(f):\phi
^{n}:(h))=\int G^{2n}(x,y)f(x)h(y)dxdy$.

In these cases, we can extend the statement of theorem \ref{iso} to the
renormalized occupation field $\widetilde{\mathcal{L}}_{\frac{1}{2}}^{x}$\ and
the Wick square $:\phi^{2}:$\ of the free field.

\bigskip

\section{Isomorphism for the renormalized occupation field}

Let us explain this in more detail in the Brownian motion case. Let $D$ be an
open subset of $\mathbb{R}^{d}$ such that the Brownian motion killed at the
boundary of $D$ is transient and has a Green function.\ Let $p_{t}(x,y)$ be
its transition density and $G(x,y)=\int_{0}^{\infty}p_{t}(x,y)dt$\ the
associated Green function. The loop measure $\mu$ was defined in \cite{LW} as
\[
\mu=\int_{D}\int_{0}^{\infty}\frac{1}{t}\mathbb{P}_{t}^{x,x}dt
\]
where $\mathbb{P}_{t}^{x,x}$ denotes the (non normalized) 
bridge measure of duration $t$ such that if $0\leq t_{1}\leq...\leq t_{h}\leq t$,%
\[
\mathbb{P}_{t}^{x,x}(\xi(t_{1})\in dx_{1},...,\xi(t_{h})\in dx_{h})=p_{t_{1}%
}(x,x_{1})p_{t_{2}-t_{1}}(x_{1},x_{2})...p_{t-t_{h}}(x_{h},x)dx_{1}...dx_{h}%
\]
(the mass of $\mathbb{P}_{t}^{x,x}$ is $p_{t}(x,x)$). Note that $\mu$ is a
priori defined on based loops but it is easily seen to be shift-invariant.

\bigskip

For any loop $l$ indexed by $[0\;T(l)]$, define the measure $\widehat{l}%
=\int_{0}^{T(l)}\delta_{l(s)}ds$: for any Borel set $A$, $\widehat{l}%
(A)=\int_{0}^{T(l)}1_{A}(l_{s})ds$.

\begin{lemma}
For any non-negative function $f$,
\[
\mu(\left\langle \widehat{l},f\right\rangle ^{n})=(n-1)!\int G(x_{1}%
,x_{2})f(x_{2})G(x_{2},x_{3})f(x_{3})...G(x_{n},x_{1})f(x_{1})\prod_{1}%
^{n}dx_{i}.
\]
\end{lemma}

\begin{proof}
From the definition of $\mu$ and $\widehat{l}$, $\mu(\left\langle \widehat
{l},f\right\rangle ^{n})$ equals:
\begin{multline*}
n!\int\int_{\{0<t_{1}<...<t_{n}<t\}}\frac{1}{t}f(x_{1})...f(x_{n})p_{t_{1}%
}(x,x_{1})...p_{t-t_{n}}(x_{n},x)\prod dt_{i}dx_{i}dtdx\\
=n!\int\int_{\{0<t_{1}<...<t_{n}<t\}}\frac{1}{t}f(x_{1})...f(x_{n}%
)p_{t_{2}-t_{1}}(x_{1},x_{2})...p_{t_{1}+t-t_{n}}(x_{n},x_{1})\prod
dt_{i}dx_{i}dt.
\end{multline*}
Performing the change of variables $v_{2}=t_{2}-t_{1},...,v_{n}=t_{n}%
-t_{n-1},v_{1}=t_{1}+t-t_{n}$, and $v=t_{1}$, we obtain:
\begin{align*}
n!\int_{\{0<v<v_{1},0<v_{i}\}}  &  \frac{1}{v_{1}+...+v_{n}}f(x_{1}%
)...f(x_{n})p_{v_{2}}(x_{1},x_{2})...p_{v_{1}}(x_{n},x_{1})\prod dv_{i}%
dx_{i}dv\\
&  =n!\int_{\{0<v_{i}\}}\frac{v_{1}}{v_{1}+...+v_{n}}f(x_{1})...f(x_{n}%
)p_{v_{2}}(x_{1},x_{2})...p_{v_{1}}(x_{n},x_{1})\prod dv_{i}dx_{i}\\
&  =(n-1)!\int_{\{0<v_{i}\}}f(x_{1})...f(x_{n})p_{v_{2}}(x_{1},x_{2})\ldots
p_{v_{1}}(x_{n},x_{1})\prod dv_{i}dx_{i}\\
&  (\mbox{as we get the same formula
with any }  v_{i} \mbox{ instead of }
v_{1})\\
&  =(n-1)!\int G(x_{1},x_{2})f(x_{2})G(x_{2},x_{3})f(x_{3})...G(x_{n}%
,x_{1})f(x_{1})\prod_{1}^{n}dx_{i}.
\end{align*}
\end{proof}

One can define in a similar way the analogous of multiple local times, and get
for their integrals with respect to $\mu$ a formula analogous to the one
obtained in the discrete case.\bigskip

Let $G$ denote the operator on $L^{2}(D,dx)$ defined by $G$. Let $f$ be a
non-negative continuous function with compact support in $D$.

Note that $\left\langle \widehat{l},f\right\rangle $ is $\mu$-integrable only
in dimension one as then, $G$\ is locally trace class. In that case, using for
all $x$ an approximation of the Dirac measure at $x$, local times $\widehat
{l}^{x}$ can be defined in such a way that $\left\langle \widehat
{l},f\right\rangle =\int\widehat{l}^{x}f(x)dx$.

$\left\langle \widehat{l},f\right\rangle $ is $\mu$-square integrable in
dimensions one, two and three, as $G$\ is Hilbert-Schmidt if $D$ is bounded,
since $\int\int_{D\times D}G(x,y)^{2}dxdy<\infty$, and otherwise locally Hilbert-Schmidt.

\textbf{N.B.:} Considering distributions $\chi$ such that $\int\int
(G(x,y)^{2}\chi(dx)\chi(dy)<\infty$, we could see that $\left\langle
\widehat{l},\chi\right\rangle $ can be defined by approximation as a square
integrable variable and $\mu\left(  \left\langle \widehat{l},\chi\right\rangle
^{2}\right)  =\int(G(x,y)^{2}\chi(dx)\chi(dy)$.

\bigskip

Let $z$ be a complex number such that $\operatorname{Re}(z)>0$.

Note that$\ e^{-z\left\langle \widehat{l},f\right\rangle }+z\left\langle
\widehat{l},f\right\rangle -1$ is bounded by $\frac{\left|  z\right|  ^{2}}%
{2}\left\langle \widehat{l},f\right\rangle ^{2}$ and expands as an alternating
series $\sum_{2}^{\infty}\frac{z^{n}}{n!}\left(  -\left\langle \widehat
{l},f\right\rangle \right)  ^{n}$, with $\left|  e^{-z\left\langle \widehat
{l},f\right\rangle }-1-\sum_{1}^{N}\frac{z^{n}}{n!}\left(  -\left\langle
\widehat{l},f\right\rangle \right)  ^{n}\right|  \leq\frac{\left|
z\left\langle \widehat{l},f\right\rangle \right|  ^{N+1}}{(N+1)!}.$ Then, for
$\left|  z\right|  $ small enough., it follows from the above lemma that
\[
\mu\left(  e^{-z\left\langle \widehat{l},f\right\rangle }+z\left\langle
\widehat{l},f\right\rangle -1\right)  =\sum_{2}^{\infty}\frac{z^{n}}%
{n}Tr(-(M_{\sqrt{f}}GM_{\sqrt{f}})^{n}).
\]
As $M_{\sqrt{f}}GM_{\sqrt{f}}$ is Hilbert-Schmidt the renormalized determinant
$\det_{2}(I+zM_{\sqrt{f}}GM_{\sqrt{f}})$ is well defined and the second member
writes -$\log(\det_{2}(I+zM_{\sqrt{f}}GM_{\sqrt{f}}))$.\newline Then the
identity%
\[
\mu(e^{-z\left\langle \widehat{l},f\right\rangle }+z\left\langle \widehat
{l},f\right\rangle -1)=-\log(\det{}_{2}(I+zM_{\sqrt{f}}GM_{\sqrt{f}})).
\]
extends, as both sides are analytic as locally uniform limits of analytic
functions, to all complex values with positive real part.

\bigskip

\index{renormalized occupation field}The renormalized occupation field $\ \widetilde{\mathcal{L}_{\alpha}}$ is
defined as the compensated sum of all $\widehat{l}$ in $\mathcal{L}_{\alpha}$
(formally, $\ \widetilde{\mathcal{L}_{\alpha}}=\widehat{\mathcal{L}_{\alpha}%
}-\int\int_{0}^{T(l)}\delta_{l_{s}}ds\mu(dl)).\ $More precisely, we apply a
standard argument used for the construction of Levy processes, setting:%
\[
\left\langle \widetilde{\mathcal{L}_{\alpha}},f\right\rangle =\lim
_{\varepsilon\rightarrow0}\left\langle \widetilde{\mathcal{L}_{\alpha
,\varepsilon}},f\right\rangle
\]
with by definition%

\[
\left\langle \widetilde{\mathcal{L}_{\alpha,\varepsilon}},f\right\rangle
=\sum_{\gamma\in\mathcal{L}_{\alpha}}\left(  1_{\{T>\varepsilon\}}\int_{0}%
^{T}f(\gamma_{s})ds-\alpha\mu(1_{\{T>\varepsilon\}}\int_{0}^{T}f(\gamma
_{s})ds)\right)  .
\]

The convergence holds a.s. and in $L^{2}$, as
\begin{multline*}
\mathbb{E((}\sum_{\gamma\in\mathcal{L}_{\alpha}}(1_{\{\varepsilon^{\prime
}>T>\varepsilon\}}\int_{0}^{T}f(\gamma_{s})ds)-\alpha\mu(1_{\{\varepsilon
^{\prime}>T>\varepsilon\}}\int_{0}^{T}f(\gamma_{s})ds))^{2})\\
=\alpha\int(1_{\{\varepsilon^{\prime}>T>\varepsilon\}}\int_{0}^{T}f(\gamma
_{s})ds)^{2}\mu(dl)
\end{multline*}

and $\mathbb{E}(\left\langle \widetilde{\mathcal{L}_{\alpha}},f\right\rangle
^{2})=Tr((M_{\sqrt{f}}GM_{\sqrt{f}})^{2})$. Note that if we fix $f$, $\alpha$
can be considered as a time parameter and $\left\langle \widetilde
{\mathcal{L}_{\alpha,\varepsilon}},f\right\rangle $ are Levy processes with
discrete positive jumps approximating a Levy process with positive jumps
$\left\langle \widetilde{\mathcal{L}_{\alpha}},f\right\rangle $. The Levy
exponent $\mu(1_{\{T>\varepsilon\}}(e^{-\left\langle \widehat{l}%
,f\right\rangle }+\left\langle \widehat{l},f\right\rangle -1))$ of
$\left\langle \widetilde{\mathcal{L}_{\alpha,\varepsilon}},f\right\rangle )$
converges towards the L\'{e}vy exponent of $\left\langle \widetilde
{\mathcal{L}_{\alpha}},f\right\rangle )$ which is $\mu((e^{-\left\langle
\widehat{l},f\right\rangle }+\left\langle \widehat{l},f\right\rangle -1))$
and, from the identity $E(e^{-\left\langle \widetilde{\mathcal{L}_{\alpha}%
},f\right\rangle })=e^{-\alpha\mu(e^{-\left\langle \widehat{l},f\right\rangle
}+\left\langle \widehat{l},f\right\rangle -1)}$, we get the

\begin{theorem}
\label{detdeux}Assume $d\leq3$. Denoting $\widetilde{\mathcal{L}_{\alpha}}$ the compensated
sum of all $\widehat{l}$ in $\mathcal{L}_{\alpha}$, we have
\[
\mathbb{E}(e^{-\left\langle \widetilde{\mathcal{L}_{\alpha}},f\right\rangle
})=\det\text{\/}_{2}(I+M_{\sqrt{f}}GM_{\sqrt{f}}))^{-\alpha}.
\]
\end{theorem}

Moreover $e^{-\left\langle \widetilde{\mathcal{L}_{\alpha,\varepsilon}%
},f\right\rangle }$ converges a.s. and in $L^{1}$\ towards $e^{-\left\langle
\widetilde{\mathcal{L}_{\alpha}},f\right\rangle }$.

Considering distributions of finite $G^{2}$-energy $\chi$ (i.e. such that
${\int(G(x,y)^{2}\chi(dx)\chi(dy)<\infty}$), we can see that $\left\langle
\widetilde{\mathcal{L}_{\alpha}},\chi\right\rangle $ can be defined by
approximation as $\lim_{\lambda\rightarrow\infty}(\left\langle \widetilde
{\mathcal{L}_{\alpha}},\lambda G_{\lambda}\chi\right\rangle )$ and
\[
\mathbb{E}(\left\langle \widetilde{\mathcal{L}_{\alpha}},\chi\right\rangle
^{2})=\alpha\int(G(x,y))^{2}\chi(dx)\chi(dy).
\]
Specializing to $\alpha=\frac{k}{2}$, $k$ being any positive integer we have:

\begin{corollary}
The renormalized occupation field $\widetilde{\mathcal{L}_{\frac{k}{2}}}$ and
the Wick square $\frac{1}{2}:\sum_{1}^{k}\phi_{l}^{2}:$ have the same distribution.
\end{corollary}

\bigskip

If $\Theta$ is a conformal map from $D$ onto $\Theta(D)$, it follows from the
conformal invariance of the Brownian trajectories that a similar property
holds for the Brownian''loop soup''(Cf \cite{LW}). More precisely, if
$c(x)=Jacobian_{x}(\Theta)$ and, given a loop $l$, if $T^{c}(l)$ denotes the
reparametrized loop $l_{\tau_{s}}$, with $\int_{0}^{\tau_{s}}c(l_{u})du=s$,
the configuration $\Theta T^{c}(\mathcal{L}_{\alpha})$ is a Brownian loop soup
of intensity parameter $\alpha$ on $\Theta(D)$. Then we have the following:

\begin{proposition}
$\Theta(c\widetilde{\mathcal{L}_{\alpha}})$ is the renormalized occupation
field on $\Theta(D)$.
\end{proposition}

\begin{proof}
We have to show that the compensated sum is the same if we perform it after or
before the time change. For this it is enough to check that
\begin{align*}
&  \mathbb{E}([\sum_{\gamma\in\mathcal{L}_{\alpha}}(1_{\{\tau_{T}>\eta
\}}1_{\{T\leq\varepsilon\}}\int_{0}^{T}f(\gamma_{s})ds-\alpha\int
(1_{\{\tau_{T}>\eta\}}1_{\{T\leq\varepsilon\}}\int_{0}^{T}f(\gamma_{s}%
)ds)\mu(d\gamma)]^{2})\\
&  =\alpha\int(1_{\{\tau_{T}>\eta\}}1_{\{T\leq\varepsilon\}}\int_{0}%
^{T}f(\gamma_{s})ds)^{2}\mu(d\gamma)
\end{align*}
and%
\begin{align*}
&  \mathbb{E}([\sum_{\gamma\in\mathcal{L}_{\alpha}}(1_{\{T>\varepsilon
\}}1_{\tau_{T}\leq\eta}\int_{0}^{T}f(\gamma_{s})ds-\alpha\int
(1_{\{T>\varepsilon\}}1_{\tau_{T}\leq\eta}\int_{0}^{T}f(\gamma_{s}%
)ds)\mu(d\gamma)]^{2})\\
&  \alpha\int(1_{\{T>\varepsilon\}}1_{\tau_{T}\leq\eta}\int_{0}^{T}%
f(\gamma_{s})ds)^{2}\mu(d\gamma)\
\end{align*}
converge to zero as $\varepsilon$ and $\eta$ go to zero. It follows from the
fact that:
\[
\int[1_{\{T\leq\varepsilon\}}\int_{0}^{T}f(\gamma_{s})ds]^{2}\mu(d\gamma)
\]
and
\[
\int[1_{\tau_{T}\leq\eta}\int_{0}^{T}f(\gamma_{s})ds]^{2}\mu(d\gamma)
\]
converge to $0$. The second follows easily from the first\ if $c$ is bounded
away from zero. We can always consider the ''loop soups'' in an increasing
sequence of relatively compact open subsets of $D$ to reduce the general case
to that situation.
\end{proof}

As in the discrete case (see corollary \ref{mom}), we can compute product
expectations. In dimension $\leq3$, for $f_{j}$ continuous functions with
compact support in $D$:
\begin{equation}
\mathbb{E(}\left\langle \widetilde{\mathcal{L}_{\alpha}},f_{1}\right\rangle
...\left\langle \widetilde{\mathcal{L}_{\alpha}},f_{k}\right\rangle )=\int
Per_{\alpha}^{0}(G(x_{l},x_{m}),1\leq l,m\leq k)\prod f_{j}(x_{j})dx_{j}.
\label{prodl2}%
\end{equation}

\section{Renormalized powers}

In dimension one, as in the discrete case, powers of the occupation field can
be viewed as integrated self intersection local times. In dimension two,
renormalized powers of the occupation field, also called \textsl{renormalized
self intersections local times }can be defined, using renormalization
polynomials derived from the polynomials $Q_{k}^{\alpha,\sigma}$\ defined in
section \ref{momts}. The polynomials $Q_{k}^{\alpha,\sigma}$ cannot be used
directly as pointed out to me by Jay Rosen. See Dynkin \cite{Dynkin1},
\cite{Dynkin2}, \cite{Legall}, \cite{MarcusRosen} for such
definitions and proofs of convergence in the case of paths.

Assume $d=2$. Let $\pi_{\varepsilon}^{x}(dy)$ be the normalized arclength on
the circle of radius $\varepsilon$ around $x$, and set $\sigma_{\varepsilon
}^{x}=\int G(y,z)\pi_{\varepsilon}^{x}(dy)\pi_{\varepsilon}^{x}(dz)$.

As the distance between $x$ and $y$ tends to $0$, $G(x,y)$ is equivalent to
$G_{0}(x,y)=\frac{1}{\pi}\log(\left\|  x-y\right\|  )$ and moreover,
$G(x,y)=G_{0}(x,y)-H^{D^{c}}(x,dz)G_{0}(z,y)$, $H^{D^{c}}$ denoting the
Poisson kernel on the boundary of $D$.

Let $G_{x,x}^{(\varepsilon)}$ (respectively $G_{y,y}^{(\varepsilon^{\prime})}%
$, $G_{x,y}^{(\varepsilon,\varepsilon^{\prime})}$, $G_{y,x}^{(\varepsilon
^{\prime},\varepsilon)}$) denote the operator from $L^{2}(\pi_{\varepsilon
}^{x})$ into $L^{2}(\pi_{\varepsilon}^{x})$ (respectively $L^{2}%
(\pi_{\varepsilon^{\prime}}^{y})$ into $L^{2}(\pi_{\varepsilon^{\prime}}^{y}%
)$, $L^{2}(\pi_{\varepsilon}^{x})$ into $L^{2}(\pi_{\varepsilon^{\prime}}%
^{y})$, $L^{2}(\pi_{\varepsilon^{\prime}}^{y})$ into $L^{2}(\pi_{\varepsilon
}^{x})$) induced by the restriction of the Green functions to the the circle
pairs. Let $\iota_{x,y}^{(\varepsilon,\varepsilon^{\prime})}$ be the isometry
$L^{2}(\pi_{\varepsilon}^{x})$ into $L^{2}(\pi_{\varepsilon^{\prime}}^{y})$
induced by the natural map between the circles.

$G_{x,x}^{(\varepsilon)}$ and $G_{y,y}^{(\varepsilon)}$ are clearly Hilbert
Schmidt operators, while the products $G_{x,y}^{(\varepsilon,\varepsilon
^{\prime})}\iota_{y,x}^{(\varepsilon^{\prime},\varepsilon)}$ and
$G_{y,x}^{(\varepsilon^{\prime}\varepsilon)}\iota_{x,y}^{(\varepsilon
,\varepsilon^{\prime})}$ are trace-class.

We define the renormalization polynomials via the following generating
function:
\[
\mathbf{q}_{x,\varepsilon,\alpha}(t,u)=e^{\frac{tu}{1+t\sigma_{\varepsilon
}^{x}}}\det\text{\/}_{2}(I-\frac{t}{1+t\sigma_{\varepsilon}^{x}}%
G_{x,x}^{(\varepsilon)})^{\alpha}%
\]

This generating function is new to our knowledge but one should note that the
generating functions of the polynomials $Q_{k}^{\alpha,\sigma}$ can be written
$e^{\frac{tu}{1+t\sigma}}(1-\frac{t\sigma}{1+t\sigma})^{\alpha}e^{\alpha
\frac{t\sigma}{1+t\sigma}}$ and therefore has the same form$.$

Define the renormalisation polynomials $Q_{k}^{x,\varepsilon,\alpha}$ by:%
\[
\sum t^{k}Q_{k}^{x,\varepsilon,\alpha}(u)=\mathbf{q}_{x,\varepsilon,\alpha
}(t,u)
\]

The coefficients of $Q_{k}^{x,\varepsilon,\alpha}$ involve products of terms
of the form

$Tr([G_{x,x}^{(\varepsilon)}]^{m})=\int G(y_{1},y_{2})G(y_{2},y_{3}%
)...G(y_{m},y_{1})\prod_{1}^{m}\pi_{\varepsilon}^{x}(dy_{i})$ which are
different from $(\sigma_{\varepsilon}^{x})^{m}$ (but both are equivalent to
$\left[  \frac{-\log(\varepsilon)}{\pi}\right]  ^{m}$ as $\varepsilon
\rightarrow0$).

We have the following

\begin{theorem}
For any bounded continuous function $f$ with compact support,

$\int f(x)Q_{k}^{x,\varepsilon,\alpha}(\left\langle \widetilde{\mathcal{L}%
_{\alpha}},\pi_{\varepsilon}^{x}\right\rangle )dx$ converges in $L^{2}$
towards a limit denoted $\left\langle \widetilde{\mathcal{L}_{\alpha}^{k}%
},f\right\rangle $ and
\[
\mathbb{E}(\left\langle \widetilde{\mathcal{L}_{\alpha}^{k}},f\right\rangle
\left\langle \widetilde{\mathcal{L}_{\alpha}^{l}},h\right\rangle
)=\delta_{l,k}\frac{\alpha(\alpha+1)...(\alpha+k-1)}{k!}\int G^{2k}%
(x,y)f(x)h(y)dxdy.
\]
\end{theorem}

\begin{proof}
The idea of the proof can be understood by trying to prove that
\[
\mathbb{E((}\int f(x)Q_{k}^{x,\varepsilon,\alpha}(\left\langle \widetilde
{\mathcal{L}_{\alpha}},\pi_{\varepsilon}^{x}\right\rangle )dx)^{2})
\]
remains bounded as $\varepsilon$ decreases to zero. One should expand this
expression in terms of sums of integrals of product of Green functions and
check that cancellations analogous to the the combinatorial identities
(\ref{null})\ imply the cancelation of the logarithmic divergences.

These cancellations become apparent if we compute%

\[
(1)=\mathbb{E(}\mathbf{q}_{x,\varepsilon,\alpha}(t,\left\langle \widetilde
{\mathcal{L}_{\alpha}},\pi_{\varepsilon}^{x}\right\rangle )\mathbf{q}%
_{y,\varepsilon^{\prime},\alpha}(s,\left\langle \widetilde{\mathcal{L}%
_{\alpha}},\pi_{\varepsilon^{\prime}}^{y}\right\rangle ))
\]

which is well defined for $s$ and $t$ small enough. As the measures
$\pi_{\varepsilon}^{x}$ and $\pi_{\varepsilon^{\prime}}^{y}$ are mutually
singular $L^{2}(\pi_{\varepsilon}^{x}+\pi_{\varepsilon^{\prime}}^{y})$ is the
direct sum of $L^{2}(\pi_{\varepsilon}^{x})$\ and $L^{2}(\pi_{\varepsilon
^{\prime}}^{y})$, and any operator on $L^{2}(\pi_{\varepsilon}^{x}%
+\pi_{\varepsilon^{\prime}}^{y})$ can be written as a matrix $\left(
\begin{array}
[c]{cc}%
A & B\\
C & D
\end{array}
\right)  $ where $A$ (respectively $D,B,C)$ is an operator from $L^{2}%
(\pi_{\varepsilon}^{x})$ into $L^{2}(\pi_{\varepsilon}^{x})$ (respectively
$L^{2}(\pi_{\varepsilon^{\prime}}^{y})$ into $L^{2}(\pi_{\varepsilon^{\prime}%
}^{y})$, $L^{2}(\pi_{\varepsilon}^{x})$ into $L^{2}(\pi_{\varepsilon^{\prime}%
}^{y})$, $L^{2}(\pi_{\varepsilon^{\prime}}^{y})$ into $L^{2}(\pi_{\varepsilon
}^{x})$).

Theorem \ref{detdeux} can be proved in the same way for the Brownian motion
time changed by the inverse of the sum of the additive functionals defined by
$\pi_{\varepsilon}^{x}$ and $\pi_{\varepsilon^{\prime}}^{y}$ (its Green
function is the restriction of $G$ to the union of the two circles.
Alternatively, one can extend theorem \ref{detdeux} to measures to get the
same result). Applying this to the function equal to $t$ (respectively $s$) on
the circle of radius $\varepsilon$ around $x$ (respectively the circle of
radius $\varepsilon^{\prime}$ around $y$) yields%
\begin{align*}
(1) &  =\det\text{\/}_{2}(I-\frac{t}{1+t\sigma_{\varepsilon}^{x}}%
G_{x,x}^{(\varepsilon)})^{\alpha}\det\text{\/}_{2}(I-\frac{s}{1+s\sigma
_{\varepsilon^{\prime}}^{y}}G_{y,y}^{(\varepsilon^{\prime})})^{\alpha}.\\
&  \left[  \det\text{\/}_{2}\left(
\begin{array}
[c]{cc}%
I-\frac{t}{1+t\sigma_{\varepsilon}^{x}}G_{x,x}^{(\varepsilon)} &
-\frac{\sqrt{st}}{\sqrt{(1+t\sigma_{\varepsilon}^{x})(1+t\sigma_{\varepsilon
^{\prime}}^{y})}}G_{x,y}^{(\varepsilon,\varepsilon^{\prime})}\\
-\frac{\sqrt{st}}{\sqrt{(1+t\sigma_{\varepsilon}^{x})(1+t\sigma_{\varepsilon
^{\prime}}^{y})}}G_{y,x}^{(\varepsilon^{\prime},\varepsilon)} & I-\frac{s}%
{1+t\sigma_{\varepsilon^{\prime}}^{y}}G_{y,y}^{(\varepsilon^{\prime})}%
\end{array}
\right)  \right]  ^{-\alpha}%
\end{align*}%

\begin{align*}
&  =\left[  \det\text{\/}_{2}\left(
\begin{array}
[c]{cc}%
I-\frac{t}{1+t\sigma_{\varepsilon}^{x}}G_{x,x}^{(\varepsilon)} &
-\frac{\sqrt{st}}{\sqrt{(1+t\sigma_{\varepsilon}^{x})(1+s\sigma_{\varepsilon
^{\prime}}^{y})}}G_{x,y}^{(\varepsilon,\varepsilon^{\prime})}\\
-\frac{\sqrt{st}}{\sqrt{(1+t\sigma_{\varepsilon}^{x})(1+s\sigma_{\varepsilon
^{\prime}}^{y})}}G_{y,x}^{(\varepsilon^{\prime},\varepsilon)} & I-\frac{s}%
{1+s\sigma_{\varepsilon^{\prime}}^{y}}G_{y,y}^{(\varepsilon^{\prime})}%
\end{array}
\right)  \right]  ^{-\alpha}\\
&  .\left[  \det\text{\/}_{2}\left(
\begin{array}
[c]{cc}%
I-\frac{t}{1+t\sigma_{\varepsilon}^{x}}G_{x,x}^{(\varepsilon)} & 0\\
0 & I-\frac{s}{1+s\sigma_{\varepsilon^{\prime}}^{y}}G_{y,y}^{(\varepsilon
^{\prime})}%
\end{array}
\right)  \right]  ^{\alpha}%
\end{align*}

Note that if, $A$ and $B$ are Hilbert-Schmidt\ operators, $\det_{2}%
(I+A)\det_{2}(I+B)=e^{-Tr(AB)}\det_{2}((I+A)(I+B))$. It follows that if, $A$
and $B^{\prime}$ are Hilbert-Schmidt\ operators such that $B^{\prime\prime
}=(I+A)^{-1}(I+B^{\prime})-I$ is trace class with zero trace and
$AB^{\prime\prime}$ has also zero trace,%
\begin{align*}
\left[  \det\text{\/}_{2}(I+A)\right]  ^{-1}\det\text{\/}_{2}(I+B^{\prime}) &
=\det\text{\/}_{2}(I+B^{\prime\prime})=\det(I+B^{\prime\prime})\\
&  =\det((I+A)^{-1}(I+B^{\prime}))=\det((I+A)^{-\frac{1}{2}}(I+B^{\prime
})(I+A)^{-\frac{1}{2}}).
\end{align*}
Taking now $A=\left(
\begin{array}
[c]{cc}%
-\frac{t}{1+t\sigma_{\varepsilon}^{x}}G_{x,x}^{(\varepsilon)} & 0\\
0 & -\frac{s}{1+t\sigma_{\varepsilon^{\prime}}^{y}}G_{y,y}^{(\varepsilon
^{\prime})}%
\end{array}
\right)  $

and $B^{\prime}=\left(
\begin{array}
[c]{cc}%
-\frac{t}{1+t\sigma_{\varepsilon}^{x}}G_{x,x}^{(\varepsilon)} & -\frac{\sqrt
{st}}{\sqrt{(1+t\sigma_{\varepsilon}^{x})(1+t\sigma_{\varepsilon^{\prime}}%
^{y})}}G_{x,y}^{(\varepsilon,\varepsilon^{\prime})}\\
-\frac{\sqrt{st}}{\sqrt{(1+t\sigma_{\varepsilon}^{x})(1+t\sigma_{\varepsilon
^{\prime}}^{y})}}G_{y,x}^{(\varepsilon^{\prime},\varepsilon)} & -\frac{s}%
{1+t\sigma_{\varepsilon^{\prime}}^{y}}G_{y,y}^{(\varepsilon^{\prime})}%
\end{array}
\right)  $ we obtain easily that $A-B^{\prime}$ is trace class

as $Tr\left(  \left|  \left(
\begin{array}
[c]{cc}%
0 & -G_{x,y}^{(\varepsilon,\varepsilon^{\prime})}\\
-G_{y,x}^{(\varepsilon^{\prime},\varepsilon)} & 0
\end{array}
\right)  \right|  \right)  =2Tr(\left|  G_{x,y}^{(\varepsilon,\varepsilon
^{\prime})}\iota_{y,x}^{(\varepsilon^{\prime},\varepsilon)}\right|  )$.
Therefore, $B^{\prime\prime}$ and $AB^{\prime\prime}$ are trace class and it
is clear they have zero trace, as both are of the form $\left(
\begin{array}
[c]{cc}%
0 & S\\
R & 0
\end{array}
\right)  $.

Therefore,setting $V=\sqrt{st}(I+t\sigma_{\varepsilon}^{x}-tG_{x,x}%
^{(\varepsilon)})^{-\frac{1}{2}}G_{x,y}^{(\varepsilon,\varepsilon^{\prime}%
)}(I+s\sigma_{\varepsilon^{\prime}}^{y}-sG_{y,y}^{(\varepsilon^{\prime}%
)})^{-\frac{1}{2}}$,%

\[
(1)=\det\left(
\begin{array}
[c]{cc}%
I & -V\\
-V^{\ast} & I
\end{array}
\right)  ^{-\alpha}.
\]

Hence,%

\begin{align*}
(1)  &  =\det\left(
\begin{array}
[c]{cc}%
I & -V\\
0 & I-V^{\ast}V
\end{array}
\right)  ^{-\alpha}=\det(I-V^{\ast}V)^{-\alpha}\\
&  =\det(I-st(I+s\sigma_{\varepsilon^{\prime}}^{y}-sG_{y,y}^{(\varepsilon
^{\prime})})^{-1}G_{y,x}^{(\varepsilon^{\prime},\varepsilon)}(I+t\sigma
_{\varepsilon}^{x}-tG_{x,x}^{(\varepsilon)})^{-1}G_{x,y}^{(\varepsilon
,\varepsilon^{\prime})})^{-\alpha}.
\end{align*}

This quantity can be expanded. Setting, for any trace class kernel
$K(z,z^{\prime})$ acting on $L^{2}(\pi_{\varepsilon}^{y})$,%
\[
Per_{\alpha}(K^{(n)})=\int Per_{\alpha}(K(z_{i},z_{j}),1\leq i,j\leq
n)\prod_{1}^{n}\pi_{\varepsilon}^{y}(dz_{i})
\]

it equals:%
\[
1+\sum_{1}^{\infty}\frac{1}{k!}Per_{\alpha}(([stG_{y,x}^{(\varepsilon^{\prime
},\varepsilon)}(I+t\sigma_{\varepsilon}^{x}-tG_{x,x}^{(\varepsilon)}%
)^{-1}G_{x,y}^{(\varepsilon,\varepsilon^{\prime})}(I+s\sigma_{\varepsilon
^{\prime}}^{y}-sG_{y,y}^{(\varepsilon^{\prime})})^{-1}]^{(k)})=(2)
\]

Identifying the coeficients of $t^{k}s^{l}$ in $(1)$ and $(2)$ yields the
identity%
\[
\mathbb{E(}Q_{k}^{x,\varepsilon,\alpha}(\left\langle \widetilde{\mathcal{L}%
_{\alpha}},\pi_{\varepsilon}^{x}\right\rangle )Q_{l}^{x,\varepsilon^{\prime
},\alpha}(\left\langle \widetilde{\mathcal{L}_{\alpha}},\pi_{\varepsilon
^{\prime}}^{y}\right\rangle ))=\delta_{l,k}\frac{1}{k!}Per_{\alpha}%
([G_{y,x}^{(\varepsilon^{\prime},\varepsilon)}G_{x,y}^{(\varepsilon
,\varepsilon^{\prime})}]^{(k)})G^{2k}(x,y)+R_{k,l}%
\]

where $R_{k,l}$ is the (finite) sum of the $t^{k}s^{l}$\ coefficients
appearing in

$\sum_{1}^{\sup(k,l)}\frac{1}{k!}Per_{\alpha}(([stG_{y,x}^{(\varepsilon
^{\prime},\varepsilon)}(I+t\sigma_{\varepsilon}^{x}I-tG_{x,x}^{(\varepsilon
)})^{-1}G_{x,y}^{(\varepsilon,\varepsilon^{\prime})}(I+s\sigma_{\varepsilon
^{\prime}}^{y}I-sG_{y,y}^{(\varepsilon^{\prime})})^{-1}]^{(k)})$ (except of
course, for $k=l$, the term $\frac{1}{k!}Per_{\alpha}([G_{y,x}^{(\varepsilon
^{\prime},\varepsilon)}G_{x,y}^{(\varepsilon,\varepsilon^{\prime})}%
]^{(k)})G^{2k}(x,y)$)

The remarkable fact is that the coefficients of $Q_{k}^{x,\varepsilon,\alpha}$
are such that this expression involves no term of the form $Tr([G_{x,x}%
^{(\varepsilon)}]^{m})$ or $Tr([G_{y,y}^{(\varepsilon^{\prime})}]^{m})$.
Decomposing the permutations which appear in the expression of the $\alpha
$-permanent into cycles, we see all the terms are products of traces of
operators of the form $\int G(y_{1},y_{2})...G(y_{n},y_{1})\pi_{\varepsilon
_{1}}^{x_{1}}(dy_{1})...\pi_{\varepsilon_{n}}^{x_{n}}(dy_{n})$ in which at
least two $x_{j}$'s are distinct. It is also clear from the expression $(2)$
above\ that if we replace $G_{x,x}^{(\varepsilon)}$ and $G_{y,y}%
^{(\varepsilon^{\prime})}$ by $\sigma_{\varepsilon}^{x}I$ and $\sigma
_{\varepsilon^{\prime}}^{y}I$, the expansion becomes very simple and all terms
vanish except for $l=k$, the term $\frac{1}{k!}Per_{\alpha}([G_{y,x}%
^{(\varepsilon^{\prime},\varepsilon)}G_{x,y}^{(\varepsilon,\varepsilon
^{\prime})}]^{(k)})$ which will be proved to converge towards $\frac{\alpha
(\alpha+1)...(\alpha+k-1)}{k!}G^{2k}(x,y)=\frac{G^{2k}(x,y)}{k!}\sum_{1}%
^{k}d(k,l)\alpha^{l}$ (see remark \ref{stirl} on Stirling numbers).

To prove this convergence, and also that $R_{k,l}\rightarrow0$ as
$\varepsilon,\varepsilon^{\prime}\rightarrow0$, it is therefore enough to
prove the following:
\end{proof}

\begin{lemma}
Consider for any $x_{1},x_{2},...,x_{n}$, $\varepsilon$ small enough and
$\varepsilon\leq\varepsilon_{1},...,\varepsilon_{n}\leq2\varepsilon$, with
$\varepsilon_{i}=\varepsilon_{j}$ if $x_{i}=x_{j}$, an expression of the
form:
\[
\Delta=\left|  \prod_{i,x_{i-1}\neq x_{i}}G(x_{i-1},x_{i})(\sigma
_{\varepsilon_{i}}^{x_{i}})^{m_{i}}-\int G(y_{1},y_{2})...G(y_{n},y_{1}%
)\pi_{\varepsilon_{1}}^{x_{1}}(dy_{1})...\pi_{\varepsilon_{n}}^{x_{n}}%
(dy_{n})\right|
\]
in which we define $m_{i}$ as $\sup(h,\ x_{i+h}=x_{i})$ and in which at least
two $x_{j}$'s are distinct.Then for some positive integer $N$, and $C>0$, on
$\cap\{\left\|  x_{i-1}-x_{i}\right\|  \geq\sqrt{\varepsilon}\}$
\[
\Delta\leq C\sqrt{\varepsilon}\log(\varepsilon)^{N}%
\]
\newline 
\end{lemma}

\begin{proof}

In the integral term, we first replace progressively $G(y_{i-1},y_{i})$ by
$G(x_{i-1},x_{i})$ whenever $x_{i-1}\neq x_{i}$, using triangle, then Schwartz
inequalities, to get an upper bound of the absolute value of the difference
made by this substitution in terms of a sum $\Delta^{\prime}$\ of expressions
of the form
\[
\prod_{l}G(x_{l},x_{l+1})\sqrt{\int(G(y_{1},y_{2})-G(x_{1},x_{2}))^{2}%
\pi_{\varepsilon_{1}}^{x_{1}}(dy_{1})\pi_{\varepsilon_{2}}^{x_{2}}(dy_{2}%
)\int\prod G^{2}(y_{k},y_{k+1})\prod\pi_{\varepsilon_{k}}^{x_{k}}(dy_{k})}.
\]
The expression obtained after these substitutions can be written
\[
W=\prod_{i,x_{i-1}\neq x_{i}}G(x_{i-1},x_{i})\int G(y_{1},y_{2}%
)...G(y_{m_{i-1}},y_{m_{i}})\pi_{\varepsilon_{i}}^{x_{i}}(dy_{1}%
)...\pi_{\varepsilon_{i}}^{x_{i}}(dy_{m_{i}})
\]
and we see the integral terms could be replaced by $(\sigma_{\varepsilon
}^{x_{i}})^{m_{i}}$ if $G$ was translation invariant. But as the distance
between $x$ and $y$ tends to $0$, $G(x,y)$ is equivalent to $G_{0}%
(x,y)=\frac{1}{\pi}\log(\left\|  x-y\right\|  )$ and moreover, $G(x,y)=G_{0}%
(x,y)-H^{D^{c}}(x,dz)G_{0}(z,y)$. As our points lie in a compact inside $D$,
it follows that for some constant $C$, for $\left\|  y_{1}-x\right\|
\leq\varepsilon$, $\left|  \int(G(y_{1},y_{2})\pi_{\varepsilon}^{x}%
(dy_{2})-\sigma_{\varepsilon}^{x}\right|  <C\varepsilon$. \newline Hence, the
difference $\Delta^{\prime\prime}$ between $W$ and $\prod_{i,x_{i-1}\neq
x_{i}}G(x_{i-1},x_{i})(\sigma_{\varepsilon}^{x_{i}})^{m_{i}}$ can be bounded
by $\varepsilon W^{\prime}$, where $W^{\prime}$ is an expression similar to
$W$.

To get a good upper bound on $\Delta$, using the previous observations, by
repeated applications of H\"{o}lder inequality. it is enough to show that for
$\varepsilon$ small enough and $\varepsilon\leq\varepsilon_{1},\varepsilon
_{2}\leq2\varepsilon$, (with $C$ and $C^{\prime}$\ denoting various constants):

\begin{enumerate}
\item[1)] $\int(G(y_{1},y_{2})-G(x_{1},x_{2})^{2}\pi_{\varepsilon_{1}}^{x_{1}%
}(dy_{1})\pi_{\varepsilon_{2}}^{x_{2}}(dy_{2})$\newline $<C(\varepsilon
1_{\{\left\|  x_{1}-x_{2}\right\|  \geq\sqrt{\varepsilon}\}}+(G(x_{1}%
,x_{2})^{2}+\log(\varepsilon)^{2})1_{\{\left\|  x_{1}-x_{2}\right\|
<\sqrt{\varepsilon}\}})\label{loli}$,

\item[2)] $\int G(y_{1},y_{2})^{k}\pi_{\varepsilon}^{x}(dy_{1})\pi
_{\varepsilon}^{x}(dy_{2})<C\left|  \log(\varepsilon)\right|  ^{k}$ and more generally

\item[3)] $\int G(y_{1},y_{2})^{k}\pi_{\varepsilon_{1}}^{x_{1}}(dy_{1}%
)\pi_{\varepsilon_{2}}^{x_{2}}(dy_{2})<C\left|  \log(\varepsilon)\right|  ^{k}.$
\end{enumerate}

As the main contributions come from the singularities of $G$, they follow from
the following simple inequalities:

\begin{enumerate}
\item[1')]
\begin{multline*}
\int\left|  \log(\varepsilon^{2}+2R\varepsilon\cos(\theta)+R^{2}%
)-\log(R)\right|  ^{2}d\theta\\
=\int\left|  \log((\varepsilon/R)^{2}+2(\varepsilon/R)\cos(\theta)+1)\right|
^{2}d\theta<C((\varepsilon1_{\{R\geq\sqrt{\varepsilon}\}\}}+\log
^{2}(R/\varepsilon)1_{\{R<\sqrt{\varepsilon}\}\}})
\end{multline*}
(considering separately the cases where $\frac{\sqrt{\varepsilon}}{R}$ is
large or small)

\item[2')] $\int\left|  \log(\varepsilon^{2}(2+2\cos(\theta)))\right|
^{k}d\theta\leq C\left|  \log(\varepsilon)\right|  ^{k}$

\item[3')] $\int\left|  \log((\varepsilon_{1}\cos(\theta_{1})+\varepsilon
_{2}\cos(\theta_{2})+r)^{2}+(\varepsilon_{1}\sin(\theta_{1})+\varepsilon
_{2}\sin(\theta_{2}))^{2}\right|  ^{k}d\theta_{1}d\theta_{2}\leq C(\left|
\log(\varepsilon)\right|  )^{k}$. It can be proved by observing that for
$r\leq\varepsilon_{1}+\varepsilon_{2}$, we have near the line of singularities
(i.e. the values $\theta_{1}(r)$ and $\theta_{2}(r)$ for which the expression
under the $\log$\ vanishes) to evaluate an integral which can be bounded
(after a change of variable) by an integral of the form $C\int_{0}^{1}%
(-\log(\varepsilon u))^{k}du\leq C^{\prime}(-\log(\varepsilon))^{k}$ for
$\varepsilon$ small enough.
\end{enumerate}

To finish the proof of the theorem, let us note that by the lemma above, and
the estimate \ref{loli}\ in its proof, for $\varepsilon\leq\varepsilon
_{1},\varepsilon_{2}\leq2\varepsilon$, we have, for some integer $N^{l,k}$%

\begin{multline}
\left|  \mathbb{E}\Big(  Q_{k}^{x,\varepsilon_{1,}\alpha}(\langle
\widetilde{\mathcal{L}_{\alpha}},\pi_{\varepsilon_{1}}^{x}\rangle
)Q_{l}^{y,\varepsilon_{2},\alpha}(\langle\widetilde{\mathcal{L}_{\alpha}}%
,\pi_{\varepsilon_{2}}^{y}\rangle)\Big)  -\delta_{l,k}G(x,y)^{2k}%
\frac{\alpha(\alpha+1)...(\alpha+k-1)}{k!}\right| \\
\leq C\log(\varepsilon)^{N_{l,k}}(\sqrt{\varepsilon}+G(x,y)^{l+k}1_{\{\left\|
x-y\right\|  <\sqrt{\varepsilon}}). \label{maj}%
\end{multline}

The bound (\ref{maj}) is uniform in $(x,y)$ only away from the diagonal as
$G(x,y)$ can be arbitrarily large but we conclude from it\ that for any
bounded integrable $f$ and $h$,
\begin{multline*}
\left|  \int(\mathbb{E}(Q_{k}^{x,\varepsilon_{1,}\alpha}(\langle
\widetilde{\mathcal{L}_{\alpha}},\pi_{\varepsilon_{1}}^{x}\rangle
)Q_{l}^{y,\varepsilon_{2},\alpha}(\langle\widetilde{\mathcal{L}_{\alpha}}%
,\pi_{\varepsilon_{2}}^{y}\rangle))-\delta_{l,k}G(x,y)^{2k}\frac{\alpha
...(\alpha+k-1)}{k!})f(x)h(y)dxdy\right| \\
\leq C^{\prime}\sqrt{\varepsilon}\log(\varepsilon)^{N_{l,k}}%
\end{multline*}
(as $\int\int G(x,y)^{2k}1_{\{\left\|  x-y\right\|  <\sqrt{\varepsilon}}dxdy$
can be bounded by $C\varepsilon^{\frac{2}{3}}$, for example).

Taking $\varepsilon_{n}=2^{-n}$, it is then straightforward to check that
$\int f(x)Q_{k}^{x,\varepsilon_{1,}\alpha}(\left\langle \widetilde
{\mathcal{L}_{\alpha}},\pi_{\varepsilon_{n}}^{x}\right\rangle )dx$\ is a
Cauchy sequence in $L^{2}$. The theorem follows.
\end{proof}

\bigskip

Specializing to $\alpha=\frac{k}{2}$, $k$ being any positive integer\ as
before, it follows that Wick powers of $\sum_{j=1}^{k}\phi_{j}^{2}$\ are
associated with self intersection local times of the loops. More precisely, we have:

\begin{proposition}
The renormalized self intersection local times $\widetilde{\mathcal{L}%
_{\frac{k}{2}}^{n}}$ and the Wick powers $\frac{1}{2^{n}n!}:(\sum_{1}^{k}%
\phi_{l}^{2})^{n}:$ have the same joint distribution.
\end{proposition}

\begin{proof}
The proof is just a calculation of the $L^{2}$-norm of
\[
\int[\frac{1}{2^{n}n!}:(\sum_{1}^{k}\phi_{l}^{2})^{n}:(x)-Q_{n}^{x,\varepsilon
,\frac{k}{2}}(\frac{1}{2}:\sum_{1}^{k}\phi_{l}^{2}:(\pi_{\varepsilon}%
^{x}))]f(x)dx
\]
which converges to zero with $\varepsilon$.

The expectation of the square of this difference is the sum of two square
expectations which both converge towards $\frac{k(k+2)...(k+2(n-1))}{2^{n}%
n!}\int G^{2n}(x,y)f(x)f(y)dxdy$ and a middle term which converges towards
twice the opposite value. The difficult term $\mathbb{E((}Q_{n}^{x,\varepsilon
,\frac{k}{2}}(:\sum_{1}^{k}\phi_{l}^{2}:(\pi_{\varepsilon}^{x}))]f(x)dx)^{2})$
is given by the previous theorem. The two others come from simple Gaussian
calculations (note that only highest degree term $\frac{u^{n}}{n!}$\ of the
polynomial $Q_{n}^{x,\varepsilon,\frac{k}{2}}(u)$\ contributes to the
expectation of the middle term) using identity \ref{freud}.
\end{proof}

In the following exercise, we study and compare the polynomials $Q_{N}%
^{\alpha,\sigma}$ and $Q_{N}^{x,\varepsilon,\alpha}$.

\begin{exercise}
Let $d_{n,k}^{0}$ be the number of $n$-permutations with no fixed points and
$k$ cycles. If $k_{j},\;2\leq j\leq n$\ are integers such that $\sum_{j}%
jk_{j}=m$ and $\sum_{j}k_{j}=k$, let $C_{m,k}(k_{j},\;2\leq j\leq n)$ be the
number of $m$-permutations with no fixed points and $k_{j}$ cycles of length
$j$. Note that $\sum C_{m,k}(k_{j},\;2\leq j\leq n)=d_{m,k}^{0}$. Show the
following identities:

a) $\sum_{0}^{n}%
\genfrac{(}{)}{}{}{n}{m}%
d_{m,k}^{0}=d(n,k)$ (the number of $n$-permutations with $k$ cycles).

b) $C_{m,k}(k_{j},\;2\leq j\leq n)=\frac{m!}{\prod k_{j}!j^{k_{j}}}$

c) $\sum t^{N}Q_{N}^{\alpha,\sigma}(u)=e^{\frac{t(u+\alpha\sigma)}{1+t\sigma}%
}(1-\frac{t\sigma}{1+t\sigma})^{\alpha}=e^{\frac{tu}{1+t\sigma}}(1+\sum
_{m=1}^{\infty}\sum_{1\leq k\leq m}d_{m,k}^{0}\frac{(-\alpha)^{k}}%
{m!}(\frac{t\sigma}{1+t\sigma})^{m})$

$=\sum_{l=0}^{\infty}\frac{t^{l}u^{l}}{l!}(1+t\sigma)^{-l}+\sum_{l=0}^{\infty
}\sum_{m=1}^{\infty}\frac{t^{l+m}u^{l}}{m!l!}(1+t\sigma)^{-m-l}\sum_{1\leq
k\leq m}d_{m,k}^{0}(-\alpha)^{k}$

d) $Q_{N}^{\alpha,\sigma}(u)=\sum_{0\leq l\leq N}\sum_{k\leq N-l}%
a_{N,l,k}u^{l}\sigma^{N-l}\alpha^{k}$ 

with $a_{N,l,k}=\sum_{m=k}%
^{N-l}(-1)^{N-l-k-m}\frac{(N-1)!}{l!m!(N-l-m)!(m+l-1)!}d_{m,k}^{0}$ for
$k\geq1$,

$a_{N,l,0}=(-1)^{N-l-k}\frac{(N-1)!}{l!(N-l)!(l-1)!}$ and $a_{N,0,0}=0$.

e) $\mathbf{q}_{x,\varepsilon,\alpha}(t,u)=\sum_{l=0}^{\infty}\frac{t^{l}%
u^{l}}{l!}(1+t\sigma_{x}^{\varepsilon})^{-l}+\sum_{l=0}^{\infty}\sum
_{m=1}^{\infty}\frac{u^{l}t^{l+m}}{l!m!}(1+t\sigma_{x}^{\varepsilon}%
)^{-(m+l)}Per_{-\alpha}(\left[  G_{x,x}^{\varepsilon}\right]  ^{(m)})$

f) $Q_{N}^{x,\varepsilon,\alpha}(u)=\sum_{0\leq l\leq N}\sum_{k\leq
N-l}A_{N,l,k}u^{l}(\sigma_{x}^{\varepsilon})^{N-l}\alpha^{k}$ 

with
$A_{N,l,k}=\sum_{m=k}^{N-l}(-1)^{N-l-k-m}\frac{(N-1)!}{l!m!(N-l-m)!(m+l-1)!}%
D_{m,k}^{0}$ and $D_{m,k}^{0}=\sum C_{m,k}(k_{j},\;2\leq j\leq n)\prod
(Tr([\frac{G_{x,x}^{(\varepsilon)}}{\sigma_{x}^{\varepsilon}}]^{j}))^{k_{j}}$,
for $k\geq1$,

$A_{N,l,0}=(-1)^{N-l-k}\frac{(N-1)!}{l!(N-l)!(l-1)!}$ and $A_{N,0,0}=0$.

In particular, $Q_{2}^{x,\varepsilon,\alpha}(u)=\frac{1}{2}(u^{2}-2\sigma
_{x}^{\varepsilon}u-\alpha Tr([\frac{G_{x,x}^{(\varepsilon)}}{\sigma
_{x}^{\varepsilon}}]^{2}))$ and $Q_{3}^{x,\varepsilon,\alpha}(u)=\frac{1}%
{6}(u^{3}-6\sigma_{x}^{\varepsilon}u^{2}+6u(\sigma_{x}^{\varepsilon}%
)^{2}-3\alpha uTr([G_{x,x}^{(\varepsilon)}]^{2})+6\alpha\sigma_{x}%
^{\varepsilon}Tr([G_{x,x}^{(\varepsilon)}]^{2})-2\alpha Tr([G_{x,x}%
^{(\varepsilon)}]^{3}))$
\end{exercise}

\begin{exercise}
Prove that $\lim_{\varepsilon\rightarrow0}\frac{D_{m,k}^{0}}{d_{m,k}^{0}}=1$
\end{exercise}

\subsubsection*{Final remarks:}

\begin{enumerate}
\item[a)] These generalized fields have two fundamental properties:

Firstly they are local fields (or more precisely local functionals of the
field $\widetilde{\mathcal{L}_{\alpha}}$ in the sense that their values on
functions supported in an open set $D$ depend only on the trace of the loops
on $D$.

Secondly, note we could have used\ a conformally covariant regularization to
define $\widetilde{\mathcal{L}_{\alpha}^{k}}$, (along the same lines but with
slightly different estimates), by taking $\pi_{\varepsilon}^{x}$ to be the
capacitary measure of the compact set $\{y,G^{x,y}\geq-\log\varepsilon\}$ and
$\sigma_{\varepsilon}^{x}$\ its capacity. Then it appears that the action of a
conformal transformation $\Theta$\ on these fields is given by \emph{the }%
$k$\emph{-th power of the conformal factor }$c=\text{Jacobian}(\Theta)$. More
precisely, $\Theta(c^{k}\widetilde{\mathcal{L}_{\alpha}^{k}})$ is the
renormalized $k$-th power of the occupation field in $\Theta(D)$.

\item[b)] It should be possible to derive from the above remark \ and from
hyperconrtactive type estimates the existence of exponential moments and
introduce non trivial local interactions as in the constructive field theory
derived from the free field (Cf \cite{Sim2}).

\item[c)] Let us also briefly consider currents. We will restrict our
attention to the one and two dimensional Brownian case, $X$ being an open
subset of the line or plane. Currents can be defined by vector fields, with
compact support.

Then, if we now denote by $\phi$ the complex valued free field (its real and
imaginary parts being two independent copies of the free field), $\int
_{l}\omega$ and $\int_{X}(\overline{\phi}\partial_{\omega}\phi-\phi
\partial_{\omega}\overline{\phi})dx$ are well defined square integrable
variables in dimension 1 (it can be checked easily by Fourier series). The
distribution of the centered occupation field of the loop process ''twisted''
by the complex exponential $\exp(\sum_{l\in\mathcal{L}_{\alpha}}\int
_{l}i\omega+\frac{1}{2}\widehat{l}(\left\|  \omega\right\|  ^{2}))$ appears to
be the same as the distribution of the field $:\phi\overline{\phi}%
:$\ ''twisted'' by the complex exponential $\exp(\int_{X}(\overline{\phi
}\partial_{\omega}\phi-\phi\partial_{\omega}\overline{\phi})dx)$ (Cf\cite{LJ2}).

In dimension 2, logarithmic divergences occur.

\item[d)] There is a lot of related investigations. The extension of the
properties proved here in the finite framework has still to be completed,
though the relation with spanning trees should follow from the remarkable
results obtained on SLE processes, especially \cite{LSW}. Note finally that
other essential relations between SLE processes, loops and free fields appear
in \cite{WW}, \cite{ScSh}\, \cite{Dub}, and more recently in \cite{WS1}\ and
\cite{WS2}.

\end{enumerate}

\end{document}